\documentclass[12pt,reqno]{amsart}
\usepackage{amssymb}
\usepackage{graphicx}
\usepackage{xcolor}
\usepackage{bm}

\usepackage[all]{xy}
\DeclareFontFamily{U}{mathb}{\hyphenchar\font45}
\DeclareFontShape{U}{mathb}{m}{n}{
      <5> <6> <7> <8> <9> <10> gen * mathb
      <10.95> mathb10 <12> <14.4> <17.28> <20.74> <24.88> mathb12
      }{}
\DeclareSymbolFont{mathb}{U}{mathb}{m}{n}
\DeclareMathSymbol{\righttoleftarrow}{3}{mathb}{"FD}
\DeclareFontFamily{U}{wncy}{}
\DeclareFontShape{U}{wncy}{m}{n}{<->wncyr10}{}
\DeclareSymbolFont{mcy}{U}{wncy}{m}{n}
\DeclareMathSymbol{\Sh}{\mathord}{mcy}{"58}

\theoremstyle{plain}
\newtheorem{prop}{Proposition}[section]
\newtheorem{theo}[prop]{Theorem}
\newtheorem{coro}[prop]{Corollary}
\newtheorem{lemm}[prop]{Lemma}
\theoremstyle{remark}
\newtheorem{rema}[prop]{Remark}

\theoremstyle{definition}
\newtheorem{defi}[prop]{Definition}

\newtheorem{exam}[prop]{Example}

\numberwithin{equation}{section}
\newcommand{\A}{{\mathbb A}}

\newcommand{\PP}{{\mathbb P}}
\newcommand{\bP}{{\mathbb P}}
\newcommand{\Q}{{\mathbb Q}}

\newcommand{\G}{{\mathbb G}}

\newcommand{\R}{{\mathbb R}}
\newcommand{\Z}{{\mathbb Z}}

\newcommand{\cD}{{\mathcal D}}
\newcommand{\cE}{{\mathcal E}}
\newcommand{\cF}{{\mathcal F}}
\newcommand{\cG}{{\mathcal G}}
\newcommand{\cO}{{\mathcal O}}
\newcommand{\cI}{{\mathcal I}}
\newcommand{\cJ}{{\mathcal J}}
\newcommand{\cK}{{\mathcal K}}
\newcommand{\cL}{{\mathcal L}}

\newcommand{\cN}{{\mathcal N}}
\newcommand{\cS}{{\mathcal S}}
\newcommand{\cT}{{\mathcal T}}
\newcommand{\cX}{{\mathcal X}}
\newcommand{\cY}{{\mathcal Y}}
\newcommand{\cZ}{{\mathcal Z}}
\newcommand{\cU}{{\mathcal U}}
\newcommand{\cV}{{\mathcal V}}
\newcommand{\cW}{{\mathcal W}}

\newcommand{\rH}{{\mathrm H}}
\newcommand{\rK}{{\mathrm K}}

\newcommand{\lra}{\longrightarrow}

\newcommand{\bQ}{{\mathbb Q}}

\newcommand{\bZ}{{\mathbb Z}}

\newcommand{\fS}{{\mathfrak S}}

\newcommand{\eqto}{\stackrel{\lower1.5pt\hbox{$\scriptstyle\sim\,$}}\to}
\newcommand{\eqdashto}{\stackrel{\lower1.5pt\hbox{$\scriptstyle\sim\,$}}\dashrightarrow}
\newcommand{\actsfromleft}{\mathrel{\reflectbox{$\righttoleftarrow$}}}
\newcommand{\actsfromright}{\righttoleftarrow}

\DeclareMathOperator{\Spec}{Spec}

\DeclareMathOperator{\Br}{Br}
\DeclareMathOperator{\Aut}{Aut}

\DeclareMathOperator{\Var}{Var}
\DeclareMathOperator{\Burn}{Burn}
\DeclareMathOperator{\cBurn}{\mathcal{B}\mathit{urn}}

\DeclareMathOperator{\HOM}{HOM}
\DeclareMathOperator{\MAP}{MAP}
\DeclareMathOperator{\BIR}{BIR}
\DeclareMathOperator{\BIRAUT}{BIR\,AUT}
\DeclareMathOperator{\Proj}{Proj}
\DeclareMathOperator{\Sym}{Sym}

\DeclareMathOperator{\colim}{colim}

\DeclareMathOperator{\Wedge}{{\textstyle\bigwedge}}

\newcommand{\oBurn}{{\overline\Burn}}
\newcommand{\ocBurn}{{\overline\cBurn}}

\DeclareFontFamily{U}{mathx}{\hyphenchar\font45}
\DeclareFontShape{U}{mathx}{m}{n}{
      <5> <6> <7> <8> <9> <10>
      <10.95> <12> <14.4> <17.28> <20.74> <24.88>
      mathx10
      }{}
\DeclareSymbolFont{mathx}{U}{mathx}{m}{n}
\DeclareMathSymbol{\bigtimes}{1}{mathx}{"91}

\begin{document}
\title[Birational geometry of stacks]{Birational geometry of Deligne-Mumford stacks}

\author{Andrew Kresch}
\address{
  Institut f\"ur Mathematik,
  Universit\"at Z\"urich,
  Winterthurerstrasse 190,
  CH-8057 Z\"urich, Switzerland
}
\email{andrew.kresch@math.uzh.ch}
\author{Yuri Tschinkel}
\address{
  Courant Institute,
  251 Mercer Street,
  New York, NY 10012, USA
}

\email{tschinkel@cims.nyu.edu}

\address{Simons Foundation\\
160 Fifth Avenue\\
New York, NY 10010\\
USA}

\date{December 21, 2023}

\begin{abstract}
We investigate the birational geometry of Deligne-Mumford stacks and define new birational invariants in this context.
\end{abstract}

\maketitle

\section{Introduction}
\label{sec.intro}

Birational geometry of algebraic varieties is a classical subject in algebraic geometry. 
%This goes back to the classical Italian school and has seen a new revival in recent years.
Closely related to this is $G$-equivariant birational geometry for a finite group $G$, also a well-studied area.
There are meaningful parallels 
between these theories: the study of birationality over an algebraically nonclosed field $k$ is analogous, in many aspects, to equivariant birationality over an algebraically closed field, where the role of $G$ is played by the absolute Galois group of $k$.
This point of view has been brought into focus by Manin~\cite{manin2}. 

In this paper, we explore the birational geometry of Deligne-Mumford stacks over fields of characteristic zero.
In the setting of a projective $G$-variety $X$ the associated Deligne-Mumford stack $[X/G]$ carries features such as stabilizer groups, but does not carry the full information of $X$ with $G$-action.
The passage
$$
X\actsfromright G \quad \rightsquigarrow \quad [X/G] \quad \rightsquigarrow \quad k(X)^G=k(X/G)
$$
from equivariant geometry to geometry of Deligne-Mumford stacks and subsequently to classical birational geometry
involves a loss information at each step.
There is extensive literature on 
rationality of fields of invariants and relevant obstructions, such as unramified 
cohomology
(see, e.g., \cite{Salt}, \cite{B-Brauer}, \cite{CT-S}, \cite{BT-noether}), 
as well as a growing literature on equivariant birationality (see, e.g., \cite{Pro-ICM} and references therein, \cite{BnG}, \cite{HKTsmall}).

The Burnside ring, introduced in \cite{KT}, encodes classical birational types:
\[ \Burn=\bigoplus_{n=0}^\infty  \Burn_n; \]
the grading is by dimension.
Refinements to Deligne-Mumford stacks and $G$-varieties, respectively, appear in \cite{Bbar} and \cite{BnG},
with corresponding Burnside groups and compatibility homomorphisms:
$$
\Burn_n(G) \quad \longrightarrow \quad \oBurn_n \quad \longrightarrow \quad \Burn_n.
$$

Important to the birational geometry of Deligne-Mumford stacks is the appropriate notion of birational equivalence.
In
\cite[Thm.\ 4.1]{Bbar},
the class in $\oBurn_n$ of an $n$-dimensional Deligne-Mumford stack (when this is defined) is shown to satisfy $[\cX]=[\cY]$ in $\oBurn_n$ whenever $\cX$ and $\cY$ are \emph{birationally equivalent}, in the sense of fitting into a diagram
\[
\xymatrix@C=12pt{
& \cX' \ar[dl] \ar[dr] \\
\cX && \cY
}
\]
where the morphisms are representable, proper, and birational.

A subtle phenomenon is the existence of $\cX$ and $\cY$, not birationally equivalent, with $[\cX]=[\cY]$ in $\oBurn_n$.
For instance, over an algebraically closed field we consider the unique (nontrivial) action of $G=\Z/2\Z$ on elliptic curves.
If $E$ and $F$ are non-isomorphic elliptic curves, then the quotient stacks $\cX=[E/G]$ and $\cY=[F/G]$ constitute such a pair ($n=1$).
By way of contrast, the class in $\Burn_n$, respectively in $\Burn_n(G)$, fully captures the birational type of a projective variety, respectively the equivariant birational type of a projective $G$-variety of dimension $n$.

In this paper, we introduce a refinement of $\oBurn_n$ which fully captures the relevant birational type.
This is an abelian group
\[
\cBurn=\bigoplus_{n=0}^\infty \cBurn_n
\]
defined by explicit generators and relations.
This is a refinement, in that a quotient by further explicit relations
\[
\ocBurn=\bigoplus_{n=0}^\infty \ocBurn_n,
\]
can be identified with the group from \cite{Bbar}:
\[ \ocBurn_n\cong \oBurn_n. \]

Our principal constructions and results are:
\begin{itemize}
\item The groups $\cBurn$ and $\ocBurn$, in Section~\ref{sect:g};
\item Associated birational invariants of Deligne-Mumford stacks, in Section~\ref{sect:burninv};
\item Specialization homomorphisms, in Section~\ref{sect:spec};
\item Comparisons among various Burnside and Grothendieck groups, in Section~\ref{sec.comparisons}.
\end{itemize}
A substantial part of the paper, Sections~\ref{sect:background} to \ref{sect:weak}, is devoted to foundations of birational geometry of stacks. Throughout, we provide examples illustrating various aspects of the theory.

\medskip
\noindent
{\bf Acknowledgments:} 
The second author was partially supported by NSF grant 2301983.

\section{Orbifolds and stabilizer groups}
\label{sect:background}
We work over a field $k$ of characteristic zero. 
We are interested in the birational geometry of \emph{orbifolds} over $k$.

Algebraic orbifolds are smooth separated irreducible Deligne-Mumford (DM) stacks of finite type over $k$, with trivial generic stabilizer.
All algebraic stacks in this section are of finite type over $k$, unless indicated otherwise.
All morphisms are morphisms over $k$.

\begin{exam}
\label{exam:bg}
A standard example of a DM stack is 
$$
\mathcal X=[X/G],
$$
where $X$ is an algebraic variety over $k$ with regular action of a finite group $G$. 
This is the category, where an object is a $G$-torsor
$E\to T$ (over an arbitrary $k$-scheme $T$) with a $G$-equivariant morphism $E\to X$, and a morphism is a $G$-equivariant isomorphism
of torsors (over a morphisms of $k$-schemes), compatible with the morphisms to $X$.
Informally, the stack $\mathcal X$ encodes the $G$-equivariant geometry of $X$, e.g., when $k$ is algebraically closed,
the category $\mathcal{X}(k)$ of $k$-points of $\mathcal{X}$ has isomorphism classes of objects in bijective correspondence with the $G$-orbits of $X(k)$.
The case $X=\Spec(k)$ gives
\[ BG, \]
the category of $G$-torsors (over $k$-schemes) and $G$-equivariant morphisms
(over morphisms of $k$-schemes).
\end{exam}

A separated DM stack, or more generally an algebraic stack satisfying the weaker hypothesis of having finite inertia, has a coarse moduli space.
For instance, $\mathcal X=[X/G]$ has coarse moduli space $X/G$, the quotient variety when this exists (e.g., when $X$ is quasi-projective), generally the algebraic space quotient.
When $X$ is nonsingular and the $G$-action is faithful, $\mathcal X=[X/G]$ is an example of an orbifold.

For an orbifold,
when the coarse moduli space is a projective scheme, we speak of a \emph{projective orbifold},
and when quasi-projective, a
\emph{quasi-projective orbifold}.
The same definitions apply to smooth separated DM stacks; cf.\ \cite{K-Seattle}.
If $\mathcal{X}$ is quasi-projective, then
$\mathcal X \cong [X/G]$ for some linear algebraic group $G$ acting on a smooth quasi-projective scheme $X$;
in fact, we may take $G$ to be $GL_N$ for some $N$.
But it is a subtle question,
whether this can be achieved with a \emph{finite} group $G$. Even in dimension one, there exist orbifolds not of the form $[X/G]$, for any variety $X$ and finite group $G$.
E.g., for any positive integers $a$ and $b$, not both equal to $1$, the stack
\[ \PP(a,b)=[\A^2\setminus\{0\}/\G_m],\]
with action $t\cdot(x,y)=(t^ax,t^by)$,
cannot be expressed as stack quotient of a scheme by a finite group (cf.\ \cite{behrendnoohi}).
This is an orbifold when $\gcd(a,b)=1$.

Algebraic stacks form a $2$-category, where besides morphisms (over $k$) there are $2$-morphisms given by natural transformations of functors (over identity morphisms in the base category of $k$-schemes).
The $2$-morphisms are all invertible.

We often perform local constructions, such as the projective fibration $\Proj \Sym \cE\to \cX$ associated with a coherent sheaf $\cE$ on a given algebraic stack $\cX$.
This, along with a more general \emph{projective} morphism to $\cX$, which by definition is one that admits a factorization up to $2$-isomorphism as a closed immersion to a projective fibration, provides an example of \emph{representable} morphism.
A morphism $\cY\to \cX$ is representable if, for any scheme $T$ with morphism to $\cX$, the fiber product $T\times_{\cX}\cY$ is an algebraic space.

\begin{lemm}
\label{lem.projective}
Let $f\colon \cY\to \cX$ be a projective morphism of separated DM stacks.
Then the induced morphism of coarse moduli spaces $e:Y\to X$ is projective.
\end{lemm}

\begin{proof}
%We denote the induced morphism of coarse moduli spaces by $e$.
The morphism $e$ is proper, since the composite $\cY\to X$ is proper and $\cY\to Y$ is proper surjective.
We will exhibit a line bundle $M$ on $Y$, whose pullback to a finite scheme cover of $Y$ is ample, relative to the composite morphism to $X$.
When $X$ and $Y$ are schemes, \cite[(2.6.2)]{EGAIII} tells us that $e$ is projective, with relatively ample line bundle $M$.
The proof given in \cite[(2.6.3)]{EGAIII} uses Chevalley's theorem.
This is available for algebraic spaces, and we reach the same conclusion, supposing just $X$ to be a scheme.
The general case reduces to this one, by base change to an \'etale atlas of $X$, since the condition of obtaining a closed immersion $Y\to \Proj\Sym e_*(M^n)$, for $n\gg 0$, may be checked \'etale locally.

By \cite[(16.6)]{LMB}, there exists a scheme $Z$ with a finite surjective morphism $Z\to \cX$.
We let $W=\cY\times_{\cX}Z$, with projection morphisms $p$ to $\cY$ and $g$ to $Z$, so $p$ is finite and surjective, and $g$ is projective.
The composite morphisms $r\colon Z\to X$ and $s\colon W\to Y$ are proper and quasi-finite, hence finite.
As well, they are surjective.

Factoring $f$ up to $2$-isomorphism as $\cY\to \Proj\Sym \cE\to \cX$ we let $L$ denote the restriction to $\cY$ of the line bundle $\cO_{\Proj\Sym \cE}(1)$.
By \cite[Lemma 2]{KV}, after replacing $L$ by a tensor power, we may suppose that $L$ is isomorphic to the pullback of a line bundle $M$ on $Y$.
So $s^*M\cong p^*L$ is ample relative to $g$, hence also relative to $r\circ g$, and $e$ is projective.
\end{proof}

Stacks with representable morphism to a given stack $\cY$ form, in a natural way, an
\emph{ordinary} category.
An object is a stack $\mathsf{X}$ with representable morphism
$\mathsf{f}\colon \mathsf{X}\to \cY$.
A morphism to another object $\mathsf{f}'\colon\mathsf{X}'\to \cY$ is an equivalence class of pairs $(g,\alpha)$,
consisting of a morphism $g\colon \mathsf{X}\to \mathsf{X}'$ and $2$-isomorphism $\alpha\colon \mathsf{f}\Rightarrow \mathsf{f}'\circ g$.
The equivalence relation is defined by $(g,\alpha)\sim (g',\alpha')$ when there exists a
$2$-isomorphism $\beta\colon g\Rightarrow g'$, such that
\[ \alpha'(\mathsf{x})=\mathsf{f}'(\beta(\mathsf{x}))\circ \alpha(\mathsf{x}) \]
for all objects $\mathsf{x}$ of $\mathsf{X}$.

A useful variant, given a stack $\cX$ with morphism $f\colon \cX\to \cY$, is the
category of factorizations of $f$ through a stack with representable morphism to $\cY$.
An object is a stack $\mathsf{X}$ with representable morphism
$\mathsf{f}\colon \mathsf{X}\to \cY$, morphism $\pi\colon \cX\to \mathsf{X}$,
and $2$-isomorphism $\iota\colon f\Rightarrow \mathsf{f}\circ \pi$.
A morphism to $\mathsf{f}'\colon \mathsf{X}'\to \cY$, with
$\pi'\colon \cX\to \mathsf{X}'$ and $\iota'\colon f\Rightarrow \mathsf{f}'\circ \pi'$, is
an equivalence class of pairs $(g,\alpha)$ as above,
admitting a compatible $2$-isomorphism $\pi'\Rightarrow g\circ \pi$.

Suppose that $\cX$ has a coarse moduli space
and $\cY$ has separated diagonal.
The construction of the coarse moduli space of $\cX$, carried out (smooth) locally over $\cY$, yields the
\emph{relative coarse moduli space} \cite[\S 3]{AOV}.
This is characterized as the
initial object in the category of factorizations of $f$ through a stack with representable morphism to $\cY$.
When this is $1_{\cY}$, with $f$ and $1_f$, we call $f$ a
\emph{coarsening morphism} \cite[\S 2.3]{toroidalorbifolds}.
Given a morphism $\cY\to \cZ$, the morphism $\pi\colon \cX\to \mathsf{X}$ induces an isomorphism of relative coarse moduli spaces over $\cZ$.
In particular, $\pi$ induces an isomorphism of coarse moduli spaces \cite[Lemma 2.3.4(iii)]{toroidalorbifolds}.

\begin{exam}
\label{ex.stabilizer}
The diagonal
\[ \Delta\colon \cX\to \cX\times \cX \]
is representable.
For a point of an algebraic stack $\cX$, represented by a morphism $x\colon \Spec(\ell)\to \cX$ for a field $\ell/k$, we obtain a stabilizer group scheme $G_x$, locally of finite type over $\ell$, fitting into a fiber diagram
\[
\xymatrix{
G_x \ar[r] \ar[d] & \mathcal{X} \ar[d]^{\Delta} \\
\Spec(\ell) \ar[r]^{\Delta\circ x} & \mathcal{X}\times \mathcal{X}
}
\]
with the diagonal $\Delta$ of $\mathcal{X}$.
For the stacks that we typically consider, DM stacks with finite inertia,
$G_x$ will be reduced and finite.
Passing to an algebraic closure of $\ell$ leads to a finite group,
the \emph{geometric stabilizer group} at $x$.
Fixing $\ell$, an isomorphism in $\mathcal{X}(\ell)$ is a $2$-isomorphism
\begin{equation}
\begin{split}
\label{eqn.2iso}
\xymatrix@C=36pt{
\Spec(\ell) \ar@/^1.0pc/[r]^x\ar@/_1.0pc/[r]_{x'}\ar@{}[r]|{\textstyle\Downarrow}& \mathcal{X}
}
\end{split}
\end{equation}
and this determines an isomorphism
\begin{equation}
\label{eqn.Gx}
G_x\to G_{x'}.
\end{equation}
A different choice of $2$-isomorphism $\Downarrow$ in \eqref{eqn.2iso} determines an isomorphism that differs from
\eqref{eqn.Gx} by an inner automorphism of $G_{x'}$.
In the sequel, DM stacks with \emph{abelian} geometric stabilizer groups, at least over a dense open substack $\mathcal{U}$ of $\cX$ will play the biggest role.
Then the geometric stabilizer groups of points of $\mathcal{U}$ are defined up to \emph{canonical} isomorphism.
\end{exam}

\begin{rema}
\label{rem.IX}
All of the stabilizers are packaged together in the \emph{inertia stack},
which is the stack $\cI_{\cX}$ in the fiber diagram
\[
\xymatrix{
\cI_{\cX} \ar[r] \ar[d] & \mathcal{X} \ar[d]^{\Delta} \\
\cX \ar[r]^(0.4){\Delta} & \mathcal{X}\times \mathcal{X}
}
\]
If $\cX$ is a smooth separated DM stack, then so is $\cI_{\cX}$, and so are the locally closed substacks of $\cX$, appearing in the stratification by isomorphism type of geometric stabilizer.
\end{rema}

\section{Normal crossing divisor and root construction}
\label{sect:root}
We continue to work with stacks of finite type over a field $k$ of characteristic $0$.

A \emph{normal crossing divisor} on a smooth separated DM stack $\cX$ is a divisor $\cD\subset \cX$, whose pre-image in one or equivalently every \'etale altas $U$ of $\cX$ is a normal crossing divisor on $U$.
If, furthermore, we can write
$\cD=\cD_1\cup\dots\cup \cD_{\ell}$, where each $\cD_i$ is smooth, then we call $\cD$ a \emph{simple normal crossing (snc) divisor}.
Equivalently, an snc divisor on $\cX$ is a normal crossing divisor whose irreducible components are smooth.

\begin{rema}
\label{rema:sing-res}
Desingularization of a reduced Noetherian quasi-excellent algebraic stack in characteristic $0$ is established in \cite{temkin-nonembedded}.
This is applicable to reduced stacks, of finite type over a field of characteristic zero. The case of reduced stacks of finite type over a DVR with residue field of characteristic zero, needed in Section \ref{sect:spec}, is also covered.
Embedded resolution of singularities is also available in this generality \cite{temkin-embedded}.
\end{rema}

Let $\cD$ be a arbitrary divisor on a smooth separated DM stack $\cX$, and let $r$ be a positive integer.
We recall the \emph{root stack}
\[ \sqrt[r]{(\cX,\cD)} 
\]
of \cite[\S 2]{cadman}, \cite[App.\ B]{AGV}.
The divisor $\cD$ determines a line bundle $\cO_{\cX}(\cD)$ with section vanishing along $\cD$, hence a morphism
\begin{equation}
\label{eqn.toGm}
\cX\to [\A^1/\G_m].
\end{equation}
The root stack is the fiber product
\[ \sqrt[r]{(\cX,\cD)}=\cX\times_{[\A^1/\G_m]}[\A^1/\G_m], \]
where $[\A^1/\G_m]\to [\A^1/\G_m]$ is the morphism induced by the $r$th power maps on $\A^1$ and on $\G_m$.
The substack
$$
\sqrt[r]{(\cX,\cD)}\times_{[\A^1/\G_m]}[\{0\}/\G_m]
$$ 
is the \emph{gerbe of the root stack} $\cG_{\cD}$ \cite[Def.\ 2.4.4]{cadman}; we recall that a gerbe is a morphism of stacks (in this case, $\cG_{\cD}\to \cD$) which admits (\'etale) local sections, such that any pair of local sections are locally isomorphic.
For $r>1$, the root stack is smooth if and only if $\cD$ is smooth.

\begin{lemm}
\label{lem.normalbundlerot}
Let $\cX$ be a smooth separated irreducible DM stack and $\cD$ a divisor on $\cX$ with normal bundle
$\cN=\cN_{\cD/\cX}$.
We denote by $\cZ$ the zero-section of $\cN$.
Then we have a canonical isomorphism
\[
\cN_{\cG_{\cD}/\sqrt[r]{(\cX,\cD)}}\cong
\sqrt[r]{(\cN,\cZ)}.
\]
In particular, the complement of the zero-section of $\cN_{\cG_{\cD}/\sqrt[r]{(\cX,\cD)}}$ is isomorphic to the complement of the zero-section of $\cN_{\cD/\cX}$.
\end{lemm}

\begin{proof}
There is a $2$-commutative diagram of stacks (\emph{not} a fiber diagram!)
\[
\xymatrix{
[\A^1/\G_m] \ar[r]^0\ar[d]_{({-})^r}& [\A^1/\G_m]\ar[d]^{({-})^r} \\
[\A^1/\G_m] \ar[r]^0 & [\A^1/\G_m]
}
\]
where the horizontal arrows are given by the $\G_m$-equivariant morphism sending $\A^1$ to $\{0\}\subset \A^1$ and
the vertical arrows are induced by the $r$th power maps on $\A^1$ and on $\G_m$.
When we perform base change by the morphism \eqref{eqn.toGm} we obtain
\[
\xymatrix{
[\A^1/\G_m]\times_{0\circ({-})^r,[\A^1/\G_m]}\cX \ar[r]\ar[d] & \sqrt[r]{(\cX,\cD)}\ar[d] \\
[\A^1/\G_m]\times_{0,[\A^1/\G_m]}\cX \ar[r] & \cX
}
\]

We analyze the terms on the left.
Identifying $[\A^1/\G_m]$ with the normal bundle to $[\{0\}/\G_m]$ in $[\A^1/\G_m]$, flatness of \eqref{eqn.toGm} leads to an identification
\[ [\A^1/\G_m]\times_{0,[\A^1/\G_m]}\cX\cong \cN_{\cD/\cX}. \]
Analogously,
\[ [\A^1/\G_m]\times_{0\circ({-})^r,[\A^1/\G_m]}\cX\cong \cN_{\cG_{\cD}/\sqrt[r]{(\cX,\cD)}}. \]
Now the fiber square
\[
\xymatrix{
\cN_{\cG_{\cD}/\sqrt[r]{(\cX,\cD)}}\ar[r] \ar[d] & [\A^1/\G_m]\ar[d]^{({-})^r} \\
\cN_{\cD/\cX} \ar[r] & [\A^1/\G_m]
}
\]
gives rise to an identification of $\cN_{\cG_{\cD}/\sqrt[r]{(\cX,\cD)}}$ with the $r$th root stack of $\cN_{\cD/\cX}$ along the zero-section.
\end{proof}

We return to the setting of an snc divisor $\cD=\cD_1\cup\dots\cup\cD_{\ell}$ on a smooth separated DM stack $\cX$.
As mentioned above, to obtain a smooth root stack $\sqrt[r]{(\cX,\cD)}$ (with $r>1$) we require $\cD$ to be smooth, equivalently, $\cD_i\cap \cD_j=\emptyset$ for $i\ne j$.
A different construction, the \emph{iterated root stack} \cite[Defn.\ 2.2.4]{cadman}, preserves smoothness when applied to an snc divisor.
Each $\cD_i$ individually determines a morphism to $[\A^1/\G_m]$, so together we have a morphism
\[ \cX\to [\A^{\ell}/\G_m^\ell]. \]
Let an $\ell$-tuple of positive integers $\mathbf{r}=(r_1,\dots,r_{\ell})$ be given.
The product of the morphisms $({-})^{r_i}$ is a morphism
\[ ({-})^{\mathbf{r}}\colon [\A^{\ell}/\G_m^\ell]\to [\A^{\ell}/\G_m^\ell]. \]
The iterated root stack is
\[ \sqrt[\mathbf{r}]{(\cX,\cD)}=\cX\times_{[\A^{\ell}/\G_m^\ell],({-})^{\mathbf{r}}}[\A^{\ell}/\G_m^\ell]. \]

Another sort of root stack, given in \cite[\S 2]{cadman}, \cite[App.\ B]{AGV},
is the $r$th root $\sqrt[r]{L/\cX}$ of a line bundle $L$ on $\cX$.
This is defined as
\begin{equation}
\label{eqn.rootlb}
\mathcal{X}\times_{B\G_m}B\G_m,
\end{equation}
the fiber product of $\mathcal{X}\to B\G_m$ associated with $L$, and the $r$th power map
$B\G_m\to B\G_m$.

\begin{lemm}
\label{lem.rootofroot}
Let $\cX$ be a smooth separated DM stack, $\cD$ a divisor on $\cX$, and $r$ and $s$ positive integers.
Let $\cX'=\sqrt[r]{(\cX,\cD)}$, with gerbe of the root stack $\cG_{\cD}$ and morphism $\pi\colon \cX'\to \cX$.
Then:
\begin{itemize}
\item[$\mathrm{(i)}$]
One has
\[ \sqrt[s]{(\cX',\cG_{\cD})}\cong \sqrt[rs]{(\cX,\cD)}. \]
\item[$\mathrm{(ii)}$]
Let $L$ be a line bundle on $\cX$ and
$L'=\pi^*L$.
Then
\[ \sqrt[s]{L'/\cX'}\cong \sqrt[r]{\big(\sqrt[s]{L/\cX},\cD\times_{\cX}\sqrt[s]{L/\cX}\big)} \]
\end{itemize}
\end{lemm}

\begin{proof}
Statement (i) follows from the definition, since the composite
\[ [\A^1/\G_m]\stackrel{({-})^s}\longrightarrow [\A^1/\G_m]\stackrel{({-})^r}\longrightarrow [\A^1/\G_m] \]
is equal (up to $2$-isomorphism) to
$$
({-})^{rs}\colon [\A^1/\G_m]\to [\A^1/\G_m].
$$
For (ii), we argue similarly, using the two composite morphisms 
$$
({-})^r\times 1\circ 1\times ({-})^s\quad \text{ and } \quad 1\times ({-})^s\circ ({-})^r\times 1
$$ 
from 
$[\A^1/\G_m]\times B\G_m$ to $[\A^1/\G_m]\times B\G_m$.
\end{proof}

\begin{prop}
\label{prop.rootcoarse}
Let $\cD$ be a divisor on a smooth separated DM stack $\cX$, and $r$ a positive integer.
Then
\[ \sqrt[r]{(\cX,\cD)}\to \cX \]
is a coarsening morphism.
The same holds when $\cD=\cD_1\cup\dots\cup \cD_{\ell}$ is an snc divisor for the iterated root stack
\[ \sqrt[\mathbf{r}]{(\cX,\cD)}\to \cX, \]
for any $\mathbf{r}$, and for the $r$th root
\[ \sqrt[r]{L/X}\to \cX \]
of any line bundle $L$.
\end{prop}

\begin{proof}
A root stack is a coarsening morphism, by \cite[Cor.\ 2.3.7]{cadman}.
The case of an iterated root stack is taken care of by the observation \cite[Rmk.\ 2.2.5]{cadman} that an iterated root stack may be obtained as a sequence of individual root stacks.
\end{proof}

\section{Birationality}
\label{sect:birationality}
Interest in birational geometry of stacks was stimulated by progress in the
Minimal Model Program
for moduli stacks, see e.g.~\cite{HH},
from which we have taken the notions in Definition \ref{def:birs}, below. All stacks in this section are DM stacks, separated and of finite type over $k$.
(In several places, clearly indicated, there are DM stacks which are separated and of finite type over a different field.) 

\begin{defi}
\label{def:birs}
Let $\cX$ and $\cY$ be stacks.
\begin{itemize}
\item[(i)]
A morphism $f\colon \cX\to \cY$ is \emph{birational} if $f$ restricts to an isomorphism $\cU\to \cV$ of dense open substacks of $\cX$, respectively $\cY$.
\item[(ii)]
A \emph{rational map} $\cX\dashrightarrow \cY$ is an object of the category
\[ \MAP(\cX,\cY)=\colim_{\cU\subset \cX} \HOM(\cU,\cY), \]
where the colimit is taken over dense open substacks $\cU\subset \cX$.
So, a rational map is represented by
$f\colon \cU\to \cY$ for some such $\cU$, with $(\cU,f)\cong (\cU',f')$
if and only if $f$ and $f'$ become $2$-isomorphic after restriction to common dense open.
\item[(iii)]
A rational map $\cX\dashrightarrow \cY$ is \emph{dominant} if, for some or equivalently every representative $(\cU,f)$, $f(\cU)$ is dense in $\cY$.
\item[(iv)]
A rational map $\cX\dashrightarrow \cY$ is \emph{birational} if, for some representative $(\cU,f)$,
$f$ maps $\cU$ isomorphically to a dense open substack of $\cY$.
\item[(v)]
A rational map $\cX\dashrightarrow \cY$ is \emph{proper} if there exist a representative $(\cU,f)$, a
proper birational morphism
$\cX'\to \cX$, and a proper morphism $\cX'\to \cY$ such that $f$ factors up to $2$-isomorphism as
\begin{equation}
\label{eqn.ffactors}
\cU\cong \cX'\times_{\cX}\cU\to \cY.
\end{equation}
\end{itemize}
\end{defi}

\begin{prop}
\label{prop.reprmaxopen}
Let $f\colon \cX'\to \cX$ be a representable proper birational morphism of stacks, with $\cX$ normal and $\cX'$ reduced.
Then the maximal open substack $\cU\subset \cX$, over which $f$ restricts to an isomorphism
$f^{-1}(\cU)\to \cU$,
has complement everywhere of codimension $\ge 2$.
\end{prop}

\begin{proof}
For any finite-type morphism of locally Noetherian schemes, the set of points of the source, where the morphism is flat, is open by \cite[(11.1.1)]{EGAIV}.
By the local nature of flatness, we have the same for any morphism of stacks.
Letting $\cU'$ denote the locus where $f$ is open, we claim that $\cU$ is the complement of $f(\cX'\setminus\cU')$.
A representable proper birational morphism, if additionally flat, must be an isomorphism.
So
\[ \cX\setminus f(\cX'\setminus\cU')\subset \cU, \]
and we have equality, since
$\cU'$ contains $f^{-1}(\cU)$.

This observation shows that the formation of $\cU$ from $f\colon \cX'\to \cX$ commutes with \'etale base change.
Thus we are reduced to proving the statement when $\cX$ is a scheme and $\cX'$ an algebraic space.
Then we obtain the statement by a straightforward application of the valuative criterion for properness.
\end{proof}

\begin{prop}
\label{prop:birs}
Let $\cX$ and $\cY$ be stacks, with $\cX$ reduced, and $\cX\dashrightarrow \cY$ a rational map.
\begin{itemize}
\item[$\mathrm{(i)}$]
The closure $\Gamma(\cX\dashrightarrow \cY)$ in $\cX\times \cY$ of the image of the graph
\[ \cU\stackrel{(j,f)}\to \cX\times \cY, \]
where
$(\cU,f)$ is a representative of the rational map and $j\colon \cU\to \cX$ denotes the inclusion, is independent of the choice of representative.
The rational map is proper if and only if $\Gamma(\cX\dashrightarrow \cY)\to \cX$ and $\Gamma(\cX\dashrightarrow \cY)\to \cY$ are proper.
\item[$\mathrm{(ii)}$]
Suppose $(\cU,f)$ is a representative of the rational map $\cX\dashrightarrow \cY$.
Then there exists a reduced stack $\cX'$ with birational morphism to $\cX$ and morphism to $\cY$, such that
$f$ factors up to $2$-isomorphism as \eqref{eqn.ffactors} and the induced morphism $\cX'\to \Gamma(\cX\dashrightarrow \cY)$ is projective.
\end{itemize}
\end{prop}

\begin{proof}
The independence of choice of representative follows from the fact that $\cX$ is reduced.
Suppose $\cX\dashrightarrow \cY$ is proper, i.e., there exist $\cX'\to \cX$ and $\cX'\to \cY$ as in Definition \ref{def:birs}(v);
we may suppose $\cX'$ reduced.
The corresponding morphism
$\cX'\to \cX\times \cY$
factors through $\cZ=\Gamma(\cX\dashrightarrow \cY)$.
Since $\cX'\to \cX$ is proper
and $\cZ\to \cX$ is separated, $\cX'\to \cZ$ is proper.
Since, also, $\cX'\to \cZ$ is surjective, properness of
$\cX'\to \cX$ implies that $\cZ\to \cX$ is proper, and the same holds with $\cY$ in place of $\cX$.

The morphism $(j,f)$ factors as the composition of $j\times 1_{\cY}$ with the graph of $f$.
The latter is obtained by base change from the diagonal of $\cY$, hence is representable and finite and thus is projective.
So we have a projective morphism to the open substack $\cU\times \cY\subset \cX\times \cY$, and this may be compactified, by applying \cite[(15.5)]{LMB} to realize the coherent sheaf on the open substack in the definition of projective morphism 
as restriction of a coherent sheaf on the ambient stack, and forming the closure of $\cU$ in the projective fibration over the ambient stack.
This yields a stack $\cX''$, with projective morphism to $\cZ$ and $\cU\cong \cX''\times_{\cX}\cU$.
Now if $\cZ\to \cX$ and $\cZ\to \cY$ are proper, then also $\cX''\to \cX$ and $\cX''\to \cY$ are proper, and the conditions in Definition \ref{def:birs}(v) are met.
We have established (i) and, by the construction of $\cX''$, also (ii).
\end{proof}

\begin{lemm}
\label{lem.openproper}
Let $\cX\dashrightarrow \cY$ be a proper birational map of stacks, and let $\cS\subset \cX$ and $\cT\subset \cY$ be dense open substacks.
Then the corresponding birational map $\cS\dashrightarrow \cT$ is proper if and only if for one or equivalently every factorization \eqref{eqn.ffactors}, the pre-images of $\cS$ and $\cT$ in $\cX'$ are equal.
\end{lemm}

\begin{proof}
First we show that equality of pre-images in a factorization \eqref{eqn.ffactors} implies properness of $\cS\dashrightarrow \cT$.
Let $\cS'\subset \cX'$ denote the common pre-image.
Then the proper birational morphisms $\cS'\to \cS$ and $\cS'\to \cT$ exhibit $\cS\dashrightarrow \cT$ as proper.

It remains to deduce equality of pre-images from the properness of $\cS\dashrightarrow \cT$.
For this we may suppose without loss of generality that $\cX$ is reduced.
Then we have $\Gamma(\cX\dashrightarrow \cY)$ as in Proposition \ref{prop:birs}, and
\[ \Gamma(\cS\dashrightarrow \cT)=p^{-1}(\cS)\cap q^{-1}(\cT), \]
where
$p\colon \Gamma(\cX\dashrightarrow \cY)\to \cX$ and
$q\colon \Gamma(\cX\dashrightarrow \cY)\to \cY$ denote the projections.
If $\cS\dashrightarrow \cT$ is proper, then by Proposition \ref{prop:birs}(i), the composite morphisms
\[
\Gamma(\cS\dashrightarrow \cT)\to p^{-1}(\cS)\to \cS
\]
and
\[
\Gamma(\cS\dashrightarrow \cT)\to q^{-1}(\cT)\to \cT
\]
are proper.
So in each composition the first map, a dense open immersion, is proper, hence an isomorphism.
Thus $p^{-1}(\cS)=q^{-1}(\cT)$, and from this follows the equality of pre-images in $\cX'$ for any factorization \eqref{eqn.ffactors}.
\end{proof}

For schemes, the canonical representative on the maximal domain of definition of a rational map requires the source to be reduced.
For stacks, even more is required: with ``ghost automorphisms'' (see, e.g., \cite[\S 1.2.1]{AF}), we have examples of pairs of isomorphisms $f$, $f'\colon \cX\to \cY$ of reduced stacks, such that $f$ and $f'$ are not $2$-isomorphic, but become $2$-isomorphic upon restriction to a dense open substack of $\cX$.

The role of the residue field at a point of a scheme is, for a stack, played by the \emph{residual gerbe}
\cite[Thm.\ B.2]{rydhdevissage}: to a point $x$ of a stack $\cX$ there is a monomorphism $\cG_x\to \cX$ and a gerbe $\cG_x\to \Spec(\kappa(x))$.
There is a finite surjective morphism to $\cG_x$ from $\Spec(\lambda)$ for some finite extension field $\lambda/\kappa(x)$, so by Chevalley's theorem the morphism $\cG_x\to \cX$ is affine.

\begin{prop}
\label{prop.bireasy}
Let $\cX$ and $\cY$ be stacks, with $\cX$ normal.
Every rational map $\cX\dashrightarrow \cY$ has a well-defined maximal domain of definition $\cU$, with the property that up to $2$-isomorphism there is a representative $(\cU,f)$, and for any $(\cU',f')$ with
dense open $\cU''\subset \cU\cap \cU'$ and
$2$-isomorphism $\alpha\colon f|_{\cU''}\Rightarrow f'|_{\cU''}$, we have
$\cU'\subset \cU$, and $\alpha$ extends uniquely to $f|_{\cU'}\Rightarrow f'$.
\end{prop}

\begin{proof}
Let $\cU$ be a dense open substack of $\cX$.
We start by showing that, for $f$, $f'\colon \cX\to \cY$, any $2$-isomorphism $f|_{\cU}\Rightarrow f'|_{\cU}$ extends uniquely to a $2$-isomorphism $f\Rightarrow f'$.
A given $2$-isomorphism $\alpha\colon f|_{\cU}\Rightarrow f'|_{\cU}$ determines a section over $\cU$ of the representable finite morphism
\[
\xymatrix{
\cX\times_{\cY\times \cY}\cY \ar[d]^{\mathrm{pr}_1} \\ \cX
}
\]
where the morphisms to $\cY\times \cY$ are $(f,f')$ and the diagonal $\Delta_{\cY}$.
Since $\cX$ is normal, the closure of the section is isomorphic to $\cX$.
Thus $\alpha$ extends uniquely.

Since $\cX$ is Noetherian, there is a maximal domain of definition $\cU$ of $\cX\dashrightarrow \cY$.
Let $(\cU,f)$ be a representative,
and let $(\cU',f')$ be another representative, with
dense open $\cU''\subset \cU\cap \cU'$ and $2$-isomorphism $\alpha\colon f|_{\cU''}\Rightarrow f'|_{\cU''}$.
By what we have shown above, $\alpha$ extends uniquely to $\alpha'\colon f|_{\cU\cap \cU'}\Rightarrow f'|_{\cU\cap \cU'}$.
With $(\cU,f)$, $(\cU',f')$, and $\alpha'$ we see that there is a representative with domain of definition $\cU\cup \cU'$, so by maximality we must have $\cU'\subset \cU$, and we have unique $\alpha'\colon f|_{\cU'}\Rightarrow f'$ extending $\alpha$.
\end{proof}

\begin{rema}
\label{rem.ismorphism}
Let the notation be as in Proposition \ref{prop.bireasy}.
Then $\cX\dashrightarrow \cY$ has maximal domain of definition equal to $\cX$ if and only if the rational map lies in the essential image of the fully faithful functor
\[ \HOM(\cX,\cY)\to \MAP(\cX,\cY). \]
To indicate this situation we permit ourselves to employ the descriptive phrase \emph{is a morphism}, and also to stipulate a further property, e.g., \emph{is a proper morphism}.
\end{rema}

\begin{coro}
\label{cor.bircompatible}
Let $\cX$ and $\cY$ be stacks, with $\cX$ normal,
$\cX\dashrightarrow \cY$ a rational map, $\cU$ the maximal domain of definition, and $g\colon \cZ\to \cX$ an \'etale morphism of stacks.
Then $g^{-1}(\cU)$ is the maximal domain of definition of the composite rational map $\cZ\dashrightarrow \cY$.
\end{coro}

\begin{proof}
There is no loss of generality in supposing $g$ surjective.
The maximal domain of definition of $\cZ\dashrightarrow \cY$ contains
$g^{-1}(\cU)$.
So it suffices to show that if $\cZ\dashrightarrow \cY$ is a morphism, then so is $\cX\dashrightarrow \cY$.

First, we suppose $\cZ=Z$, a scheme,
so, an \'etale atlas for $\cX$.
Then $R=Z\times_{\cX}Z$ is also a scheme, and with its projection morphisms we have the structure of groupoid in schemes $R\rightrightarrows Z$ that recovers $\cX$ in the sense that a naturally defined prestack $[R\rightrightarrows Z]^{\mathrm{pre}}$ has $\cX$ as stackification
\cite{vistoli}.
The two composite morphisms $R\rightrightarrows Z\to \cY$,
restricted to $\cU\times_{\cX}R$,
are canonically identified by the identification of the latter with $g^{-1}(\cU)\times_{\cU}g^{-1}(\cU)$, hence by Proposition \ref{prop.bireasy} are themselves canonically $2$-isomorphic.
The object of $\cY$ over $Z$, corresponding to $Z\to \cY$, and the isomorphism of the two pullbacks to $R$, given by the $2$-isomorphism, satisfy the compatibility condition (commutative diagram of isomorphisms in $\cY$ over a fiber product $R\times_ZR$) to determine a morphism $[R\rightrightarrows Z]^{\mathrm{pre}}\to \cY$
(again by the uniqueness assertion of Proposition \ref{prop.bireasy}).
We obtain $\cX\to \cY$ by the stackification property.

The general case follows by taking an \'etale atlas $V\to \cZ$ and applying the previous case to the composite $V\to \cX$.
\end{proof}

\begin{coro}
\label{cor.bireasy}
Let $\cX$ and $\cY$ be stacks, with $\cX$ normal.
Every birational map
$\cX\dashrightarrow \cY$
has a well-defined maximal open $\cU\subset \cX$, mapped isomorphically by a representative morphism to a dense open substack of $\cY$.
If $\cU'\subset \cX$ is open and a representive morphism maps $\cU'$ isomorphically to a dense open substack of $\cY$, then $\cU'\subset \cU$.
\end{coro}

\begin{proof}
Nothing changes if we replace $\cY$ by its normal locus, so we may suppose $\cY$ normal as well.
Since $\cX$ is Noetherian, a maximal $\cU$ exists.
Let $(\cU,f)$ be a representative, where $f$ maps $\cU$ isomorphically to $\cV\subset \cY$, and let $(\cU',f')$ be another representative, such that $f'$ maps $\cU'$ isomorphically to $\cV'\subset \cY$.
We apply Proposition \ref{prop.bireasy} to obtain a representative with domain of definition $\cU\cup \cU'$ extending $f$ and $f'$, for the inverse rational map a representative with domain of definition $\cV\cup \cV'$ extending $f^{-1}$ and $f'^{-1}$, and $2$-isomorphisms from the composite in either order to the identity.
So, by maximality we have $\cU'\subset \cU$.
\end{proof}

\begin{lemm}
\label{lem.extendtoY}
Let $\cX'\to \cX$ be a proper birational morphism of stacks that is a coarsening morphism, with $\cX$ smooth, and let $\cX'\to \cY$ be a morphism of stacks.
Suppose that the corresponding rational map $\cX\dashrightarrow \cY$ admits a representative $(\cU,f)$, such that the complement of $\cU$ has everywhere codimension $\ge 2$.
Then $\cX\dashrightarrow \cY$ is a morphism.
\end{lemm}

\begin{proof}
First we treat the case that $\cX=X$ is a scheme, and
$\cY=[W/G]$ for some affine $W$ and finite group $G$.
The morphism $f$ determines a $G$-torsor $E\to \cU$ and $G$-equivariant morphism $E\to W$.
By Zariski-Nagata purity, $E$ extends to a $G$-torsor on $X$, and the regular functions defining the morphism to $W$ extend as well, so in this case the result is established without even requiring to have a morphism $\cX'\to \cY$.

Next, we treat the case that $\cX=X$ is a scheme and $\cY$ is general.
Without loss of generality we may suppose that $\cU$ is the
maximal domain of definition and argue by contradiction, taking $x$ to be a point of the complement of $\cU$, above $x$ the point $x'$ of $\cX'$, mapping to the point $y$ of $\cY$.
Letting $Y$ denote the coarse moduli space of $\cY$, there is
an \'etale neighborhood $Z\to Y$ of the image in $Y$ of $y$, such that $\cY\times_YZ\cong [W/G]$ for some $W$ and $G$ as above (by construction of $Y$ \cite{keelmori}).
Since $X$ is the coarse moduli space of $\cX'$ the morphism $\cX'\to \cY$ determines $X\to Y$.
By the previous case, $X\times_YZ\dashrightarrow \cY\times_YZ$ is a morphism, hence as well $X\times_YZ\dashrightarrow \cY$ is a morphism.
By Corollary \ref{cor.bircompatible}, the maximal domain of definition of $X\dashrightarrow \cY$ contains $x$, and we have a contradiction.

Finally, we treat the general case.
Let $V\to \cX$ be an \'etale atlas, so
$V$ is the coarse moduli space of $V\times_{\cX}\cX'$.
By the previous case, $V\dashrightarrow \cY$ is a morphism.
We conclude by Corollary \ref{cor.bircompatible}.
\end{proof}

The composition of a pair of rational maps
\[ \cX\dashrightarrow \cY\qquad\text{and}\qquad \cY\dashrightarrow \cZ \]
is defined, provided $\cX\dashrightarrow \cY$ is dominant.

\begin{prop}
\label{prop.biratcompose}
Let $\cX$, $\cY$, and $\cZ$ be stacks, $\cX\dashrightarrow \cY$ a dominant rational map, $\cY\dashrightarrow \cZ$ a rational map, and $\cX\dashrightarrow \cZ$ the composite rational map.
If any two of the rational maps
\[
\cX\dashrightarrow \cY,\qquad \cY\dashrightarrow \cZ,\qquad \cX\dashrightarrow \cZ \]
are proper, then so is the third.
\end{prop}

\begin{proof}
Suppose, first, $\cX\dashrightarrow \cY$ and $\cY\dashrightarrow \cZ$ are proper.
Let $(\cU,f)$ and $\cX'$ be as in Definition \ref{def:birs}(v),
and let $(\cV,g)$ and $\cY'$ play the analogous role for $\cY\dashrightarrow \cZ$; without loss of generality we have
$f(\cU)\subset \cV$.
Then, with the stack-theoretic closure of $\cU$ in
$\cX'\times_{\cY}\cY'$
(closed substack, defined in an altas by the
scheme-theoretic closure),
we see that $\cX\dashrightarrow \cZ$ is proper.

For the remainder of the proof, we impose the additional hypothesis that $\cX$ and $\cY$ are reduced, which we may do without loss of generality.

Suppose that $\cX\dashrightarrow \cY$ and $\cX\dashrightarrow \cZ$ are proper.
We take $\cX'$ as above, let $\cX''$ play the analogous role for $\cX\dashrightarrow \cZ$,
and suppose that respective representatives have a common
domain of definition $\cU$.
Now we consider the closure of $\cU$ in
$\cX'\times_{\cX}\cX''$,
which we denote by $\cX'''$.
The composite morphisms from $\cX'''$ to $\cY$ and $\cZ$ induce a morphism to $\cY\times \cZ$, which factors through $\Gamma(\cY\dashrightarrow \cZ)$.
We see as in the proof of Proposition \ref{prop:birs}(i) that $\Gamma(\cY\dashrightarrow \cZ)\to \cY$ and $\Gamma(\cY\dashrightarrow \cZ)\to \cZ$ are proper. So $\cY\dashrightarrow \cZ$ is proper by Proposition \ref{prop:birs}(i).

It remains to show that if
$\cX\dashrightarrow \cZ$ and
$\cY\dashrightarrow \cZ$ are proper,
then $\cX\dashrightarrow \cY$ is proper.
It suffices to treat the case that we have proper morphisms $h\colon \cX\to \cZ$ and $g\colon \cY\to \cZ$.
Let $(\cU,f)$ be a representative for
$\cX\dashrightarrow \cY$, for which we have a $2$-isomorphism from $g\circ f$ to
$h|_{\cU}\colon \cU\to \cZ$.
The latter determines
\[ \cU\to \cX\times_{\cZ}\cY\cong (\cX\times \cY)\times_{\cZ\times \cZ}\cZ. \]
The diagonal of $\cZ$ is finite, hence so is
$\cX\times_{\cZ}\cY\to \cX\times \cY$,
and the closure $\cV$ of the image of $\cU$ in $\cX\times_{\cZ}\cY$
has a finite surjective morphism to $\Gamma(\cX\dashrightarrow \cY)\subset \cX\times \cY$.
The morphisms $\cV\to \cX$ and $\cV\to \cY$ are proper, hence so are $\Gamma(\cX\dashrightarrow \cY)\to \cX$ and $\Gamma(\cX\dashrightarrow \cY)\to \cY$, and $\cX\dashrightarrow \cY$ is proper by Proposition \ref{prop:birs}(i).
\end{proof}

\begin{lemm}
\label{lem.birational}
Let $f\colon \cX'\to \cX$ be a representable proper birational morphism of smooth stacks, restricting to an isomorphism $\cU'\cong \cU$ of dense open substacks.
Then there exist smooth stacks $\mathcal{X}_0$, $\dots$, $\mathcal{X}_m$, for some $m$, such that $\cX_0=\cX'$, $\cX_m=\cX$,
and for every $i$ there is a morphism
$\mathcal{X}_{i+1}\to \mathcal{X}_i$ or
$\mathcal{X}_i\to \mathcal{X}_{i+1}$ given by blowing up a smooth substack:
\[
\cX_{i+1}\cong B\ell_{\cZ_i}\cX_i\to \cX_i,\quad \text{resp.} \quad \cX_i\cong B\ell_{\cW_{i+1}}\cX_{i+1}\to \cX_{i+1}.
\]
We require compatibility with the isomorphism $f$: starting with $\cU_0=\cU'$ and continuing to $\cU_m=\cU$, we have
\[
\cU_i\cap \cZ_i=\emptyset, \cU_{i+1}=\cX_{i+1}\times_{\cX_i}\cU_i,\, \text{resp.} \, \cU_{i+1}\cap \cW_{i+1}=\emptyset, \cU_i=\cX_i\times_{\cX_{i+1}}\cU_{i+1},\]
for all $i$, and the composite isomorphism 
\[
\cU'=\cU_0\cong\dots\cong \cU_m=\cU
\] 
is $2$-isomorphic to the restriction of $f$.
Furthermore, for every $i$ the rational map
$\cX_i\dashrightarrow\cX$represented by $(\cU_i,\cU_i\cong\dots\cong \cU)$ is a representable proper morphism.
\end{lemm}

\begin{proof}
If $f$ is an isomorphism, then we are done with $m=1$ and $\cZ_0=\emptyset$, so we suppose the contrary.
We get the result by applying
\emph{weak factorization} of Abramovich and Temkin \cite{AT} to $f$.
By \cite[\S 1.6]{AT} there is a coherent sheaf of ideals on $\cX$, defining a closed substack disjoint from $\cU$, such that the associated blow-up $\cX''\to \cX$ factors up to $2$-isomorphism through $f$, and $\cX''$ is also the blow-up of a
coherent sheaf of ideals on $\cX'$.
We use resolution of singularities to obtain $\cX'''\to \cX''$, with $\cX'''$ smooth.
The morphism $\cX'''\to \cX$ is a composite of blow-ups, so by
\cite[(5.1.4)]{RG} (really, its proof, carried over to the setting of stacks using \cite[(15.5)]{LMB} for the needed existence of coherent sheaves), there is a coherent sheaf of ideals on $\cX$ whose blow-up yields $\cX'''$, and
the same holds with $\cX'$ in place of $\cX$.
We apply
\cite[Thm.\ 6.1.3]{AT} to obtain weak factorizations $\cX'=\mathcal{X}_0$, $\dots$, $\mathcal{X}_c=\cX'''$ and $\mathcal{X}_c$, $\dots$, $\mathcal{X}_m=\cX$.
By condition (5) of \cite[\S 1.2]{AT}, the $\mathcal{X}_i\dashrightarrow \cX'$ for $i\le c$ and the $\mathcal{X}_i\dashrightarrow \cX$ for $i\ge c$ are projective morphisms.
\end{proof}

\begin{lemm}
\label{lem.representable}
Let $\cX$ and $\cY$ be stacks, with $\cX$ smooth, and let
\[ \cX'\stackrel{f}\longrightarrow
\cX\stackrel{g}\longrightarrow \cY \]
be morphisms, such that $f$ is representable, proper, and birational.
Then $g$ is representable if and only if $g\circ f$ is representable.
\end{lemm}

\begin{proof}
Since a composite of representable morphisms is representable, the only nontrivial assertion is that representability of $g\circ f$ implies representability of $g$.
By resolution of singularities, we may suppose that $\cX'$ is smooth.
We apply Lemma \ref{lem.birational} to $f$ and, with the factorization via $\cX_i$, $i=0$, $\dots$, $m$, we see that it suffices to deduce representability of $\cX_{i+1}\to \cY$ from representability of $\cX_i\to \cY$, for every $i$.
This is an instance of the result we are trying to prove, specifically when $f$ is the blowing up of a smooth substack of $\cX$.
So it suffices to treat the case $\mathcal{X}'\cong B\ell_{\mathcal{Z}}\mathcal{X}$, which we do by contradiction.
Suppose $g$ is not representable.
Then, for some $k$-point $x$ of $\mathcal{X}$, mapping to $k$-point $y$ of $\mathcal{Y}$, there is $1\ne \gamma\in G_x$, such that $\gamma$ is sent by $g$ to the identity element of $G_y$.
If $x\notin \mathcal{Z}$, then $G_x=G_{x'}$, where $f(x')=x$, and $g\circ f$ is not representable.
If $x\in \mathcal{Z}$, then the normal bundle $\mathcal{N}_{\mathcal{Z}/\mathcal{X}}$ gives rise to a representation of the cyclic group $\langle \gamma\rangle$, which splits as a sum of one-dimensional representations.
The fiber under $f$ is the projectivization, and to each one-dimensional representation there corresponds a point whose stabilizer group contains $\gamma$.
So $g\circ f$ is not representable.
\end{proof}

\begin{defi}
\label{def.reprat}
Let $\cX$ and $\cY$ be stacks, with $\cX$ smooth.
We say that a rational map $\cX\dashrightarrow \cY$ is \emph{representable}
if there exist a representative $(\cU,f)$, a representable birational morphism $\cX'\to \cX$, and a representable morphism $\cX'\to \cY$, such that $f$ factors up to $2$-isomorphism as \eqref{eqn.ffactors}, and
the induced morphism $\cX'\to \Gamma(\cX\dashrightarrow \cY)$ is proper.
\end{defi}

\begin{exam}
\label{exa.reproverZ}
Let $\cX$, $\cY$, and $\cZ$ be stacks, and let
\[ g\colon \cX\to \cZ \qquad\text{and}\qquad h\colon \cY\to \cZ \]
be representable morphisms.
Suppose that $\cX$ is smooth.
Then any rational map $\cX\dashrightarrow \cY$, such that the composite with $h$ is $2$-isomorphic to $g$, is representable.
(For a representative $(\cU,f)$, the closure of the graph in $\cX\times_{\cZ}\cY$ supplies $\cX'$ as in Definition \ref{def.reprat}.)
The particular case $\cZ=BG$ tells us that $G$-equivariant rational maps of smooth varieties determine representable maps of quotient stacks.
\end{exam}

\begin{prop}
\label{prop:birr}
Let $\cX\dashrightarrow \cY$ be a rational map of stacks, with $\cX$ smooth.
Then $\cX\dashrightarrow \cY$ is representable if and only if the following two conditions are satisfied:
\begin{itemize}
\item[$\mathrm{(i)}$] There exist a representative $(\cU,f)$, representable birational morphism $\cX'\to \cX$, and morphism $\cX'\to \cY$, such that $\cX'$ is smooth, $f$ factors up to $2$-isomorphism as \eqref{eqn.ffactors}, and the induced morphism $\cX'\to \Gamma(\cX\dashrightarrow \cY)$ is representable and proper.
\item[$\mathrm{(ii)}$] For one or, equivalently, every representative and pair of morphisms in $\mathrm{(i)}$, the morphism $\cX'\to \cY$ is representable.
\end{itemize}
\end{prop}

\begin{proof}
Suppose that the rational map $\cX\dashrightarrow \cY$ is representable.
In Definition \ref{def.reprat}, there is no loss of generality in supposing $\cX'$ to be reduced, and by resolution of singularities, we may further suppose that $\cX'$ is smooth.
Conditions (i) and (ii) are satisfied.

If we have $(\cU,f)$, $\cX'\to \cX$ and $\cX'\to \cY$ satisfying (i) and (ii), then Definition \ref{def.reprat} is satisfied, so $\cX\dashrightarrow \cY$ is representable.
It remains only to verify that if statement (ii) holds for one representative and pair of morphisms in (i), then it holds for every representative and pair of morphisms.
Given a representative $(\cU,f)$ and pair of morphisms $\cX'\to \cX$ and $\cX'\to \cY$ as in (i), with $\cX'\to \cY$ representable, let $\cX''\to \cX$ and $\cX''\to \cY$ be another pair of morphisms as in (i), where the representative may be taken as well to be $(\cU,f)$, after possibly shrinking $\cU$.
We define $\cX'''$ by applying resolution of singularities to the closure of $\cU$ in $\cX'\times_{\cX}\cX''$.
Now $\cX'''\to \cY$ is representable, since this factors up to $2$-isomorphism through $\cX'$.
Applying Lemma \ref{lem.representable} to $\cX'''\to \cX''$ and $\cX''\to \cY$, we deduce that the latter is representable.
\end{proof}

\begin{rema}
\label{rem.normalizegraph}
We give another criterion for representability of $\cX\dashrightarrow \cY$.
Let $(\cU,f)$ be a representative, where $\cU$ is chosen to have affine coarse moduli space.
From the fact that $\cX$ and $\cY$ have finite diagonal, we deduce that the inclusion $j\colon \cU\to \cX$ and the induced $(j,f)\colon \cU\to \Gamma(\cX\dashrightarrow \cY)$ are affine morphisms.
The integral closure of $\Gamma(\cX\dashrightarrow \cY)$ relative to $(j,f)_*\cO_{\cU}$ (a local construction for stacks \cite[Chap.\ 14]{LMB}) is a normal stack $\cW$ with finite morphism to $\Gamma(\cX\dashrightarrow \cY)$ and dense open substack isomorphic to $\cU$.
We may see:
\begin{itemize}
\item A different choice of $\cU$ leads to the same $\cW$ (up to canonical isomorphism).
\item For any normal stack $\cX'$ with birational morphism to $\cX$ and morphism to $\cY$, such that $f$ factors up to $2$-isomorphism as \eqref{eqn.ffactors}, the induced morphism $\cX'\to \Gamma(\cX\dashrightarrow \cY)$ factors up to $2$-morphism through $\cW$.
\end{itemize}
With these facts, we see: $\cX\dashrightarrow \cY$ is representable if and only if, for some or, equivalently, any smooth $\widetilde{\cW}$ with representable proper birational morphism to $\cW$, the composite morphisms
\[ \widetilde{\cW}\to \Gamma(\cX\dashrightarrow \cY)\to \cX\qquad\text{and}\qquad \widetilde{\cW}\to \Gamma(\cX\dashrightarrow \cY)\to \cY \]
are representable.
\end{rema}

\begin{exam}
\label{exa.orbifoldnotbirat}
In Definition \ref{def.reprat} it is not possible to drop the smoothness hypothesis, or even relax smoothness to normality.
Fix $\mathcal{Y}=\Spec(k)$, and consider an elliptic curve $(E,\infty)$ with a $2$-torsion point $P\in E(k)$. Let translation by $P$ define an action of $G=\Z/2\Z$ on $E$, and let $L=\cO_E(-2[P]-2[\infty])$.
The $G$-action extends to an action on $L$ and its compactification $X'=\bP(L\oplus \cO_E)$, and to an action
on the variety $X$ with elliptic singularity that we obtain by contracting the zero-section of $L$ in $X'$.
The $G$-action on $X'$ is free, so
\[ \cX'=[X'/G] \]
is isomorphic to a smooth projective variety.
The contraction morphism is $G$-equivariant and thus determines a representable proper morphism from $\cX'$ to \[ \cX=[X/G]. \]
The stack $\cX$ is normal and has stabilizer $G$ at its singular point.
Thus we have a morphism $\cX\to \cY$ that is not representable but admits a factorization \eqref{eqn.ffactors}, satisfying the conditions of Proposition \ref{prop:birr}(i), with $\cX'\to \cY$ representable.
\end{exam}

\begin{prop}
\label{prop.reprcompose}
Let $\cX$, $\cY$, and $\cZ$ be stacks, with $\cX$ and $\cY$ smooth. Let $\cX\dashrightarrow \cY$ be a dominant rational map and $\cY\dashrightarrow \cZ$ a rational map.
\begin{itemize}
\item[$\mathrm{(i)}$]
If $\cX\dashrightarrow \cY$ and $\cY\dashrightarrow \cZ$ are representable and at least one of them is proper, then the composite $\cX\dashrightarrow \cZ$ is representable.
\item[$\mathrm{(ii)}$]
If $\cX\dashrightarrow \cY$ is a representable morphism and $\cY\dashrightarrow \cZ$ is representable, then the composite $\cX\dashrightarrow \cZ$ is representable.
\item[$\mathrm{(iii)}$]
If the composite $\cX\dashrightarrow \cZ$ is proper and representable, and $\cY\dashrightarrow \cZ$ is representable, then $\cX\dashrightarrow \cY$ is representable.
\end{itemize}
\end{prop}

\begin{proof}
For (i), we take $(\cU,f)$ and $\cX'$ with morphisms to $\cX$ and $\cY$ as in Definition \ref{def.reprat} and, analogously, $(\cV,g)$ and $\cY'$ with morphisms to $\cY$ and $\cZ$, for $\cY\dashrightarrow \cZ$.
Without loss of generality, we have $f(\cU)\subset \cV$.
Then we have the stack-theoretic closure of $\cU$ in $\cX'\times_{\cY}\cY'$, with composite morphisms to $\cX$ and $\cZ$, exhibiting $\cX\dashrightarrow \cZ$ as representable.

Essentially the same argument, with $\cX'=\cX$, establishes (ii).

We reduce (iii) to the case of morphisms $\cX\to \cZ$ and $\cY\to \cZ$, which is Example \ref{exa.reproverZ}.
The reduction to the case of a morphism $\cX\to \cZ$ is straightforward.
The further reduction to the case of a morphism $\cY\to \cZ$ makes use of the $2$-commutative diagram
\[
\xymatrix{
\cX'\ar[r]\ar[d] & \Gamma(\cX\dashrightarrow \cY') \ar[d] \\
\Gamma(\cX\dashrightarrow \cY) \ar[r] & \cX\times \cY\times \cZ
}
\]
where $\cX'$, taken smooth, is as in Definition \ref{def.reprat} applied to $\cX\dashrightarrow \cY'$, and the bottom morphism is induced by the graph of $\cX\to \cZ$ and the projection to $\cY$.
\end{proof}

\begin{exam}
\label{exa.notYZ}
In the setting of Proposition \ref{prop.reprcompose},
even when all the stacks are proper it is not possible to deduce from the representability of
$\cX\dashrightarrow \cY$ and $\cX\dashrightarrow \cZ$, that $\cY\dashrightarrow \cZ$ is representable.
(Take, e.g.,
$\cX=\cZ=\Spec(k)$ and $\cY=BG$ for a nontrivial finite group $G$.)
Also, statements (i) and (iii) may fail if we drop the respective properness hypothesis, and (ii) may fail if instead of $\cX\dashrightarrow \cY$ it is $\cY\dashrightarrow \cZ$ that is assumed to be a morphism.
(For (i) and (ii), consider
$\cX=\bP(1,2)$, $\cY=\A^1$, and $\cZ=\bP^1$, for (iii), the same with
$\cY$ and $\cZ$ swapped.)
\end{exam}

\begin{prop}
\label{prop.reprproperU}
Let $\cX$ and $\cY$ be stacks, with $\cX$ smooth, and
$\cX\dashrightarrow \cY$ a representable rational map, with representative $(\cU,f)$. Then there exist a smooth stack $\cX'$,
a representable birational morphism $\cX'\to \cX$, and a representable morphism $\cX'\to \cY$, such that $f$ factors up to $2$-isomorphism as \eqref{eqn.ffactors} and the induced morphism $\cX'\to \Gamma(\cX\dashrightarrow \cY)$ is projective.
\end{prop}

\begin{proof}
We let $\cZ=\Gamma(\cX\dashrightarrow \cY)$
and apply Proposition \ref{prop:birs}(ii) to obtain $\cX''$ with $\cX''\to \cZ$ projective, and $\cU\cong \cX''\times_{\cX}\cU$.
Let $\mathsf{X}''\to \cX$, with $\cX''\to \mathsf{X}''$ and $2$-isomorphism from $\cX''\to \cX$ to the composite, be the relative coarse moduli space.
So, $\cX''\to \mathsf{X}''$ is quasi-finite and proper, and $\mathsf{X}''\to \cX$ is representable.
We apply resolution of singularities to $\mathsf{X}''$ to obtain smooth $\widetilde{\mathsf{X}}''$.
We have the bottom two rows of the diagram
\[
\xymatrix@R=28pt@C=24pt{
\widehat{\cX}'\ar[r]\ar[d]&\cX'\ar[d] \ar@/^24pt/@<2pt>[ddrrr]
\\
\cX''\times_{\mathsf{X}''}\widetilde{\mathsf{X}}''\ar[r] \ar[d] &\widetilde{\mathsf{X}}''\ar[d]|(0.56)\hole&&\cZ \ar[dl] \ar[dr]
\\
\cX''\ar[urrr] \ar[r] &
\mathsf{X}''\ar[r] &
\cX \ar@{-->}[rr] && \cY
}
\]
By \cite[Prop.\ 3.4]{AOV}, $\widetilde{\mathsf{X}}''\to \cX$ is the relative coarse moduli space of
\[ \cX''\times_{\mathsf{X}''}\widetilde{\mathsf{X}}''\to \cX. \]

The composite $\widetilde{\mathsf{X}}''\dashrightarrow \cY$ is representable
by Proposition \ref{prop.reprcompose}(ii).
So, there exists a smooth stack $\cX'$ with representable birational morphism to $\widetilde{\mathsf{X}}''$ and representable morphism to $\cY$, such that a suitable representative factors up to $2$-isomorphism as in \eqref{eqn.ffactors} and the induced $\cX'\to \widetilde{\mathsf{X}}''\times \cY$ is proper.
In fact, we claim that $\cX'\to \widetilde{\mathsf{X}}''$
is proper.
For this, letting $\widehat{\cX}'$ denote the
normalization of $\cX''\times_{\mathsf{X}''}\cX'$, it
suffices to show that the composite $\widehat{\cX}'\to \widetilde{\mathsf{X}}''$ is proper.
We apply Proposition \ref{prop.bireasy} to see that we
have a $2$-commutative square
\[
\xymatrix{
\widehat{\cX}'\ar[r]\ar[d] & \cX''\times_{\mathsf{X}''}\widetilde{\mathsf{X}}''\times \cY\ar[d] \\
\cX''\times_{\mathsf{X}''}\widetilde{\mathsf{X}}''\ar[r] & \cX\times \cY\times \cY
}
\]
where the bottom morphism factors through $\cZ$.
The top morphism is proper, since it is a morphism, obtained from $\cX'\to \widetilde{\mathsf{X}}''\times \cY$ via base change, composed with a normalization morphism.
Also, the right-hand morphism is proper.
So the left-hand morphism is proper and remains so when followed by the second projection morphism.
The claim is justified.

By Proposition \ref{prop.reprmaxopen}, the maximal open substack over which $\cX'\to \widetilde{\mathsf{X}}''$ is an isomorphism, has complement everywhere of codimension $\ge 2$.
This open substack is a domain of definition for $\widetilde{\mathsf{X}}''\dashrightarrow \cY$.
So, by Lemma \ref{lem.extendtoY}, $\widetilde{\mathsf{X}}''\dashrightarrow \cY$ is a morphism.
We are done, since $\widetilde{\mathsf{X}}''\to \cX$ restricts to an isomorphism over $\cU$ and $\widetilde{\mathsf{X}}''\to \cZ$ is projective.
Indeed, $\widetilde{\mathsf{X}}''\to \cZ$ is proper, since
$\cX''\times_{\mathsf{X}''}\widetilde{\mathsf{X}}''\to \cZ$ is, and projectivity follows from the fact that by \cite[Lemma 2]{KV}, some power of a relatively ample line bundle for $\cX''\times_{\mathsf{X}''}\widetilde{\mathsf{X}}''$ over $\cZ$ descends to $\widetilde{\mathsf{X}}''$ and is relatively ample on $\widetilde{\mathsf{X}}''$ by \cite[(2.6.2)]{EGAIII}.
\end{proof}

Looking at Definition \ref{def:birs}, we can recognize several notions of birational equivalence of stacks and match these to notions that have appeared in the literature, e.g., in \cite{yasuda}, \cite{HH}, \cite{berghweak}, \cite{ATW}, \cite{Bbar}.
Our choice is compatible with
\cite{Bbar}, see also \cite[Defn.\ 6]{KT-map}.

\begin{defi}
\label{def.biratstacks}
Two smooth stacks $\cX$ and $\cY$ are \emph{birationally equivalent}, written
$$
\cX\sim\cY,
$$
if there exists a representable proper birational map from $\cX$ to $\cY$.
Concretely, this means that there exists a stack $\cX'$ with representable proper birational morphisms $\sigma$ and $\tau$,
\begin{equation}
\begin{split}
\label{eqn.biratstacks}
\xymatrix@C=12pt{
&\cX'\ar[dl]_\sigma \ar[dr]^\tau \\
\cX && \cY
}
\end{split}
\end{equation}
Smooth stacks, representable proper birational maps, and $2$-isomorphisms of representable proper birational maps form a $2$-groupoid ($2$-category in which every morphism is an equivalence and every $2$-morphism is invertible).
We denote by
\[ 
\BIR(\cX,\cY)
\]
the category of representable proper birational maps from $\cX$ to $\cY$ and the case $\cY=\cX$ by
\[
\BIRAUT(\cX),
\]
the birational automorphism $2$-group (monoidal category in which every object is weakly invertible and every morphism is invertible) of $\cX$.
\end{defi}

\begin{rema}
\label{rem.birational}
By Propositions \ref{prop.biratcompose} and \ref{prop.reprcompose}(i), birational equivalence in Definition \ref{def.biratstacks} is an equivalence relation.
In the setting of \eqref{eqn.biratstacks},
we have already observed that without loss of generality $\cX'$ may be taken to be smooth (by resolution of singularities).
By Chow's lemma, $\cX'$ may further be taken to be quasi-projective.
We also record the observation, that when $\cX'$ is a smooth quasi-projective stack and $\cZ$ is a stack, the following conditions on a morphism $g\colon \cX'\to \cZ$ are equivalent:
\begin{itemize}
\item $g$ is representable and proper;
\item $g$ is projective.
\end{itemize}
Generally, projective morphisms are representable and proper.
Now suppose that $g$ is representable and proper.
Letting $X'$ denote the coarse moduli space of $\cX'$, we factor $g$ as
\[ \cX'\to \cX'\times \cZ\to X'\times \cZ\to \cZ. \]
The first morphism, the graph of $g$, is representable and finite.
The second morphism
is quasi-finite and proper.
Their composite is quasi-finite and proper, yet is representable (since $g$ is representable), hence is a representable finite morphism.
The last morphism is quasi-projective.
So $g$ is as well, and we have what we want, since a proper quasi-projective morphism is projective.
In conclusion, Definition \ref{def.biratstacks} admits an equivalent reformulation, where $\cX'$ is a smooth quasi-projective stack and $\sigma$ and $\tau$ are birational projective morphisms.
\end{rema}

\begin{exam}
\label{exa.orbifoldcurves}
If $\mathcal{X}$ and $\mathcal{Y}$ have dimension $1$, then they are birationally equivalent if and only if they are isomorphic.
In particular, the birational type of an orbifold curve is the same as its isomorphism type.
\end{exam}

\begin{exam}
\label{exa.XmodG}
Suppose that $k$ is algebraically closed, and
let $X$ and $Y$ be nonsingular projective varieties with faithful $G$-actions, for a finite group $G$.
If $X$ and $Y$ are $G$-birationally equivalent then the corresponding orbifolds
\[ \mathcal{X}=[X/G]\qquad\text{and}\qquad \mathcal{Y}=[Y/G] \]
are birationally equivalent.
It is possible, however, for $\mathcal{X}$ and $\mathcal{Y}$ to be birationally equivalent, even when $X$ and $Y$ are not $G$-birationally equivalent.
For instance, if $G$ is abelian and we consider two faithful $n$-dimensional representations of $G$ and the corresponding actions on $X=Y=\bP^n$, Reichstein and Youssin \cite[Thm.\ 7.1]{RYinvariant} have shown that $X$ and $Y$ are $G$-birationally equivalent if and only if the
determinants of the representations
\[ \omega_X,\omega_Y\in \Wedge^nG^\vee \]
satisfy
\[ \omega_X=\pm \omega_Y. \]
But there is always an automorphism $\iota\in \Aut(G)$ satisfying
\[ (\Wedge^n\iota)^*\omega_X=\omega_Y. \]
Changing a $G$-action by an automorphism does not change the isomorphism type of the stack quotient.
Consequently, for this class of actions $\mathcal{X}$ and $\mathcal{Y}$ are birationally equivalent even when $X$ and $Y$ are not $G$-birationally equivalent.
\end{exam}

\begin{theo}
\label{thm.weakfactrep}
Let $\cX$ and $\cY$ be smooth stacks and $\cX\dashrightarrow \cY$ a representable proper birational map, represented by $(\cU,f)$ where $f$ maps dense open $\cU\subset \cX$ isomorphically to dense open $\cV\subset \cY$.
Then there exist smooth stacks $\cX_0$, $\dots$, $\cX_m$ for some $m$, such that $\cX_0=\cX$, $\cX_m=\cY$, and for every $i$ there is a morphism
$\cX_{i+1}\to \cX_i$ or $\cX_i\to \cX_{i+1}$, given by blowing up a smooth substack
\[
\cX_{i+1}\cong B\ell_{\cZ_i}\cX_i\to \cX_i,\quad \text{resp.} \quad \cX_i\cong B\ell_{\cW_{i+1}}\cX_{i+1}\to \cX_{i+1}.
\]
We require compatibility with the isomorphism $f$: starting with $\cU_0=\cU$ and continuing to $\cU_m=\cV$, we have
\[
\cU_i\cap \cZ_i=\emptyset, \cU_{i+1}=\cX_{i+1}\times_{\cX_i}\cU_i,\, \text{resp.} \, \cU_{i+1}\cap \cW_{i+1}=\emptyset, \cU_i=\cX_i\times_{\cX_{i+1}}\cU_{i+1},\]
for all $i$, and the composite isomorphism 
\[
\cU=\cU_0\cong\dots\cong \cU_m=\cV
\] 
is $2$-isomorphic to $f$.
Furthermore, there exists $0\le c\le m$ such that for every $i\le c$ the rational map $\cX_i\dashrightarrow \cX$ represented by $(\cU_i,\cU_i\cong\dots\cong \cU)$ is a projective morphism, and for every $i\ge c$ the rational map $\cX_i\dashrightarrow \cY$ represented by $(\cU_i,\cU_i\cong\dots\cong \cV)$ is a projective morphism.
\end{theo}

\begin{proof}
We start by applying Proposition \ref{prop.reprproperU} to $\cX\dashrightarrow \cY$ and $(\cU,f)$ to obtain smooth $\cX'$,
with representable birational proper morphism to $\cX$ restricting to an isomorphism over $\cU$ and
representable birational proper morphism to $\cY$ resricting to an isomorphism over $\cV$.
We claim that such $\cX'$ exists with the additional requirement that $\cX'\to \cX$ and $\cX'\to \cY$ are projective morphisms, given as the blowing up of respective coherent sheaves of ideals.
Indeed, we may apply the construction from the start of the proof of Lemma \ref{lem.birational} to obtain a coherent sheaf of ideals $\cI$ on $\cX$, so that the corresponding blow-up factors up to $2$-isomorphism through $\cX'$ and is identified with the blow-up of $\cI'=\cI\cO_{\cX'}$.
Analogously, we have a coherent sheaf of ideals $\cJ$ on $\cY$ and $\cJ'=\cJ\cO_{\cX'}$.
Through both of these blow-ups, we have factorizations up to $2$-isomorphism of the blow-up associated with $\cI'\cJ'$, by \cite[(5.1.2)(v)]{RG}.
As in the proof of Lemma \ref{lem.birational} we conclude with resolution of singularities and two applications of \cite[Thm.\ 6.1.3]{AT}.
\end{proof}

\begin{rema}
\label{rem.weakfactrep}
As in \cite[Rmk.\ 3.7]{BnG} we may consider snc divisors $\cD\subset \cX$ and $\cE\subset \cY$,
disjoint from $\cU$, respectively $\cV$.
Suppose that the given representable proper birational map $\cX\dashrightarrow \cY$ restricts to a (necessarily representable) proper birational map $\cX\setminus \cD\dashrightarrow \cY\setminus \cE$.
Following the proof of Theorem \ref{thm.weakfactrep}, we have $\cX'\to \cX$ and $\cX'\to \cY$, projective morphisms given by blowing up coherent sheaves of ideals, and by Lemma \ref{lem.openproper} the pre-images of $\cD$ and $\cE$ in $\cX'$ coincide set-theoretically.
Applying embedded resolution of singularities to the common pre-image, we may suppose that this is an snc divisor.
Then by weak factorization with boundary divisors \cite[Thm.\ 6.1.3]{AT}, we obtain a sequence of blow-ups as in Theorem \ref{thm.weakfactrep}, where every stack $\cX_i$ is equipped with a compatible snc divisor $\cD_i$, and every center of blow-up $\cZ_i$ or $\cW_i$ has normal crossing with $\cD_i$.
\end{rema}

\section{Divisoriality and action construction}
\label{sect:action}
One of the ingredients of \cite{Bbar} is a birational modification procedure known as \emph{divisorialification} \cite{bergh}, \cite{berghrydh}, used to bring a given $n$-dimensional orbifold $\mathcal{X}$ into a form, suitable for the determination of its class
\[ [\mathcal{X}]\in \overline{\mathrm{Burn}}_n. \]
Specifically, a sequence of blow-ups in smooth centers is performed, leading to an orbifold that satisfies the following condition.

\begin{defi}
\label{def.divisorial}
We call a DM stack $\mathcal{X}$ \emph{divisorial} if $\mathcal{X}$ admits
a representable morphism to a stack of the form
\[ B\G_m^r \]
for some positive integer $r$.
\end{defi}

In other words, there should exist a finite collection $L_1$, $\dots$, $L_r$ of line bundles on $\mathcal{X}$, which induces faithful representations of the stabilizer groups of $\mathcal{X}$.
(It suffices to consider stabilizers at geometric points of $\mathcal{X}$.)

\begin{rema}
\label{rema.divy}
If $\cY\to \cX$ is a representable morphism and $\cX$ is divisorial then so is $\cY$. 
\end{rema}

\begin{prop}[Bergh-Rydh]
\label{prop.orbifolddivisorialification}
Let $\mathcal{X}$ be an orbifold over $k$.
Then there is a canonical sequence of blow-ups in smooth centers
\[ \mathcal{Y}\to\dots\to\ \mathcal{X} \]
such that $\mathcal{Y}$ is divisorial.
\end{prop}

\begin{proof}
This is proved in \cite{berghrydh},
which supplies an algorithm that takes, as input $\mathcal{X}$ with an snc divisor $\mathcal{D}=\mathcal{D}_1\cup\dots\cup \mathcal{D}_\ell$ and yields a blow-up in a smooth center $\mathcal{X}'\to\mathcal{X}$, with snc divisor consisting of the proper transforms of the $\mathcal{D}_i$ and the exceptional divisor, which decreases a numerical quantity, the maximal value of the \emph{divisorial index}.
When the divisorial index is identically zero, we have a divisorial orbifold.
\end{proof}

\begin{rema}
\label{rem.orbifolddivisorialification}
Let $\cX$ be an orbifold with snc divisor $\mathcal{D}=\mathcal{D}_1\cup\dots\cup \mathcal{D}_\ell$, such that the divisorial index is not identically zero.
Then the center of blow-up $\cZ\subset \cX$, defined by the divisorialification algorithm as the locus of maximal divisorial index, is smooth and meets $\mathcal{D}$ transversally.
Smoothness is established in \cite[Prop.\ 6.5(iii)]{berghrydh}.
Transverse intersection with $\cD$ follows from the characterization \cite[Rmk.\ 3.4]{berghrydh} in terms of smoothness of the associated morphism to $[\A^\ell/\G_m^\ell]$ and the identification of divisorial index with the codimension of stackiness \cite[Defn.\ 5.8]{berghrydh} of $X$ over $[\A^\ell/\G_m^\ell]$, which we see by
arguing as in the proof of \cite[Prop.\ 6.5]{berghrydh}.
\end{rema}

For the next statement we recall the well-known fact
\cite[Prop.\ 2.1]{olssonboundedness} that every smooth separated irreducible DM stack
is a gerbe over an orbifold, i.e., to given $\mathcal{X}$ with geometric generic stabilizer $H$ there is an orbifold $\mathcal{Y}$ with morphism
\begin{equation}
\label{eqn.gerbeoverorbifold}
\mathcal{X}\to \mathcal{Y},
\end{equation}
which \'etale locally over $\mathcal{Y}$ is the projection from a product with $BH$; in particular, the morphism \eqref{eqn.gerbeoverorbifold} is smooth.
As recalled in the proof of \cite[Lemma 7.3]{BnG},
a line bundle on $\mathcal{X}$ descends to $\mathcal{Y}$ if and only if it induces the trivial representation of $H$.

\begin{prop}
\label{prop.divisorial}
Suppose that $\mathcal{X}$ is a smooth separated irreducible DM stack of finite type over $k$, and suppose that $k$ contains all roots of unity.
The following are equivalent:
\begin{itemize}
\item[(1)] $\mathcal{X}$ is divisorial.
\item[(2)] The orbifold $\mathcal{Y}$, associated to $\mathcal{X}$ as in \eqref{eqn.gerbeoverorbifold}, is divisorial, and a dense open substack of $\mathcal{X}$ is isomorphic to $\mathcal{U}\times BH$, for some dense open
$\mathcal{U}\subset \mathcal{Y}$ and finite abelian group $H$.
\item[(3)] There exist a divisorial orbifold $\mathcal{Y}$, line bundles $N_1$, $\dots$, $N_s$ on $\mathcal{Y}$ and positive integers $n_1$, $\dots$, $n_s$, such that
\[ \mathcal{X}\cong \sqrt[n_1]{N_1/\mathcal{Y}}\times_{\mathcal{Y}}\dots \times_{\mathcal{Y}}\sqrt[n_s]{N_s/\mathcal{Y}}. \]
\end{itemize}
\end{prop}

\begin{proof}
To show (1) $\Rightarrow$ (2), we suppose that $\mathcal{X}$ is divisorial, with line bundles
$L_1$, $\dots$, $L_r$ inducing faithful representations of stabilizer groups, and let
$\mathcal{X}\to \mathcal{Y}$ and $H$ be as in \eqref{eqn.gerbeoverorbifold}.
Since divisorial stacks have abelian stabilizer groups, $H$ is abelian.
Since $\mathcal{Y}$ has trivial generic stabilizer, it has a dense open substack $\mathcal{U}$, isomorphic to a scheme.
We write
\begin{equation}
\label{eqn.Hstructure}
H\cong \Z/n_1\Z\times\dots\times \Z/n_s\Z
\end{equation}
according to the structure theorem of finite abelian groups
(that is, $n_i\ge 2$ for all $i$, with $n_i\mid n_{i+1}$)
and fix a primitive $n_i$th root of unity $\lambda_i$ for every $i$.
For each $1\le i\le s$ some tensor combination $M_i$ of
$L_1$, $\dots$, $L_r$ induces the character
\begin{equation}
\label{eqn.Hcharacter}
H\stackrel{\mathrm{pr}_i}\longrightarrow \Z/n_i\Z\stackrel{\lambda_i}\longrightarrow k^\times.
\end{equation}
So $M_i^{n_i}$ induces the trivial character of $H$, hence descends to a line bundle $N_i$ on $\mathcal{Y}$.
Shrinking $\mathcal{U}$, we may suppose that the line bundles $N_1$, $\dots$, $N_s$ are all trivial.
Choosing trivializations, we get for each $i$ a morphism
\[ \mathcal{U}\times_{\mathcal{Y}}\mathcal{X}\to B\Z/n_i\Z \]
from the identification of $B\Z/n_i\Z$ with $B\mu_{n_i}$ (using the chosen root of unity $\lambda_i$ to identify $\Z/n_i\Z$ with $\mu_{n_i}$) and the fiber square
\[
\xymatrix{
B\mu_{n_i}\ar[r]\ar[d] &B\G_m \ar[d]^{({-})^{n_i}}\\
\Spec(k) \ar[r] & B\G_m
}
\]
Combined, these lead to an isomorphism
\[ \mathcal{U}\times_{\mathcal{Y}}\mathcal{X}\to \mathcal{U}\times B\Z/n_1\Z\times\dots\times B\Z/n_s\Z\cong \mathcal{U}\times BH. \]
Furthermore, for each $1\le j\le r$ we find a suitable tensor combination $P_j$ of $M_1$, $\dots$, $M_s$, that induces the same character of $H$ as $L_j$.
Consequently $L_j\otimes P_j^{-1}$ descends to $\mathcal{Y}$, i.e., is isomorphic to the pullback of some line bundle $Q_j$ on $\mathcal{Y}$.
Using that $L_1$, $\dots$, $L_r$ induce faithful representations of the stabilizer groups of $\mathcal{X}$, we see that
\[
N_1, \dots, N_s, Q_1, \dots, Q_r
\]
induce faithful representations of the stabilizer groups of $\mathcal{Y}$.
So $\mathcal{Y}$ is divisorial.

For (2) $\Rightarrow$ (3), we need the observation that a line bundle on an open substack of $\mathcal{X}$ can always be extended to $\mathcal{X}$ \cite[Lemma 4.1]{lev}.
Writing $H$ as in \eqref{eqn.Hstructure}, the
characters \eqref{eqn.Hcharacter} determine line bundles on $\mathcal{U}\times BH$, which are the restrictions of line bundles $M_1$, $\dots$, $M_s$ on $\mathcal{X}$.
Now $M_i^{n_i}$ descends to $\mathcal{Y}$.
An identification of $M_i^{n_i}$ with the pullback to $\mathcal{X}$ of a line bundle $N_i$ on $\mathcal{Y}$, for every $i$, leads to an isomorphism as in statement (3).

For (3) $\Rightarrow$ (1), the pullbacks of line bundles on $\mathcal{Y}$ and the tautological line bundles on $\sqrt[n_i]{N_i/\mathcal{Y}}$ (coming from the identity representation of the second factor $B\G_m$ in \eqref{eqn.rootlb}) exhibit $\mathcal{X}$ as divisorial.
\end{proof}

\begin{coro}
\label{cor.rootdivisorial}
Suppose that $\cX$ is a smooth separated irreducible DM stack, $\cD$ is a smooth divisor on $\cX$, and $r$ is a positive integer.
Then $\cX$ is divisorial if and only if $\sqrt[r]{(\cX,\cD)}$ is divisorial.
\end{coro}

\begin{proof}
There is no loss of generality in supposing that $\cD$ is irreducible.
By \cite[Cor.\ 3.1.2]{cadman}, every line bundle on $\sqrt[r]{(\cX,\cD)}$ is a twist by a power of $\cO(\cG_{\cD})$ of a line bundle, pulled back from $\cX$.
The result follows quickly from this fact.
\end{proof}

Unlike in the case of orbifolds, a general smooth separated Deligne-Mumford stack does not become divisorial after suitable blow-ups.
Indeed, condition (2) of Proposition \ref{prop.divisorial}, if not satisfied already by $\mathcal{X}$, will never hold after performing blow-ups.

\begin{prop}
\label{prop.brdmdiv}
Suppose that $\mathcal{X}$ is a smooth separated irreducible DM stack of finite type over $k$, where $k$ contains all roots of unity.
Then there exists a smooth separated irreducible DM stack $\mathcal{X}'$ that is divisorial and birationally equivalent to $\mathcal{X}$, if and only if $\mathcal{X}$ has a dense open substack isomorphic to $U\times BH$, for some algebraic variety $U$ and finite abelian group $H$.

In this case, suitable $\mathcal{X}'$ may be obtained from the orbifold $\mathcal{Y}$, associated to $\mathcal{X}$, by 
\begin{itemize}
\item 
applying the divisorialification procedure (Propositon \ref{prop.orbifolddivisorialification}) to obtain
\[ \mathcal{Y}'\to \mathcal{Y}, \]
a composition of blow-ups in smooth centers, with $\mathcal{Y}'$ divisorial, and 
\item forming the fiber diagram
\[
\xymatrix{
\mathcal{X}'\ar[r]\ar[d] & \mathcal{X}\ar[d] \\
\mathcal{Y}'\ar[r] & \mathcal{Y}
}
\]
\end{itemize}
\end{prop}

\begin{proof}
If $\mathcal{X}'$ is birationally equivalent to $\mathcal{X}$, then $\mathcal{X}$ and $\mathcal{X}'$ have dense open substacks that are isomorphic.
By condition (2) of Proposition \ref{prop.divisorial}, then, if $\mathcal{X}'$ is divisorial then $\mathcal{X}$ has a dense open substack isomorphic to $\mathcal{U}'\times BH$, for some orbifold $\mathcal{U}'$.
There is an algebraic variety $U$, isomorphic to a dense open in $\mathcal{U}'$, and $\mathcal{X}$ has a dense open substack isomorphic to $U\times BH$.

Now suppose $\mathcal{X}$ has a dense open substack isomorphic to $U\times BH$.
Let $\mathcal{Y}$ be the orbifold, associated to $\mathcal{X}$, and let us apply Proposition \ref{prop.orbifolddivisorialification} to obtain $\mathcal{Y}'\to \mathcal{Y}$.
Since $\mathcal{X}\to \mathcal{Y}$ is smooth, $\mathcal{X}'$ is smooth over $\mathcal{Y}'$ and therefore itself smooth.
The morphism $\mathcal{Y}'\to \mathcal{Y}$ is representable, proper, and birational, hence so is $\mathcal{X}'\to \mathcal{X}$.
Since the rigidification in \cite[Prop.\ 2.1]{olssonboundedness} is canonical, and we have $\mathcal{X}'\to \mathcal{Y}'$ \'etale locally projection from a product with $BH$ and $\mathcal{Y}'$ an orbifold; $\mathcal{Y}'$ is the orbifold, associated to $\mathcal{X}'$.
By Proposition \ref{prop.divisorial}, $\mathcal{X}'$ is divisorial.
\end{proof}

\begin{defi}
\label{defn.linearizable}
Suppose that the base field $k$ contains all roots of unity.
A \emph{linearizable} DM stack is a
smooth separated irreducible DM stack, satisfying the equivalent conditions of Proposition \ref{prop.brdmdiv}.
\end{defi}

\begin{rema}
\label{rem.linearizable}
Suppose that $\cX$ and $\cY$ are smooth separated irreducible DM stacks, and there exists a birational map $\cX\dashrightarrow \cY$.
Then $\cX$ is linearizable if and only if $\cY$ is linearizable.
\end{rema}

\begin{lemm}
\label{lem.linearizable}
Suppose that the base field $k$ contains all roots of unity. Let $\mathcal{X}$ be a linearizable DM stack, $H$ the generic stabilizer of $\mathcal{X}$, and $\chi$ a character of $H$.
Then there exists a line bundle $L$ on $\mathcal{X}$, such that the character of $H$ induced by $L$ is $\chi$.
If $L'$ is another such line bundle on $\mathcal{X}$, then there exists a dense open substack $\mathcal{U}\subset \mathcal{X}$, such that the restrictions of $L$ and $L'$ to $\mathcal{U}$ are isomorphic.
\end{lemm}

\begin{proof}
The character $\chi$ determines a line bundle on a dense open substack of $\mathcal{X}$, isomorphic to $U\times BH$.
As noted in the proof of Proposition \ref{prop.divisorial}, any line bundle on an open substack of $\mathcal{X}$ is the restriction of a line bundle $L$ on $\mathcal{X}$, then $L$ is as desired.

Let $L'$ be another such line bundle.
It suffices to show that $L^{-1}\otimes L'$ is trivial on some dense open substack of $\mathcal{X}$, i.e., it suffices to treat the case $\chi$ is trivial and $L'=\cO_{\mathcal{X}}$.
Let $\mathcal{Y}$ be the orbifold, associated to $\mathcal{X}$.
The triviality of $\chi$ implies that $L$ is isomorphic to the pullback of a line bundle on $\mathcal{Y}$ (cf.\ the morphism \eqref{eqn.gerbeoverorbifold} and subsequent observations).
Since $\mathcal{Y}$ is an orbifold there is a dense open substack, isomorphic to an algebraic variety.
We are done, since a line bundle on an algebraic variety is Zariski locally trivial.
\end{proof}

\begin{lemm}
\label{lem.easybirat}
Let $\mathcal{X}$, $\mathcal{Y}$, and $\mathcal{Y}'$ be smooth separated irreducible DM stacks and let
\[ \mathcal{Y}\to\mathcal{X}\qquad \text{and}\qquad \mathcal{Y'}\to\mathcal{X} \]
be representable proper dominant morphisms, such that for some dense open substack $\cU$ of $\cX$ there exists an isomorphism $\cY\times_{\cX}\cU\to \cY'\times_{\cX}\cU$ over $\cU$.
Then
\[ \mathcal{Y}\sim \mathcal{Y'}. \]
\end{lemm}

\begin{proof}
The induced birational map $\cY\dashrightarrow \cY'$ is proper by Proposition \ref{prop.biratcompose}, and representable by Example \ref{exa.reproverZ}.
\end{proof}

\begin{defi}
\label{def.action}
Suppose $k$ contains all roots of unity.
Let $\mathcal{K}$ be a linearizable DM stack with generic stabilizer $H$ and $a_1$, $\dots$, $a_c\in H^\vee$.
We take $L_i$ as in Lemma \ref{lem.linearizable}, for the character $a_i$.
Now
\[ \mathcal{Y}:=\bP(\cO_{\mathcal{K}}\oplus L_1\oplus\dots\oplus L_c) \]
is a linearizable DM stack with generic stabilizer
\[ \overline{H}:=\ker(a_1)\cap\dots\cap \ker(a_c). \]
We call this the \emph{action construction} \textbf{(A)}.

Suppose, for some $i$, we replace $L_i$ by another line bundle $L'_i$, satisfying the requirement of Lemma \ref{lem.linearizable} for the character $a_i$.
Then there is dense open $\mathcal{U}\subset \mathcal{X}$, such that
$L_i|_{\mathcal{U}}$ is isomorphic to
$L'_i|_{\mathcal{U}}$.
Then, by Lemma \ref{lem.easybirat}, the construction with $L'_i$ in place of $L_i$ yields a birationally equivalent stack $\mathcal{Y}'$.
The \emph{action construction}, performed with $\mathcal{K}$ and $a_1$, $\dots$, $a_c$, yields the stack $\mathcal{Y}$, which is well-defined up to birational equivalence.
\end{defi}

\begin{rema}
\label{rem.action}
The birational equivalence class of $\mathcal{Y}$, in the action construction, only depends on the birational equivalence class of $\mathcal{K}$.
\end{rema}

\begin{rema}
\label{rem.stabilizer2}
Continuing with the theme of Example \ref{ex.stabilizer}, we explain how linearizability impacts the identifications of stabilizer groups of a smooth separated irreducible DM stack $\mathcal{X}$.
The inertia stack (Remark \ref{rem.IX}) $\mathcal{I}_{\mathcal{X}}$
is finite over $\mathcal{X}$
and has the structure of a group object in the category of stacks with representable morphism to $\mathcal{X}$.
Restricting to the maximal open substack $\mathcal{U}$ of $\mathcal{X}$ where $\mathcal{I}_{\mathcal{X}}\to \mathcal{X}$ has constant degree we obtain a finite \'etale group object over $\mathcal{U}$, and $\mathcal{U}$ is a gerbe, \'etale locally over its coarse moduli space $U$ isomorphic to a product with $BH$ for some finite group $H$, the geometric stabilier of $\mathcal{X}$.
\emph{Suppose $H$ is abelian.}
Then the phenomenon reported in Example \ref{ex.stabilizer} at points of $\mathcal{U}$ manifests itself in finite \'etale $G\to U$, with $G\times_U\mathcal{U}\cong \mathcal{I}_{\mathcal{U}}$.
Generally, $G\to U$ is a \emph{twisted form} of $H$, e.g., if $\Z/2\Z$ acts without fixed points on a smooth variety $Y$ with quotient $X$ and we let $\fS_3$ act on $Y$ via $\fS_3\to \Z/2\Z$, then for $\mathcal{U}=\mathcal{X}=[Y/\fS_3]$ we obtain $G=X\sqcup Y$, a twisted form of $H=\Z/3\Z$.
If $G\cong U\times H$ then the gerbe has \emph{trivial band} and there is a \emph{canonical} identification 
of the geometric stabilizer groups at any pair of points of $\mathcal{U}$.
When $k$ has all roots of unity,
we have implications
\[
\text{trivial gerbe (over $U$)}\ \Rightarrow\  \text{$\mathcal{X}$ linearizable}\ \Rightarrow\ \text{trivial band}
\]
and examples of possibilities for gerbe/band, when $H=\Z/m\Z$ ($m\ge 2$):
\begin{itemize}
\item $\bP(m,m)$ is linearizable but as a gerbe over $\bP^1$ is nontrivial;
\item there is a Brauer group obstruction to linearizability, for a gerbe with trivial band (cf.\ \cite[Prop.\ 4]{HKTconic}).
\end{itemize}
\end{rema}

\begin{rema}
\label{rem.compactify}
Quasiprojective smooth DM stacks admit projective compactifications.
Specifically,
by combining \cite[Thm.\ 5.3]{K-Seattle}
with Remark \ref{rema:sing-res},
we see that every quasiprojective smooth DM stack $\cU$ may be embedded as an open substack of a projective smooth DM stack $\cX$, with complement an snc divisor.
If $\cU$ is divisorial, then so is $\cX$
(by the ability, mentioned in the proof of Proposition \ref{prop.divisorial}, to extend line bundles to $\cX$, and the argument of Remark \ref{rem.orbifolddivisorialification}).
\end{rema}

\section{Weak factorization for stacks}
\label{sect:weak}

Section \ref{sect:birationality} concluded with a weak factorization result for representable proper birational maps of stacks.
The case of birational maps that are proper but not necessarily representable is more subtle.
Its treatment requires the \emph{functorial destackification} construction of Bergh \cite{bergh} and Bergh-Rydh \cite{berghrydh} and appears in several formulations: Bergh \cite[Cor.\ 1.5]{bergh} and \cite[Thm.\ 1.2]{berghweak}, Harper \cite[Thm.\ 1.1]{harper}, and Bergh-Rydh \cite[Thm.\ D]{berghrydh}.
A proper birational map $\cX\dashrightarrow \cY$ is factored as a sequence of \emph{stacky blow-ups} and their inverses, where a stacky blow-up is by definition a blow-up of a smooth substack or a root operation along a smooth divisor.

We continue to work with the convention, that all stacks are DM stacks, separated and of finite type over $k$, a field of characteristic $0$.

As remarked in \cite[Sect.\ 1]{berghrydh}, Harper's proof of weak factorization can be streamlined by applying the \emph{relative} form of functorial destackification of Bergh and Rydh \cite[Thm.\ 7.1]{berghrydh}.
We recall this result, in the case that we need.

\begin{prop}[Bergh-Rydh]
\label{prop.relativedestackification}
Let $f\colon \cX'\to \cX$ be a birational morphism of smooth stacks, restricting to an isomorphism $\cU'\cong \cU$ of dense open substacks.
Then there exist smooth stacks $\cX_0$, $\dots$, $\cX_m$ for some $m$, and a sequence of stacky blow-ups
\begin{gather*}
\cX''=\cX_m\to\dots\to \cX_1\to \cX_0=\cX', \\
\cX_{i+1}\cong B\ell_{\cZ_i}\cX_i,\quad\text{resp.}\quad \sqrt[q_i\!]{(\cX_i,\cZ_i)}, \\
\cU_0=\cU',\qquad \cU_{i+1}=\cX_{i+1}\times_{\cX_i}\cU_i, \qquad \cU''=\cU_m, \qquad \cU_i\cap \cZ_i=\emptyset,
\end{gather*}
such that, taking $\cX'''\to \cX$ to be the relative coarse moduli space of $\cX''\to \cX$, the stack $\cX'''$ is smooth, with
\[ \cX''\cong \sqrt[\mathbf{s}]{(\cX''',\cD''')}, \]
iterated root stack of some snc divisor $\cD'''$ on $\cX'''$ and tuple of positive integers $\mathbf{s}$, such that
$\cX''\times_{\cX'''}\cD'''$ is disjoint from $\cU''$.
\end{prop}

We restate the theorem of Bergh, Harper, and Bergh-Rydh, with attention to the properties of the morphisms to $\cX$, respectively $\cY$ from the intermediate stacks in the factorization.

\begin{theo}[Bergh, Harper, Bergh-Rydh]
\label{thm.weakfact}
Let $\cX$ and $\cY$ be smooth stacks and $\cX\dashrightarrow \cY$ a proper birational map, represented by $(\cU,f)$ where $f$ maps dense open $\cU\subset \cX$ isomorphically to dense open $\cV\subset \cY$.
Then there exist smooth stacks $\cX_0$, $\dots$, $\cX_m$ for some $m$, with $\cX_0=\cX$, $\cX_m=\cY$, and for every $i$ a birational proper morphism $\cX_{i+1}\to \cX_i$ or $\cX_i\to \cX_{i+1}$ such that, starting from $\cU_0=\cU$ and proceeding to $\cU_m=\cV$, the morphism $\cX_{i+1}\to \cX_i$ or $\cX_i\to \cX_{i+1}$ identifies dense open substacks $\cU_i\subset \cX_i$ and $\cU_{i+1}\subset \cX_{i+1}$, and the composite isomorphism
\[ \cU=\cU_0\cong\dots\cong \cU_m=\cV \]
is $2$-isomorphic to $f$.
Furthermore, there are
\[ 0<b<c<d<m \]
such that:
\begin{itemize}
\item The rational maps $\cX_c\dashrightarrow \cX$ and $\cX_c\dashrightarrow \cY$, represented by $(\cU_c,\cU_c\cong\dots\cong \cU)$, respectively $(\cU_c,\cU_c\cong\dots\cong \cV)$, are morphisms, and the induced morphism $\cX_c\to \Gamma(\cX\dashrightarrow \cY)$
is projective.
\item For all $b\le i<c$ we have $\cX_i\to \cX_{i+1}$, a stacky blow-up
\[ \cX_i\cong B\ell_{\cW_{i+1}}\cX_{i+1}\to \cX_{i+1}\quad\text{or}\quad \sqrt[p_i\!]{(\cX_{i+1},\cW_{i+1})}\to \cX_{i+1}. \]
\item For all $c\le i<d$ we have $\cX_{i+1}\to \cX_i$, a stacky blow-up
\[
\cX_{i+1}\cong B\ell_{\cZ_i}\cX_i\to \cX_i\quad\text{or}\quad \sqrt[q_i\!]{(\cX_i,\cZ_i)}\to \cX_i.
\]
\item We have $\cX_b\to \cX_{b-1}$, appearing in an expression of $\cX_{b-1}\to \cX$ as relative coarse moduli space of $\cX_b\to \cX$, and
\[
\cX_b\cong \sqrt[\mathbf{r}]{(\cX_{b-1},\cD)}
\]
for some snc divisor $\cD$ and tuple of positive integers $\mathbf{r}$.
\item We have $\cX_d\to \cX_{d+1}$, appearing in an expression of $\cX_{d+1}\to \cY$ as relative coarse moduli space of $\cX_d\to \cY$, and
\[
\cX_d\cong
\sqrt[\mathbf{s}]{(\cX_{d+1},\cE)}
\]
for some snc divisor $\cE$ and tuple of positive integers $\mathbf{s}$.
\item For every $0\le i<b-1$ and every $d+1\le i<m$ the birational morphism between $\cX_i$ and $\cX_{i+1}$ is a blow-up of a smooth substack
\[
\cX_{i+1}\cong B\ell_{\cZ_i}\cX_i\to \cX_i\quad \text{or} \quad \cX_i\cong B\ell_{\cW_{i+1}}\cX_{i+1}\to \cX_{i+1}.
\]
\end{itemize}
We define $\cZ_{b-1}=\cD$ and $\cW_{d+1}=\cE$ and require $\cU_i\cap \cZ_i=\emptyset$, respectively $\cU_{i+1}\cap \cW_{i+1}=\emptyset$ for every $i$.
Additionally, for every $0\le i\le b-1$, resp.\ $d+1\le i\le m$ the rational map $\cX_i\dashrightarrow \cX$, resp.\ $\cX_i\dashrightarrow \cY$ represented by $(\cU_i,\cU_i\cong\dots\cong \cU)$, resp.\ $(\cU_i,\cU_i\cong\dots\cong \cV)$, is a representable proper morphism.
\end{theo}

\begin{proof}
We start by applying Proposition~\ref{prop:birs}(ii) to $\cX\dashrightarrow \cY$.
This gives us a reduced stack $\cX'$, which by resolution of singularities we may suppose to be smooth, 
with birational morphisms to $\cX$ and $\cY$, restricting to isomorphisms $\cU'\to \cU$ and $\cU'\to \cV$, where $\cU'=\cX'\times_{\cX}\cU$.
The morphism $f$ factors up to $2$-isomorphism as
\[ \cU\cong \cU'\to \cY. \]
The induced morphism $\cX'\to \Gamma(\cX\dashrightarrow \cY)$ is projective, and the morphisms
$\Gamma(\cX\dashrightarrow \cY)\to \cX$ and $\Gamma(\cX\dashrightarrow \cY)\to \cY$ are proper, as are the composite morphisms $\cX'\to \cX$ and $\cX'\to \cY$.

We apply Proposition \ref{prop.relativedestackification} to $\cX'\to \cX$ and to $\cX'\to \cY$.
This yields sequences of stacky blowups of $\cX'$, terminating in $\cX''$, respectively $\cY''$, and relative coarse moduli spaces
\[ \cX''\to \cX'''\to \cX\qquad\text{resp.}\qquad \cY''\to \cY'''\to \cY. \]
Now we apply Lemma \ref{lem.birational} to the representable proper birational morphisms $\cX'''\to \cX$ and $\cY'''\to \cY$.
We relabel $\cX'''$ as $\cX_{b-1}$ and $\cX''$ as $\cX_b$, where $b-1$ is the number of blow-ups from the application of Lemma \ref{lem.birational} to $\cX'''\to \cX$, and similarly relabel the stacky blow-ups from the application of Proposition \ref{prop.relativedestackification} to $\cX'\to \cX$ as
\[ \cX_b\to \dots\to \cX_c=\cX' \]
and in an analogous fashion obtain
\[ \cX_d\to \dots\to \cX_c=\cX', \]
where $\cY''$ is now denoted by $\cX_d$.
Finally, we relabel the blow-ups from the remaining application of Lemma \ref{lem.birational} to obtain
$\cX_{d+1}$, $\dots$, $\cX_m$, with
$\cX_{d+1}=\cY'''$ and $\cX_m=\cY$.
\end{proof}

\begin{prop}
\label{prop.weakfactproperties}
Let $\cX$ and $\cY$ be smooth stacks, $\cX\dashrightarrow \cY$ a proper birational map represented by $(\cU,f)$ where $f$ maps dense open $\cU\subset \cX$ isomorphically to dense open $\cV\subset \cY$, and let smooth stacks $\cX_0$, $\dots$, $\cX_m$, and birational morphism $\cX_{i+1}\to \cX_i$ or $\cX_i\to \cX_{i+1}$ for every $i$ be as in Theorem \ref{thm.weakfact}.
Then:
\begin{itemize}
\item[$\mathrm{(i)}$] If $\cX$ and $\cY$ are divisorial, then $\cX_i$ is divisorial for every $i$.
\item[$\mathrm{(ii)}$] If $\cX$ and $\cY$ are quasi-projective (resp.\ projective), then $\cX_i$ is quasi-projective (resp.\ projective) for every $i$.
\end{itemize}
\end{prop}

\begin{proof}
Let $0<b<c<d<m$ be as in Theorem \ref{thm.weakfact}.
For (i), if $\cX$ and $\cY$ are divisorial, then so is $\cX\times \cY$, hence also $\cX_c$, by Remark \ref{rema.divy}.
By Remark \ref{rema.divy} and Corollary \ref{cor.rootdivisorial}, so are the stacks $\cX_i$ for all $b\le i\le d$.
We have projective morphisms to $\cX$ or $\cY$ from the remaining stacks $\cX_i$, so these are divisorial as well by Remark \ref{rema.divy}.

We obtain (ii) from Lemma \ref{lem.projective} and
Proposition \ref{prop.rootcoarse}.
\end{proof}

\begin{rema}
\label{rem.weakfactproperties}
In the setting of Theorem \ref{thm.weakfact}, suppose that the complement of $\cU$ is an snc divisor in $\cX$, and the complement of $\cV$ is an snc divisor in $\cY$.
Then the outcome of Theorem \ref{thm.weakfact} maybe be strengthened, by the requirement, that the complement of $\cU_i$ is an snc divisor $\cD_i\subset \cX_i$ for every $i$, and every $\cZ_i$ and $\cW_i$ that appears as a center of a blow-up $\cX_{i+1}\cong B\ell_{\cZ_i}\cX_i$ respectively $\cX_{i-1}\cong B\ell_{\cW_i}\cX_i$ has normal crossing with $\cD_i$.
In the proof, we apply embedded resolution of singularities to achieve that $\cX'\setminus \cU'$ is an snc divisor.
We apply the evident variant of Proposition \ref{prop.relativedestackification},
where $\cX\setminus \cU$ and $\cX'\setminus \cU'$ are snc divisors and in the conclusion every $\cX_i\setminus \cU_i$ is an snc divisor, with which $\cZ_i$ has normal crossing.
We invoke a similarly evident variant of Lemma \ref{lem.birational}, with boundary divisors.
\end{rema}

\section{Refined Burnside groups for orbifolds}
\label{sect:g}
In this section, we 
assume that the base field $k$ (of characteristic $0$) contains all roots of unity.
We define 
$$
\cBurn_n\qquad\text{and}\qquad \ocBurn_n,
$$ 
abelian groups receiving birational invariants of $n$-dimensional orbifolds; the second group is a quotient of the first.
The definitions are in terms of generators and relations. 
%Further relations are imposed to obtain $\ocBurn_n$ from $\cBurn_n$, making $\ocBurn_n$ a quotient of $\cBurn_n$.
A projective, respectively quasi-projective orbifold of dimension $n$ will determine a class in $\cBurn_n$, respectively $\ocBurn_n$.

The following definition is in parallel with the definition of 
$$
\Burn_n(G),
$$
the equivariant Burnside group of $n$-dimensional varieties with generically free action of a finite group $G$, in \cite{BnG}.

\begin{defi}
\label{defn.cBurn}
The abelian groups $\cBurn_n$
and $\ocBurn_n$ are generated by 
pairs
\[
(\mathcal{K},\alpha),
\]
where
\begin{itemize}
\item 
$\mathcal{K}$ is a projective linearizable DM stack of dimension $d\le n$, and 
\item 
$\alpha=(a_1,\ldots, a_{n-d})$ 
is a sequence of length $n-d$ of nonzero characters of the generic stabilizer $H$ of $\mathcal{K}$, generating $H^\vee$.
\end{itemize}
The group $\cBurn_n$ is defined by the following relations.

\medskip
\noindent
\textbf{(O)} (re)ordering: For all permutations $\sigma\in \mathfrak S_{n-d}$ we have
$$
(\mathcal K,\alpha) = (\mathcal K, \alpha'), \quad 
\alpha':=(a_{\sigma(1)}, \ldots, a_{\sigma(n-d)}).
$$

\noindent
\textbf{(I)} (birational) isomorphisms: if $\mathcal{K}'$ is another projective linearizable DM stack of dimension $d$, with corresponding characters $\alpha'=(a'_1,\dots,a'_{n-d})$, and there is a birational equivalence
\[ \mathcal{K}\sim \mathcal{K}' \]
transforming $\alpha$ to $\alpha'$, then
$$
(\mathcal{K}, \alpha) =  (\mathcal{K}', \alpha').
$$

\noindent
\textbf{(V)} vanishing: For all $\mathcal K$ and 
$\alpha=(a_1,a_2,\ldots, a_{n-d})$ with 
$$
a_1+a_2=0
$$
we have
$$
(\mathcal K,\alpha) = 0.
$$

\noindent
\textbf{(B)} blowup: suppose $n-d\ge 2$, set $a:=a_1-a_2$, then
\[ (\mathcal K,\alpha)=\Theta_1+\Theta_2, \]
where
\[
\Theta_1=\begin{cases} 0,& \text{if $a=0$},\\
(\mathcal K,\alpha_1)+(\mathcal K,\alpha_2),&\text{if $a\ne 0$},\end{cases}
\]
with
\[ \alpha_1:=(a_1,a_2-a_1,a_3,\dots,a_{n-d}),\quad
\alpha_2:=(a_2,a_1-a_2,a_3,\dots,a_{n-d}),
\]
and
\[
\Theta_2=\begin{cases}
0,&\text{if $a_i\in \langle a\rangle$ for some $i$},\\
(\mathcal{Y},(b_2,\dots,b_{n-d})),&\text{otherwise},
\end{cases}
\]
with $\mathcal{Y}$ obtained by construction \textbf{(A)} from
$\mathcal{K}$ and the character $a$; the generic stabilizer of
$\mathcal{Y}$ is naturally a subgroup of the generic stabilizer of $\mathcal{K}$,
and the character $b_i$ is obtained by restricting $a_i$, for $i=2$, $\dots$, $n-d$.

The group $\ocBurn_n$ is defined by the same relations, except that \textbf{(I)} is replaced by

\medskip
\noindent
\textbf{(}$\overline{\mathbf{I}}$\textbf{)}
$(\cK,\alpha)=(\cK',\alpha')$
if there is a birational map $\cK\dashrightarrow \cK'$ transforming $\alpha$ to $\alpha'$.
\end{defi}

%\begin{rema}
%\label{rem.notanyauto}
%Notice that relation \textbf{(I)} identifies $(\mathcal{K},\alpha)$ with $(\mathcal{K},\alpha')$ only when $\alpha'$ is obtained from $\alpha$ by applying an automorphism of the generic stabilizer group of $\mathcal{K}$ that \emph{arises from a birational automorphism} of $\mathcal{K}$.
%In an example such as $\bP(2,2)\times B(\Z/2)$, not every automorphism of the generic stabilizer arises from a birational automorphism.
%(Every birational automorphism is an automorphism, here, by Example \ref{exa.orbifoldcurves}).
%\end{rema}

\begin{rema}
\label{rem.Ibar}
Relation \textbf{(}$\overline{\mathbf{I}}$\textbf{)} differs from \textbf{(I)}, in that the requirement $\cK\sim \cK'$ of \textbf{(I)} is weakened to the requirement, in \textbf{(}$\overline{\mathbf{I}}$\textbf{)}, of just a birational map, i.e., an isomorphism of dense open substacks.
Since $\cK$ always has a dense open substack isomorphic to $U\times BH$ for some algebraic variety $U$, there are birational automorphisms of $\cK$ realizing every automorphism of $H$.
Consequently, in $\ocBurn_n$ we have $(\cK,\alpha)=(\cK,\alpha')$ when $\alpha'$ is obtained from $\alpha$ by applying an automorphism of $H$.
\end{rema}

\begin{defi}
\label{def.classX}
Let $\cX$ be a projective orbifold of dimension $n$, that is divisorial.
For a finite abelian group $H$,
we call
a \emph{stabilizer-$H$ component} of
$\cX$ the closure of a component of the nonsingular locally closed substack of $\cX$, defined by the
condition of having geometric stabilizer group isomorphic to $H$; this is a smooth substack of $\cX$.
A \emph{stabilizer component} of $\cX$ is a closed substack, that is a stabilizer-$H$ component for some $H$.
The class
\[ 
[\mathcal{X}]\in \cBurn_n
\]
is defined to be the sum
\begin{equation}
\label{eqn.sum}
[\mathcal{X}]:=\sum_{\mathcal{K}\subset \mathcal{X}} (\mathcal{K},\beta_{\mathcal{K}}(\mathcal{X}))
\end{equation}
over stabilizer components of $\cX$,
where the normal bundle $\cN_{\mathcal{K}/\mathcal{X}}$ determines a representation and hence a sequence of characters (up to order) of the generic stabilizer of $\mathcal{K}$, denoted above by $\beta_{\mathcal{K}}(\mathcal{X})$.
\end{defi}

\begin{rema}
\label{rem.classX}
Suppose $\cX$ is a smooth separated irreducible DM stack of finite type over $k$, that is divisorial.
Then by Proposition \ref{prop.divisorial} we know that the associated orbifold $\cY$ is divisorial.
We call a stabilizer component of $\cX$ the inverse image by $\cX\to \cY$ of any stabilizer component of $\cY$.
If $H_0$ denotes the generic stabilizer of $\cX$, and $\cK$ is a stabilizer component of $\cX$, inverse image of stabilizer component $\cL\subset \cY$, then $H_0$ is a subgroup of the generic stabilizer $H$ of $\cK$, and $H/H_0$ is the generic stabilizer of $\cL$.
We identify the characters in $\beta_{\cL}(\cY)$ with characters of $H$, trivial on $H_0$, which we denote by $\beta_{\cK}(\cX)$.
\end{rema}

\begin{prop}
\label{prop.goodrelations}
In $\cBurn_n$ we have the following relations:
\begin{itemize}
\item[$\mathrm{(i)}$] If $\alpha=(a_1,\dots,a_{n-d})$ with $a_1+\dots+a_j=0$ for some $2\le j\le n-d$, then
\[ (\mathcal{K},\alpha)=0. \]
\item[$\mathrm{(ii)}$]
For any $2\le j\le n-d$,
\[ (\mathcal{K},\alpha)=\sum_{\substack{\emptyset\ne I\subseteq\{1,\dots, j\}\\\text{satisfying $(*)$}}}(\mathcal{Y}_I,\beta_I) \]
where a nonempty subset $I\subseteq \{1,\dots,j\}$,
$i_0\in I$,
determines the subgroup $H_I\subseteq H$,
intersection of the kernels of $a_i-a_{i_0}$ for all $i\in I$, we apply construction \emph{\textbf{(A)}} to the characters
\[ (a_i-a_{i_0})_{i\in I\setminus \{i_0\}} \]
to obtain the linearizable DM stack $\mathcal{Y}_I$ with generic stabilizer $H_I$, we define $\alpha_I$ to consist of the characters
\[ a_{i_0},\qquad (a_i-a_{i_0})_{i\in \{1,\dots,j\}\setminus I} \qquad \text{and}\qquad (a_{j+1},\dots,a_{n-d}), \]
we define the condition
\[ (*):\ \text{the character $a$ restricts nontrivially to $H_I$, for all $a\in \alpha_I$}, \]
and we define $\beta_I$ to consist of the restrictions to $H_I$ of the characters in $\alpha_I$.
\end{itemize}
\end{prop}

\begin{proof}
We prove $\mathrm{(i)}$ by induction on $j$.
The base case $j=2$ is relation \textbf{(V)}.
For the inductive step, we apply relation \textbf{(B)} to $(\mathcal{K},\alpha')$, where
\[ \alpha'=(a_1,a_1+a_2,a_3,\dots,a_{n-d}). \]
By relation \textbf{(O)} and the induction hypothesis we have $(\mathcal{K},\alpha')=0$.
The right-hand side of \textbf{(B)} consists of terms that vanish for the same reason, along with $(\mathcal{K},\alpha)$.
So
$(\mathcal{K},\alpha)=0$.

We also prove $\mathrm{(ii)}$ by induction on $j$.
The base case $j=2$ is relation \textbf{(B)}.
For the inductive step, we apply the induction hypothesis and for each resulting $(\mathcal{Y}_I,\beta_I)$ we apply relation \textbf{(B)} to the restrictions of the characters $a_{i_0}$ and $a_{j+1}$.
The rest of the argument is combinatorial and follows exactly the proof of \cite[Prop.\ 4.7]{BnG}, on which this statement is based.
\end{proof}

\section{Birational invariants of orbifolds}
\label{sect:burninv}
We continue to work over a base field $k$ of characteristic $0$ containing all roots of unity.
The goal of this section is to exhibit birational invariants of $n$-dimensional orbifolds, with values in $\cBurn_n$ and $\ocBurn_n$.

\begin{theo}
\label{theo.main-th}
Let $\cX$ and $\cX'$ be projective orbifolds of dimension $n$, that are divisorial.
If $\cX$ and $\cX'$ are birationally equivalent, then
$$
[\cX]=[\cX']
$$
in $\cBurn_n$.
\end{theo}

\begin{proof} 
By Theorem~\ref{thm.weakfactrep}, it suffices to consider a blowup $\pi\colon \cX' \to \cX$
of divisorial projective orbifolds, with center in a smooth substack $\cW$.
We split \eqref{eqn.sum} into two sums, according to whether $\cK$ is contained in $\cW$:
\[ [\cX]=\sum_{\cK\not\subset \cW}(\mathcal{K},\beta_{\mathcal{K}}(\mathcal{X}))
+\sum_{\cK\subset \cW}(\mathcal{K},\beta_{\mathcal{K}}(\mathcal{X})). \]
Note first that $\cK\cap \cW$ is smooth; this follows from the analysis of local stack geometry in \cite{alperkresch}.

The analogous sum for $\cX'$ is split analogously, according to containment in the exceptional divisor $\cE$:
\[ [\cX']=\sum_{\cK'\not\subset \cE}(\cK',\beta_{\cK'}(\cX'))+\sum_{\cK'\subset \cE}(\cK',\beta_{\cK'}(\cX')). \]
The respective first sums are equal.
For any $\cK'$ in the second sum for $\cX'$ the image $\pi(\cK')$ is contained in a unique $\cK$, appearing in the formula for $[\cX]$, with the same generic stabilizer as $\pi(\cK')$.
We claim that $\pi(\cK')$ is an entire connected component $\cV$ of $\cK\cap \cW$.
This follows by an analysis of $
\cE\times_{\cX}\cV\cong \bP(\cN_{\cW/\cX}|_{\cV})$, which reveals also that when $\mathcal{K}\not\subset \mathcal{W}$ there are always two characters summing to $0$ in $\beta_{\cK'}(\cX')$, hence
$(\cK',\beta_{\cK'}(\cX'))=0$ by \textbf{(V)}.
When $\mathcal{K}\subset \mathcal{W}$ the contribution to $[\cX']$ from the $\cK'\subset \cE$ with $\pi(\cK')=\cK$ is precisely the right-hand side of the relation from Proposition \ref{prop.goodrelations}(ii), where $j$ is the codimension of $\cW$,
consequently $[\cX]=[\cX']$.
\end{proof}

In the next definition we allow ourselves to write a symbol $(\cK,\alpha)\in \ocBurn_n$ when $\cK$ is a quasiprojective linearizable DM stack.
By the results about projective compactifications stated in Section \ref{sect:action}, such $\cK$ may be embedded as an open substack in a projective linearizable DM stack $\overline{\cK}$.
We introduce
\[ (\cK,\alpha):=(\overline{\cK},\alpha)\in \ocBurn_n. \]
This, by relation
\textbf{(}$\overline{\mathbf{I}}$\textbf{)}, is independent of the choice of $\overline{\cK}$.

Since $\ocBurn_n$ is defined by a larger set of relations than $\cBurn_n$, there is a canonical homomorphism
\[ \cBurn_n\to \ocBurn_n. \]
In any formula involving classes from both $\cBurn_n$ and $\ocBurn_n$ there is the implicit use of this homomorphism.

\begin{defi}
\label{def.naiveclassU}
Let $\cU$ be a quasi-projective orbifold of dimension $n$, that is divisorial.
The corresponding naive class in $\ocBurn_n$ is
\[
[\cU]^{\mathrm{naive}}:=\sum_{\cK\subset \cU} (\cK,\beta_{\cK}(\cU)),
\]
where the sum is as in \eqref{eqn.sum}, but for $\cU$.
If $\cU$ is the complement, in a divisorial projective orbifold $\cX$ of an snc divisor $\cD$,
\[ \cD=\bigcup_{i\in \cI} \cD_i, \]
then we define the class in $\ocBurn_n$ to be
\[
[\cU]:=[\cX]+\sum_{\emptyset\ne I\subseteq \cI}(-1)^{|I|}[\cN_{\cD_I/\cX}]^{\mathrm{naive}},
\]
where $\cD_I$ denotes the intersection of the $\cD_i$ with $i\in I$.
\end{defi}

Notice, by Remark \ref{rema.divy}, $\cN_{\cD_I/\cX}$ is divisorial, for every $I$.
By the same argument as in the proof of Theorem \ref{theo.main-th}, $[\cU]^{\mathrm{naive}}$ is a birational invariant.
It also remains unchanged upon passage to suitable open substacks.

\begin{lemm}
\label{lem.naiveU}
Let $\cU$ and $\cU'$ be a quasi-projective orbifolds of dimension $n$. We suppose that $\cU'$ is divisorial.
\begin{itemize}
\item[$\mathrm{(i)}$] If $\cU$ is an open substack of $\cU'$,
such that $\cU\cap \cK\ne \emptyset$ for every stabilizer component $\cK$ of $\cU'$, then
in $\ocBurn_n$ we have
\[ [\cU]^{\mathrm{naive}}=[\cU']^{\mathrm{naive}}. \]
\item[$\mathrm{(ii)}$] 
Let $\pi\colon \cU\to \cZ$ and $\pi'\colon \cU'\to \cZ$ be representable smooth morphisms to a smooth quasi-projective DM stack $\cZ$.
Suppose that $\cZ$ admits a covering by open substacks $\cW\subset \cZ$ with the property, that there exist an open immersion
\[
j\colon \pi^{-1}(\cW)\to \pi'^{-1}(\cW)
\]
and a $2$-isomorphism $\pi\Rightarrow \pi'\circ j$, such that the condition in $\mathrm{(i)}$ is satisfied for the open substack
$j(\pi^{-1}(\cW))$ of $\pi'^{-1}(\cW)$.
Then $\cU$ is divisorial, and in $\ocBurn_n$ we have
\[ [\cU]^{\mathrm{naive}}=[\cU']^{\mathrm{naive}}. \]
\end{itemize}
\end{lemm}

\begin{proof}
Generally, the stabilizer components of $\cU$ are the nonempty intersections with $\cU$ of the stabilizer components of $\cU'$.
Under the assumption in (i),
all of the intersections are nonempty.
The equality is clear, then,
from the definition of naive class in $\ocBurn_n$.

For (ii), every substack $\pi^{-1}(\cW)$ of $\cU$ is divisorial, and the line bundles that exhibit this extend to $\cU$ (cf.\ the proof of Proposition \ref{prop.divisorial}), so $\cU$ is divisorial.
For the rest of the assertion, we argue as in the proof of \cite[Lemma 5.10]{BnG}.
\end{proof}

In the rest of this section, we will prove that $[\cU]\in \ocBurn_n$ is well-defined and is a birational invariant.

\begin{lemm}
\label{lem.puncturedsumoflb}
Let $\cU$ be a divisorial orbifold,
$H_0$ the generic stabilizer of $\cU$, and $\cL_1$, $\dots$, $\cL_t$ line bundles on $\cU$ whose corresponding characters of $H_0$ generate $H_0^\vee$.
We denote by
\[ \cT\subset \cL_1\oplus\dots\oplus \cL_t, \]
with projection
\[ \pi\colon \cT\to \cU, \]
the complement of the union of $\cL_1\oplus\dots\oplus 0\oplus\dots\oplus \cL_t$ with $0$ as $i$th summand, for $i=1$, $\dots$, $t$.
Then every stabilizer component of $\cT$ is of the form $\pi^{-1}(\cK)$ for some stabilizer component $\cK\subset \cU$.
If $H$ denotes the generic stabilizer of $\cK$ and $\alpha_i$ the character induced by $\cL_i$, then $\pi^{-1}(\cK)$ is a stabilizer component of $\cT$ if and only if
\[ \beta_{\cK}(\cU)\cap \langle \alpha_1,\dots,\alpha_t\rangle=\emptyset. \]
In this case $\pi^{-1}(\cK)$ has generic stabilizer
$H/\langle \alpha_1,\dots,\alpha_t\rangle$, and we have
$\beta_{\pi^{-1}(\cK)}(\cT)$ equal to
$\overline{\beta_{\cK}(\cU)}$,
the images in $H/\langle \alpha_1,\dots,\alpha_t\rangle$ of $\beta_{\cK}(\cU)$.
\end{lemm}

\begin{proof}
It suffices to treat the case $t=1$.
The points in any fiber of $\pi$ all have the same stabilizer group, hence every stabilizer component of $\cT$ is of the claimed form.
For any stabilizer component $\cK$, with generic stabilizer $H$, the generic stabilizer of $\pi^{-1}(\cK)$ is as claimed.
Also as claimed is the description of the characters of the normal bundle at the generic point, whose nonvanishing is necessary and sufficient for $\pi^{-1}(\cK)$ to be a stabilizer component of $\cT$.
\end{proof}

\begin{defi}
\label{def.puncturednormal}
Let $\cD=\bigcup_{i\in \cI}\cD_i$ be an snc divisor on a
projective orbifold $\cX$.
The \emph{snc stratum}
\[
\cD_I^\circ
\]
indexed by $I\subseteq \cI$ is
the complement in $\cD_I$ of the union of $\cD_j$ for $j\in \cI\setminus I$.
The \emph{punctured normal bundle}
\[ \cN^\circ_{\cD_I/\cX}, \]
with projection morphism
\[
\pi^\circ_I\colon \cN^\circ_{\cD_I/\cX}\to \cD_I^\circ,
\]
is the complement in
\[ \cN_{\cD_I^\circ/\cX}\cong\bigoplus_{i\in I}\cN_{\cD_i/\cX}|_{\cD_I^\circ} \]
of the union of the zero-sections of the summands.
\end{defi}

\begin{lemm}
\label{lem.punctured}
With the notation of Definitions \ref{def.naiveclassU} and \ref{def.puncturednormal} we have
\[
[\cU]=[\cU]^{\mathrm{naive}}+\sum_{\emptyset\ne I\subseteq \cI}(-1)^{|I|} \,[\cN^\circ_{\cD_I/\cX}]^{\mathrm{naive}}
\]
in $\ocBurn_n$.
\end{lemm}

\begin{proof}
The proof is combinatorial and follows the same steps as the proof of \cite[Lemma 5.7]{BnG}.
\end{proof}

\begin{rema}
\label{rem.classUroot}
With the notation of Definition \ref{def.naiveclassU}, if for $i\in \cI$ and a positive integer $r$ we perform the $r$th root operation on $\cD_i$ we obtain
\[ \cU\cong \sqrt[r]{(\cX,\cD_i)}\setminus \cD', \]
where $\cD'$ is an snc divisor, union of $\cG_{\cD_i}$ and the pre-images of $\cD_{i'}$ for $i'\ne i$.
It follows from Lemma \ref{lem.normalbundlerot} that the formula from Lemma \ref{lem.punctured} remains unchanged when we replace $\cX$ by $\sqrt[r]{(\cX,\cD_i)}$.
\end{rema}

\begin{prop}
\label{prop.isomorphisminvariant}
With the notation of Definition \ref{def.naiveclassU}, suppose that $\cU$ is also the complement, in a divisorial projective orbifold $\cX'$ of an snc divisor $\cD'$,
\[ \cD'=\bigcup_{i\in \cI'}\cD'_i. \]
Then
\[
[\cX]+\sum_{\emptyset\ne I\subseteq \cI}(-1)^{|I|}[\cN_{\cD_I/\cX}]^{\mathrm{naive}}=
[\cX']+\sum_{\emptyset\ne I\subseteq \cI'}(-1)^{|I|}[\cN_{\cD'_I/\cX'}]^{\mathrm{naive}}.
\]
in $\ocBurn_n$.
\end{prop}

\begin{proof}
Applying Proposition \ref{prop.weakfactproperties} (see also Remark \ref{rem.weakfactproperties}),
we are reduced to treating the case that $\cX'$ is obtained from $\cX$ by one of the following operations:
\begin{itemize}
\item[(i)] Blowing up a smooth substack $\cZ$,
contained in and having normal crossing with $\cD$. 
\item[(ii)] Root operation along a component of $\cD$.
\end{itemize}
The result in case (ii) is taken care of by Remark \ref{rem.classUroot}.

We treat case (i).
The stack $\cZ$ is projective, and without loss of generality may be assumed irreducible, so the coarse moduli space is a projective variety:
\[ \pi\colon \cZ\to Z. \]
We remark that any short exact sequence of vector bundles on $\cZ$ admits a splitting after
restricting to $\pi^{-1}(V)$ for any affine open $V\subset Z$.
This follows from the Leray spectral sequence, once
we know that $R^i\pi_*\cF=0$ for all $i>0$ and quasi-coherent sheaves $\cF$.
For this, it is enough to know the vanishing of $\rH^i([\Spec(A)/G],\cF)$ for a quotient stack, where $G$ is finite and $A$ is a $k$-algebra.
We know the latter by a standard spectral sequence (see \cite[Thm.\ III.2.20]{milne}) and the vanishing of group cohomology of $k$-vector spaces.

Let $\cI'=\{i\in \cI\,|\,\cZ\subset \cD_i\}$,
with complement $\cI''$.
For $I\subseteq \cI$ we denote the respective intersections with $I$ by $I'$ and $I''$.
We define
\[ \cZ_I=\cD^\circ_{\cI'\cup I}\cap \cZ. \]

As in the proof of \cite[Prop.\ 5.8]{BnG}, by applying Lemma~\ref{lem.punctured} and rearranging terms, the desired equality is equivalent to
\begin{align}
\begin{split}
\label{eqn.sufficestoverify}
\sum_{\emptyset\ne I\subseteq \cI}
& (-1)^{|I|}
\big([\cN^\circ_{\cD_I/\cX}]^{\mathrm{naive}}
-[\cN^\circ_{\cD'_I/\cX'}]^{\mathrm{naive}}
\big)
= \\
& \sum_{\emptyset\ne I\subseteq \cI}
(-1)^{|I|}
\big([\cN^\circ_{\cD'_{I''}\cap \cE/\cX'}]^{\mathrm{naive}}
-[\cN^\circ_{\cD'_I\cap \cE/\cX'}]^{\mathrm{naive}}
\big).
\end{split}
\end{align}
We prove \eqref{eqn.sufficestoverify} by establishing the equality of summands for each $\emptyset\ne I\subseteq \cI$.
When $\cI'\subseteq I$ we have
$$
\cN^\circ_{\cD'_I/\cX'}\cong
\cN^\circ_{\cD_I/\cX}\setminus (\pi_I^\circ)^{-1}(\cZ_I),
$$
so by Lemma \ref{lem.puncturedsumoflb} the
summand on the left-hand side of \eqref{eqn.sufficestoverify} is
\begin{equation}
\label{eqn.KpiL}
(-1)^{|I|}\sum_{\substack{\text{stabilizer components}\ \cK\subset \cD_I^\circ\\ \text{with}\ \cK\subset \cZ_I}}((\pi_I^\circ)^{-1}(\cK),\overline{\beta_{\cK}(\cD_I^\circ)}),
\end{equation}
whereas when $\cI'\nsubseteq I$ we have
$$
\cN^\circ_{\cD_I/\cX}\cong \cN^\circ_{\cD'_I/\cX'},
$$
and the summand is $0$.
In \eqref{eqn.KpiL},
a summand needs to be understood as $0$ when $\overline{\beta_{\cK}(\cD_I^\circ)}$ contains the trivial character.
In particular, the sum only receives nontrivial contributions from
$\cK\subset \cZ'_I$, where the closed substack
\[ \cZ'_I\subset \cZ_I \]
is defined by the condition that none of characters associated with the divisors $\cD_i$, $i\in I$, appears in $\beta_{\cZ_I}(\cD^\circ_I)$.

To understand the right-hand side of \eqref{eqn.sufficestoverify}, we start with the decomposition as a sum of line bundles
\[
\cN_{\cD'_{I''}\cap \cE/\cX'}\cong \cN_{\cD'_{I''}\cap \cE/\cD'_{I''}}\oplus \bigoplus_{j\in I''}\cN_{\cD_j/\cX}|_{\cD'_{I''}\cap \cE},
\]
which leads to an identification of $\cN^\circ_{\cD'_{I''}\cap \cE/\cX'}$ with
\begin{equation}
\label{eqn.Nidentify}
\Big(
\cN_{\cZ_I/\cD_{I''}}\setminus \bigcup_{i\in \cI'}\cN_{\cZ_I/\cD_{I''\cup \{i\}}}
\Big)
\times_{\cZ_I}
\Big(
\bigtimes_{j\in I''}(\cN_{\cD_j/\cX}|_{\cZ_I}\setminus \cZ_I)
\Big).
\end{equation}

Suppose $\cI'\subseteq I$.
Then $\cN_{\cD'_I\cap \cE/\cX'}$ is a direct sum of line bundles.
One of these is identified with
$\cO_{\bP(\cN_{\cZ\cap \cD_I/\cD_I })}(-1)$.
The other are the restrictions of $\cN_{\cD_j/\cX}$ for $j\in I$, possibly twisted by $\cO_{\bP(\cN_{\cZ\cap \cD_I/\cD_I })}(1)$.
A punctured sum of line bundles remains unchanged, up to isomorphism, if one of the line bundles is twisted by another.
So $\cN^\circ_{\cD'_I\cap \cE/\cX'}$ may be identified with
\begin{equation}
\label{eqn.NNidentify}
\Big(
\cN_{\cZ_I/\cD_I}\setminus \cZ_I
\Big)
\times_{\cZ_I}
\Big(
\bigtimes_{j\in I}(\cN_{\cD_j/\cX}|_{\cZ_I}\setminus \cZ_I)
\Big).
\end{equation}
The short exact sequence
\begin{equation}
\label{eqn.vectorbundlesplit}
0\to \cN_{\cZ_I/\cD_I}\to \cN_{\cZ_I/\cD_{I''}}\to \bigoplus_{j\in I'} \cN_{\cD_j/\cX}|_{\cZ_I}\to 0
\end{equation}
admits splittings Zariski locally on $\cZ_I$.
As well, \eqref{eqn.vectorbundlesplit}
admits a canonical splitting after restricting to $\cZ'_I$.
The latter observation lets us define a closed substack $\mathcal{S}$ of
the restriction of $\cN_{\cZ_I/\cD_{I''}}$ to $\cZ'_I$,
\[
\mathcal{S}:=\ker\big(\cN_{\cZ_I/\cD_{I''}}|_{\cZ'_I}\to \cN_{\cZ_I/\cD_I}|_{\cZ'_I}\big).
\]
Removing $\mathcal{S}$ from $\cN_{\cZ_I/\cD_{I''}}$ and using the identification \eqref{eqn.Nidentify},
we obtain an open substack
\begin{equation}
\label{eqn.UinN}
\cU'\subset \cN^\circ_{\cD'_{I''}\cap \cE/\cX'}.
\end{equation}
We apply Lemma \ref{lem.naiveU}(ii), with $\cU=\cN^\circ_{\cD'_I\cap \cE/\cX'}$, where the covering of $\cZ_I$ by open substacks comes from the Zariski local splittings of \eqref{eqn.vectorbundlesplit}.
Over such an open substack, $\cU$ is identified with the pre-image of $\cN_{\cZ_I/\cD_I}\setminus \cZ_I$, and $\cU'$, with the pre-image of $\cN_{\cZ_I/\cD_I}\setminus \cZ'_I$, for the projection from the fiber product with
$
\bigtimes_{j\in I}(\cN_{\cD_j/\cX}|_{\cZ_I}\setminus \cZ_I)
$.
Application of Lemma \ref{lem.puncturedsumoflb} establishes the
condition on Lemma \ref{lem.naiveU}(i) for $\cU$ and $\cU'$ over such an open substack of $\cZ_I$, thus by Lemma \ref{lem.naiveU}(ii) we have
\[
[\cU']^{\mathrm{naive}}=[\cN^\circ_{\cD'_I\cap \cE/\cX'}]^{\mathrm{naive}}.
\]
From \eqref{eqn.UinN} we obtain an identification of
\[
[\cN^\circ_{\cD'_{I''}\cap \cE/\cX'}]^{\mathrm{naive}}-[\cU']^{\mathrm{naive}}
\]
with the sum in \eqref{eqn.KpiL}, and we have the equality of summands of \eqref{eqn.sufficestoverify} in this case.

We suppose, now, $\cI'\nsubseteq I$.
Then
$\cN^\circ_{\cD'_I\cap \cE/\cX'}$ is identified with
\[
\Big(
\cN_{\cZ_I/\cD_I}\setminus \bigcup_{i\in \cI'\setminus I}\cN_{\cZ_I/\cD_{I\cup \{i\}}}
\Big)
\times_{\cZ_I}
\Big(
\bigtimes_{j\in I}(\cN_{\cD_j/\cX}|_{\cZ_I}\setminus \cZ_I)
\Big).
\]
We have the short exact sequence
\begin{equation}
\label{eqn.vectorbundlesplit2}
0\to \cN_{\cZ_I/\cD_{\cI'\cup I}}\to \cN_{\cZ_I/\cD_{I''}}\to \bigoplus_{j\in \cI'} \cN_{\cD_j/\cX}|_{\cZ_I}\to 0.
\end{equation}
We take a splitting of \eqref{eqn.vectorbundlesplit2} over an open substack of $\cZ_I$.
The restriction to the direct sum of $\cN_{\cD_j/\cX}|_{\cZ_I}$ over $j\in I'$ is a splitting $s$ of
\begin{equation}
\label{eqn.vectorbundlesplit3}
0\to \cN_{\cZ_I/\cD_I}\to \cN_{\cZ_I/\cD_{I''}}\to \bigoplus_{j\in I'} \cN_{\cD_j/\cX}|_{\cZ_I}\to 0.
\end{equation}
We claim:
\begin{align}
\cN_{\cZ_I/\cD_I}\oplus\bigoplus_{j\in I'\setminus\{i\}}s(\cN_{\cD_j/\cX}|_{\cZ_I}) & = \cN_{\cZ_I/\cD_{I''\cup \{i\}}} & (i&\in I'), \label{eqn.p1identity} \\
\cN_{\cZ_I/\cD_{I\cup \{i\}}}\oplus \bigoplus_{j\in I'}s(\cN_{\cD_j/\cX}|_{\cZ_I}) & = \cN_{\cZ_I/\cD_{I''\cup \{i\}}} & (i&\in \cI'\setminus I).
\label{eqn.p2identity}
\end{align}
For $i$, $j\in I'$, $i\ne j$, we have $s(\cN_{\cD_j/\cX}|_{\cZ_I})\subset \cN_{\cZ_I/\cD_{I''\cup \{i\}}}$,
since $\cN_{\cZ_I/\cD_{I''\cup \{i\}}}$
is the kernel of
$\cN_{\cZ_I/\cD_{I''}}\to \cN_{D_i/\cX}|_{\cZ_I}$;
\eqref{eqn.p1identity} quickly follows.
The argument for \eqref{eqn.p2identity} is similar and relies on the fact, that $s$ is the restriction of a splitting of \eqref{eqn.vectorbundlesplit2}.
With \eqref{eqn.p1identity}--\eqref{eqn.p2identity} we have identifications, Zariski locally over $\cZ_I$, of $\cN^{\circ}_{\cD'_{I} \cap \cE/\cX'}$ and $\cN^{\circ}_{\cD'_{I''} \cap \cE/\cX'}$, so by
Lemma \ref{lem.naiveU}(ii) we have
\[
[\cN^{\circ}_{\cD'_{I} \cap \cE/\cX'}]^{\mathrm{naive}}=[\cN^{\circ}_{\cD'_{I''} \cap \cE/\cX'}]^{\mathrm{naive}}.
\]
Thus the summand on the right-hand side of \eqref{eqn.sufficestoverify} is $0$, as desired.
\end{proof}

\begin{prop}
\label{prop.birationalinvariant}
Let $\cU$ be as in Definition \ref{def.naiveclassU}
and let, analogously, $\cU'$ be the complement in a
divisorial projective orbifold $\cX'$ of an snc divisor $\cD'$.
If $\cU$ and $\cU'$ are birationally equivalent, then
$[\cU]=[\cU']$ in $\ocBurn_n$.
\end{prop}

\begin{proof}
We first show that it is enough to treat the case that $\cX'$ admits a proper birational morphism to $\cX$, restricting to a representable proper birational morphism $\cU'\to \cU$.
Proposition \ref{prop.reprproperU} gives rise to an
orbifold $\cW$ with projective morphism to
$\Gamma(\cU\dashrightarrow \cU')$, such that the composite morphisms to $\cU$ and to $\cU'$ are representable and birational.
The projective morphism $\cW\to \Gamma(\cU\dashrightarrow \cU')$ compactifies to some $\cZ\to \Gamma(\cX\dashrightarrow \cX')$, where by resolution of singularities we may suppose $\cZ$ smooth, and by applying embedded resolution of sinigularities to the common pre-image in $\cZ$ of $\cD\subset \cX$ and $\cD'\subset \cX'$ (cf.\ Lemma \ref{lem.openproper}), we may suppose that $\cW$ in $\cZ$ is the complement of an snc divisor.
Now $\cZ\to \cX$ and $\cZ\to \cX'$ are proper birational morphisms of projective orbifolds, which restrict to the representable proper birational morphisms $\cW\to \cU$ and $\cW\to \cU'$.

So we may suppose that we have proper birational
$\cX'\to \cX$, restricting to
representable proper birational $\cU'\to \cU$.
We let $\mathsf{X}'\to \cX$ be the relative coarse moduli space of $\cX'$ over $\cX$ and apply resolution of singularities and
embedded resolution of singularities to obtain
representable proper $\widetilde{\mathsf{X}}'\to \cX$,
restricting to $\cU'\to \cU$, where the complement of $\cU'$ is an snc divisor.
We do the same to $\cX'\times_{\mathsf{X}'}\widetilde{\mathsf{X}}'$ to obtain $\cX''$, also with complement of $\cU'$ an snc divisor.
Proposition \ref{prop.isomorphisminvariant} is applicable to $\cU'$ in
$\cX''$ and in $\widetilde{\mathsf{X}}'$.
The morphism $\widetilde{\mathsf{X}}'\to \cX$ is representable, so by Remark \ref{rem.weakfactrep} it suffices to treat the blow-up of a smooth substack $\cZ$ of $\cX$ that has normal crossing with $\cD$.
The case $\cZ\subset \cD$ is taken care of by Proposition \ref{prop.isomorphisminvariant}.
When $\cZ$ meets $\cD$ transversally we conclude directly by the definition of $[\cU]$ and the birational invariance of the naive class in $\ocBurn_n$.
\end{proof}

\begin{defi}
\label{def.classXclassU}
Let $\cX$ be a projective orbifold of dimension $n$.
The class
\[
[\cX]\in \cBurn_n
\]
is defined by applying
Proposition \ref{prop.orbifolddivisorialification}
to $\cX$, to obtain a divisorial projective orbifold $\cX'$, and setting $[\cX]:=[\cX']$ (Definition \ref{def.classX}).
Let $\cU$ be a quasiprojective orbifold of dimension $n$.
The classes
\[
[\cU], [\cU]^{\mathrm{naive}}\in \ocBurn_n
\]
are defined by taking $\cX$ to be a projective orbifold with snc divisor $\cD$ and $\cX\setminus \cD\cong \cU$ (Remark \ref{rem.compactify}), passing to $\cX'$ as before
(Proposition \ref{prop.orbifolddivisorialification}, applied to $(\cX,\cD)$), and with $\cU':=\cU\times_{\cX}\cX'$, setting $[\cU]:=[\cU']$ and $[\cU]^{\mathrm{naive}}:=[\cU']^{\mathrm{naive}}$ (Definition \ref{def.naiveclassU}).
\end{defi}

The classes in Definition \ref{def.classXclassU} are independent of choices (by Lemma \ref{lem.easybirat} and the birational invariance results in this section) and are birational invariants of projective, respectively, quasiprojective orbifolds.

\section{Specialization}
\label{sect:spec}
Still taking $k$ to be a field of characteristic $0$ with all roots of unity,
in this section we define a specialization homomorphism
$$
\ocBurn_K\to \ocBurn_k,
$$
where $K$ is the fraction field of a complete DVR $\mathfrak o$, with residue field $k$.
We fix a uniformizer $\pi$ of $\mathfrak o$. 

Let $\cX_K$ be a projective orbifold over $K$.
We start by showing the existence of a regular model $\cX_{\mathfrak{o}}$ over $\mathfrak o$,
with simple normal crossing special fiber.
We have $\cX_K\cong [Y_K/GL_N]$ for some smooth quasiprojective $Y_K$ with action of $GL_N$ for some $N$.
By \cite[Thm.\ 3.3]{thomason}, there exists an equivariant embedding $Y_K\to \PP^M_{\mathfrak{o}}$, where the projective space is the projectivization of some representation of $GL_N$ over $\mathfrak{o}$.
We have $Y_K$ contained in the stable locus (in the sense of GIT) for the action of $GL_N$ on $\PP^M_K$.
Applying the blow-up procedure of \cite[Thm.\ 2.11]{ER}, to $[(\PP^M_{\mathfrak{o}})^{ss}/GL_N]$ (where the notation refers to semistable locus, in the sense of GIT over a base \cite[App.\ G to Chap.\ 1]{mumford-git}), we obtain an ambient projective orbifold, containing $\cX_K$, which is the generic fiber of a smooth DM stack over $\mathfrak{o}$ whose coarse moduli space is projective over $\mathfrak{o}$.
The closure of $\cX_K$ is then a DM stack over $\mathfrak{o}$ whose coarse moduli space is projective over $\mathfrak{o}$.
Applying resolution of singularities (cf.\ Remark \ref{rema:sing-res}), we obtain a regular model $\cX_{\mathfrak{o}}$ of $\cX_K$ with snc special fiber.
Moreover, if $\cX_K$ is divisorial, then so is $\cX_{\mathfrak{o}}$.
(Argue as in Remark \ref{rem.compactify}.)
We also recall, by Proposition \ref{prop.brdmdiv}, that if $\cX_K$ is linearizable, then after suitable blow-up it becomes divisorial.

\begin{defi}
\label{def.obv}
The \emph{orbifold Burnside volume} is the homomorphism
\[
\bar\rho_\pi\colon \ocBurn_{n,K}\to \ocBurn_{n,k}
\]
given by 
$$
(\cX_K,\alpha)\mapsto
\sum_{\emptyset\ne I\subseteq \{1,\dots,r\}}
(-1)^{|I|-1}(\omega_I^{-1}(1),\alpha),
$$
where $\cX_K$ is assumed divisorial and $\cX_{\mathfrak{o}}$ is a regular model with snc special fiber $\cD=\cD_1\cup\dots\cup\cD_r$, and for every $\emptyset\ne I\subseteq \{1,\dots,r\}$ the morphism
\[
\omega_I\colon
\cN^\circ_{\cD_I/\cX_{\mathfrak{o}}}\to \G_m
\]
is obtained from
the trivialization of
\[
\bigotimes_{i\in I}\cN^{\otimes d_i}_{\cD_i/\cX_{\mathfrak{o}}}|_{\cD_I},
\]
determined by $\pi$, where $d_i$ is the multiplicity of $\cD_i$ in the special fiber.
\end{defi}

As in \cite[Sect.\ 6]{BnG}, the verification that
Definition \ref{def.obv} yields a well-defined homomorphism is straightforward.
We only require the additional observation, that a root stack operation along a regular divisor in $\cX_{\mathfrak{o}}$, meeting the special fiber transversely, does not change $(\omega_I^{-1}(1),\alpha)$, as an element of $\ocBurn_{n,k}$.
This is clear by relation \textbf{(}$\overline{\mathbf{I}}$\textbf{)}.

\begin{theo}
\label{thm.obvregularmodel}
Let $\cX_K$ be a projective orbifold of dimension $n$, that is divisorial, and let $\cX_{\mathfrak{o}}$ be a regular model with snc special fiber $\cD=\cD_1\cup\dots\cup\cD_r$, and let $\omega_I$ be as above, for $\emptyset\ne I\subseteq \{1,\dots,r\}$.
Then
\[ \bar\rho_\pi([\cX_K])=
\sum_{\emptyset\ne I\subseteq \{1,\dots,r\}}
(-1)^{|I|-1}
\sum_{\cK\subset \cN^\circ_{\cD_I/\cX_{\mathfrak{o}}}}(\cK\cap \omega_I^{-1}(1),\beta_{\cK}(\cN^\circ_{\cD_I/\cX_{\mathfrak{o}}})), \]
where the sum is over stabilizer components of $\cN^\circ_{\cD_I/\cX_{\mathfrak{o}}}$.
\end{theo}

\begin{proof}
This is clear from Definition \ref{def.obv}.
\end{proof}

We also have the analogue of \cite[Rmk.\ 6.7]{BnG} (compatibility with passage to finite extensions of $K$).

\begin{rema}
\label{rem.obv}
In Definition \ref{def.obv} we may always take $\cX_K$ to be of the form $X_K\times BH$ for some scheme $X_K$.
Let $X_{\mathfrak{o}}$ be a regular model with snc special fiber.
Then $X_{\mathfrak{o}}\times BH$ is a regular model of $X_K\times BH$, and we see that $\bar\rho_\pi$ may be
computed by applying the standard specialization map from \cite{KT}, forming the product with $BH$, and carrying along the characters $\alpha$.
\end{rema}

%SPECIALIZATION OF BIRATIONAL TYPES IN SMOOTH FAMILIES?

%\subsection*{Examples}
It is natural to ask whether or not $\bar{\rho}_{\pi}$ admits a refinement to a homomorphism on the level of $\cBurn_n$.
In Section \ref{sec.comparisons} we show that is no such refinement, that is compatible with the refinement homomorphism of $\Burn_n(G)$ (for $n\ge 2$ and finite groups $G$).
Here, we give some examples, illustrating issues that can arise.

\begin{exam}
\label{exa.toric}
Let $C:=\A^1$ and $C_0:=\A^1\setminus \{0\}$.
Given the family of orbifold curves $\PP(1,2)\times C_0$ over $C_0$, we may ask about proper models over $C$.
We illustrate this using
toric DM stacks.
In this theory, a \emph{stacky fan} \cite{BCS} consists of a finitely generated abelian group $N$, a fan $\Sigma$ of cones in $N\otimes \R$, and a collection of ray generators in $N$.
For instance,
\[ \mathbf{\Sigma}_0:=(\Z^2,\Sigma_0,\beta_0),\,\Sigma_0:=\{0,\R_{\ge 0}\times 0,\R_{\le 0}\times 0\},\,\beta_0:=\{(2,0),(0,-1)\}) \]
is a stacky fan, whose associated toric DM stack is
\[ \cX(\mathbf{\Sigma}_0)\cong \PP(1,2)\times C_0. \]
Functoriality \cite[Rmk.\ 4.5]{BCS} associates to a homomorphism of lattices satisfying requirements about images of cones and ray generators, a morphism of toric DM stacks.
With $1_{\Z^2}$ we obtain:
\[ \cX(\mathbf{\Sigma}_0)\to \cX(\mathbf{\Sigma}_1),\qquad
\cX(\mathbf{\Sigma}_0)\to \cX(\mathbf{\Sigma}_2), \]
where $e_1$, $e_2$ denotes the standard basis of $\Z^2$ and:
\begin{gather*}
\mathbf{\Sigma}_1:=(\Z^2,\Sigma_1,\beta_1),\qquad 
\mathbf{\Sigma}_2:=(\Z^2,\Sigma_2,\beta_2), \\
\Sigma_1:=\Sigma_0\cup \{\R_{\ge 0}\langle e_2\rangle,\R_{\ge 0}\langle e_1,e_2\rangle, \R_{\ge 0}\langle -e_1,e_2\rangle\},\,\beta_1:=\beta_0\cup \{e_2\}, \\
\Sigma_2:=\Sigma_0\cup \{\R_{\ge 0}\langle e_1+e_2\rangle,\R_{\ge 0}\langle e_1,e_1+e_2\rangle, \R_{\ge 0}\langle -e_1,e_1+e_2\rangle\}, \\
\beta_2:=\beta_0\cup \{e_1+e_2\}.
\end{gather*}
This way, $\cX(\mathbf{\Sigma}_1)$ and $\cX(\mathbf{\Sigma}_2)$ are proper models over $C$ of $\cX(\mathbf{\Sigma}_0)$:
\begin{equation}
\label{eqn.XS1XS2}
\cX(\mathbf{\Sigma}_1)\dashrightarrow \cX(\mathbf{\Sigma}_2),
\end{equation}
restricting to $1_{\cX(\mathbf{\Sigma}_0)}$ over $C_0$.
We give:
\begin{itemize}
\item a proof that this birational map is not representable,
\item a factorization of this birational map into blow-up and root stack operations.
\end{itemize}

The representability of \eqref{eqn.XS1XS2} is, by Remark \ref{rem.normalizegraph} and equivariant resolution of singularities, equivalent to the torus-equivariant birational equivalence of $\cX(\mathbf{\Sigma}_1)$ and $\cX(\mathbf{\Sigma}_2)$;
a criterion for this due to Schmitt \cite{schmitt} (for proper toric orbifolds, but easily adapted to the present setting), may be tested and seen to fail.
However, we can also give a direct proof of non-representability.
If the birational map
\eqref{eqn.XS1XS2} were representable,
then by Theorem \ref{thm.weakfactrep}, there would exist a factorization as a sequence of blow-ups and their inverses.
As we are considering orbifold surfaces, only blow-ups of points are relevant.
So the restriction of \eqref{eqn.XS1XS2} to any orbifold curve in $\cX(\mathbf{\Sigma}_1)$ would be a morphism.
We obtain a contradiction, by analyzing the restriction of
\eqref{eqn.XS1XS2} to the orbifold curve $\cY_1$ in $\cX(\mathbf{\Sigma}_1)$, corresponding to the ray $\R_{\ge 0}\times 0$.
There is the analogous orbifold curve
$\cY_2$ in $\cX(\mathbf{\Sigma}_2)$.
Using the description of charts of toric DM stacks \cite[Prop.\ 4.3]{BCS}, we determine that the composite
\[
\A^1\times B\Z/2\Z  \cong \cY_1\dashrightarrow \cY_2\cong \A^1\times B\Z/2\Z\stackrel{\mathrm{pr}_2}\to B\Z/2\Z,
\]
restricted to $(\A^1\setminus \{0\})\times B\Z/2\Z$, sends a $\Z/2\Z$-torsor $T\to S$ with invertible function $r$ on $S$ (for any $k$-scheme $S$) to the $\Z/2\Z$-torsor
$T\times_S\widetilde{S}/{\sim}$,
where $\widetilde{S}\to S$ is the $\Z/2\Z$-torsor determined by $r$, and $\sim$ denotes the diagonal $\Z/2\Z$-action.
This morphism does not extend to $\A^1\times B\Z/2\Z$.

The first step in a chain of stacky blow-ups relating $\cX(\mathbf{\Sigma}_1)$ and $\cX(\mathbf{\Sigma}_2)$ (cf.\ Theorem \ref{thm.weakfact}) is the blow-up of the point on $\cY_1$ over $0\in \A^1$.
The corresponding operation on the stacky fan is the \emph{stacky star subdivision} \cite[\S 4]{edidinmore}, which adds the ray $\R_{\ge 0}\langle (1,2)\rangle$ and ray generator $(1,2)$, and subdivides the maximal cone containing this ray.
The diagram
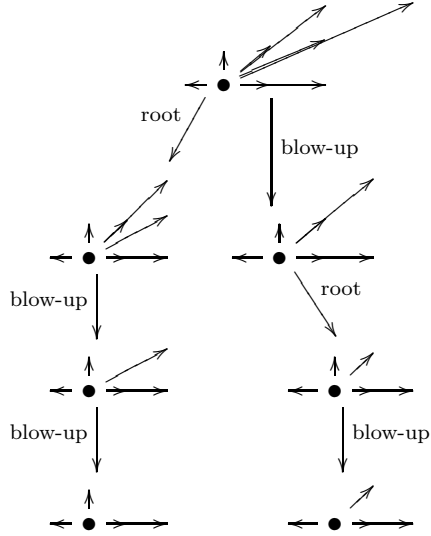
\[
\xymatrix@R=8pt@C=8pt{
{}&{}&{}&{}&{}&{}&{}&{}&{}\\
{}&{}&{}&{}&{}&{}&{}&{}&{}\\
{}&{}&{}&{}&\bullet\ar[l]\ar[u]\ar[ur]\ar[uurr]\ar@<1pt>[urr]\ar[uurrrr]\ar[r]\ar[rr]\ar@<-4pt>[ddl]_{\text{root}}&{}\ar@<-3pt>[ddd]^{\text{blow-up}}&{}&{}&{}\\
{}&{}&{}&{}&{}&{}&{}&{}&{}\\
{}&{}&{}&{}&{}&{}&{}&{}&{}\\
{}&{}&{}&{}&{}&{}&{}&{}&{}\\
{}&\bullet\ar[l]\ar[u]\ar[ur]\ar[uurr]\ar[urr]\ar[r]\ar[rr]\ar@<3pt>[dd]_{\text{blow-up}}&{}&{}&{}&\bullet\ar[l]\ar[u]\ar[ur]\ar[uurr]\ar[r]\ar[rr]\ar@<2pt>[ddr]^{\text{root}}&{}&{}&{}\\
{}&{}&{}&{}&{}&{}&{}&{}&{}\\
{}&{}&{}&{}&{}&{}&{}&{}&{}\\
{}&\bullet\ar[l]\ar[u]\ar[urr]\ar[r]\ar[rr]\ar@<3pt>[dd]_{\text{blow-up}}&{}&{}&{}&{}&\bullet\ar[l]\ar[u]\ar[ur]\ar[r]\ar[rr]\ar@<3pt>[dd]^{\text{blow-up}}&{}&{}\\
{}&{}&{}&{}&{}&{}&{}&{}&{}\\
{}&{}&{}&{}&{}&{}&{}&{}&{}\\
{}&\bullet\ar[l]\ar[u]\ar[r]\ar[rr]&{}&{}&{}&{}&\bullet\ar[l]\ar[ur]\ar[r]\ar[rr]&{}&{}
}
\]
shows all of the steps in a chain of stacky blow-ups relating $\cX(\mathbf{\Sigma}_1)$ and $\cX(\mathbf{\Sigma}_2)$.
Blow-ups are represented by stacky star subdivision, root operations by the doubling of an entry in the $\beta$-component of the stacky fan (shown graphically by a double-arrow).
\end{exam}

Although the rational map in Example \ref{exa.toric}, extending the identity morphism over $C_0$, between the two proper models over $C$ is not representable (so ordinary blow-ups alone do not suffice to relate the two models), the fiber over $0\in C$ in each case is $\PP(1,2)$.
(In fact, $\cX(\mathbf{\Sigma}_1)\cong\cX(\mathbf{\Sigma}_2)$, although there is no isomorphism that restricts to $1_{\cX(\mathbf{\Sigma}_0)}$.)
In the next example, we will see proper models whose special fibers differ in an essential way.

\begin{exam}
\label{exa.ellipticcurve}
Keeping $C$ and $C_0$ as above, we let $E$ be an elliptic curve and $p$ a $k$-point of $E$, and we consider the family of one-dimensional DM stacks $C_0\times E\times B(\Z/2\Z)$ over $C_0$.
Of course, this admits the proper model
\[ \cX_1:=C\times E\times B\Z/2\Z \]
over $C$.

We use the square root of a line bundle, recalled in \eqref{eqn.rootlb}, to construct another proper model over $C$.
In the blow-up
\[ B:=B\ell_{(0,p)}C\times E, \]
we denote the exceptional divisor by $D$.
We let
\[ \cX_2:=\sqrt{\cO_B(D)/B}. \]
Restricted to any Zariski open subset $U$ of $B$, where the divisor $D$ is principal, this stack is isomorphic to $U\times B(\Z/2\Z)$.
So its local geometry is like a product of $B(\Z/2\Z)$ with a semistable family of curves; in particular, $\cX_2$ is smooth.
The line bundle $\cO(D)$, restricted to the copy of $E$ in $B$ over $0$, is $\cO_E(p)$, hence in $\cX_2$ over $0$ we find a copy of
\[ \sqrt{\cO_E(p)/E}, \]
which is a nontrivial $\Z/2\Z$-gerbe over $E$, i.e., is not isomorphic to the product $E\times B(\Z/2\Z)$.
The special fibers of models $\cX_1$ and $\cX_2$ cannot be related in an evident manner (cf.\ Example \ref{exa.orbifoldcurves}).
\end{exam}

\begin{rema}
\label{rem.ellipticcurve}
It is straightforward to modify Example \ref{exa.ellipticcurve} to exhibit families of projective \emph{orbifolds} over $C$, identical over $C_0$, with contrasting special fibers.
In the root construction \eqref{eqn.rootlb}, the right-hand factor supplies a line bundle, whose $r$th power is isomorphic to the pull-back of $L$.
We denote this line bundle on $\cX_2$ by $M$.
Now
\[ \cY_1:=[C\times E\times \PP^1/(\Z/2\Z)]
\qquad\text{and}\qquad
\cY_2:=\PP(\cO_{\cX_2}\oplus M)
\]
(on the left, $\Z/2\Z$ acts nontrivially on $\PP^1$) are proper models over $C$ of $C_0\times E\times [\PP^1/(\Z/2\Z)]$ over $C_0$.
The special fiber of $\cY_1$ is
$E\times [\PP^1/(\Z/2\Z)]$.
The special fiber of $\cY_2$ has a component, isomorphic to
\[ \PP( \cO_{\sqrt{\cO_E(p)/E}}\oplus M|_{\sqrt{\cO_E(p)/E}}), \]
which we claim is not birationally equivalent to $E\times [\PP^1/(\Z/2\Z)]$.
Indeed, the direct argument for non-representability from the previous example reduces this to the fact that $\sqrt{\cO_E(p)/E}$ is not isomorphic to $E\times B\Z/2\Z$.
\end{rema}

\section{Comparisons}
\label{sec.comparisons}
We continue to work over a field $k$ of characteristic $0$, containing all roots of unity.

\subsection*{Burnside groups of varieties and orbifolds}

In \cite{Bbar}, we defined another abelian group, receiving birational invariants of orbifolds of dimension $n$:
\[
\overline{\mathrm{Burn}}_n.
\]
Generators are pairs
\[ (K,\alpha), \]
consisting of a field $K$, finitely generated and of transcendence degree $d\le n$ over $k$, and $\alpha\in \overline{\mathcal{B}}_{n-d}$, a group which encodes the normal bundle representation at the generic point of a locus with given stabilizer group; see \cite[Defn.\ 3.1]{Bbar}.
The pairs are subject to an equivalence relation, which identifies a pair $(K,\alpha)$ such that $K\cong K_0(t)$, with the pair $(K_0,\alpha')$, where $\alpha'$ is obtained from $\alpha$ by addition of a trivial character.
We denote by
\[ \mathrm{cl}(\cU)\in \oBurn_n \]
the class of a
quasiprojective $n$-dimensional orbifold $\cU$, which is defined and shown to be a birational invariant in \cite[Thm.\ 4.1]{Bbar}.
When $\cU$ is divisorial, $\mathrm{cl}(\cU)$ is obtained by stratifying $\cU$ by isomorphism class of stabilizer, assigning to each stratum a class in $\oBurn_n$, and summing.

The module $\overline{\mathcal{B}}$ admits split monomorphisms from modules $\overline{\mathcal{B}}^{[e]}$, for positive integers $e$, corresponding to $e$-torsion stabilizer groups.
The splitting, for $e=1$, determines $\oBurn_n\to\Burn_n$, sending $(K,\alpha)$ as above to the class in $\Burn_n$ of $K(t_1,\dots,t_{n-d})$.
There is also the free abelian group on birational equivalence classes of projective $n$-dimensional orbifolds
$$
\Z[\mathcal{B}\mathit{ir}_n],
$$
where the invariants of birational maps of orbifolds of \cite{KT-map} take values.

\begin{defi}
\label{defn.Burnmaps}
The \emph{comparison homomorphisms}
$$
\Z[\mathcal{B}\mathit{ir}_n] \to \cBurn_n \to 
\ocBurn_n\stackrel{\bar\kappa}\to \oBurn_n\to \Burn_n,
$$
are defined as follows.
\begin{itemize}
\item The leftmost map sends the class of the $n$-dimensional projective orbifold $\cX$ to $[\cX]\in \cBurn_n$, which is well-defined by Theorem \ref{theo.main-th}.
\item The next map is the canonical homomorphism (cf.\ Section \ref{sect:burninv}).
\item The map $\bar\kappa$ given by
\[ (\cK,\alpha)\mapsto \mathrm{cl}(\cK_{\mathrm{cms}}\times [\A^{n-d}/H]), \]
where $\cK_{\mathrm{cms}}$ denotes the coarse moduli space of $\cK$ and $H$ acts linearly on the affine space by the characters in $\alpha$.
\item The final map is given by
$(K,\alpha)\mapsto [K(t_1,\dots,t_{n-d})]$; see above.
\end{itemize}
We denote by
\[ \kappa\colon \cBurn_n\to \oBurn_n \]
the composite of $\bar\kappa$ with the canonical homomorphism.
\end{defi}

\begin{rema}
\label{rem.explicitHquotient}
The map $\bar\kappa$ may be described more explicitly, by applying the following formula, given in \cite[Exa.\ 7.1]{BnG}.
Let $X$ be a projective variety of dimension $d$, and $H$ a finite abelian group with faithful representation of dimension $n-d$, given by characters $a_1$, $\dots$, $a_{n-d}$.
Then
\[ \mathrm{cl}(X\times [\A^{n-d}/H])=\sum_{I\subseteq \{1,\dots,n-d\}}(-1)^{|I|}(k(X),(\bar a_1,\dots,\bar a_{n-d})), \]
where $\bar a_1$, $\dots$, $\bar a_{n-d}$ denote the restrictions of the characters to the intersection of $\ker(a_i)$ for $i\in I$.
\end{rema}

\begin{prop}
\label{prop.kappa}
The map $\bar\kappa$ in Definition \ref{defn.Burnmaps} respects the relations in $\ocBurn_n$, thus is a well-defined homomorphism.
\end{prop}

\begin{proof}
The map $\bar\kappa$ clearly respects relations $\textbf{(O)}$ and \textbf{(}$\overline{\mathbf{I}}$\textbf{)}.
In fact, modulo relation \textbf{(}$\overline{\mathbf{I}}$\textbf{)} every symbol in $\ocBurn_n$ is of the form $(X\times BH,\alpha)$ for an irreducible projective scheme $X$.
If we write $\alpha$ as a pair $(A,S)$ with $A=H^\vee$ and $S$ a sequence of nontrivial characters generating $A$ of length equal to $n-\dim(X)$, then we may omit the $BH$-factor and record only the birational type of $X$:
\begin{equation}
\label{eqn.newsymbol}
(k(X),A,S)\in \ocBurn_n.
\end{equation}

To simplify the checking of the remaining relations we adopt the convention, that we allow the trivial character to occur in $S$, but additionally impose the relation $(k(X),A,S)=0$ if $S$ contains the trivial character.
With this convention,
relation \textbf{(B)} takes the simple form
\begin{align}
\begin{split}
\label{eqn.relationB}
(&k(X),A,(a_1,a_2,\dots)) \\
&=(k(X),A,(a_1,a_2-a_1,\dots))+(k(X),A,(a_1-a_2,a_2,\dots))\\ &\qquad\qquad\qquad\qquad+(k(X\times \bP^1),A/\langle a_1-a_2\rangle,(\bar a_2,\dots)),
\end{split}
\end{align}
and relation \textbf{(V)} is a special case of \textbf{(B)}.

The value of $\bar\kappa$ on a symbol \eqref{eqn.newsymbol} is given by the formula in Remark \ref{rem.explicitHquotient}; we write $S=(a_1,\dots,a_m)$, with $m=n-\dim(X)$, then
on the right-hand side we have $\bar a_1$, $\dots$, $\bar a_m\in A/\langle a_i\rangle_{i\in I}$.
Due to the sign $(-1)^{|I|}$, this vanishes whenever $S$ contains the trivial character.
Consequently, to prove the proposition, it suffices to check that $\bar\kappa$ respects the relation \eqref{eqn.relationB}.
That is, we need to prove the vanishing in $\oBurn_n$ of
\[ B+C+D+E, \]
where
\begin{align*}
B&=\sum_{I\subseteq \{1,\dots,m\}}(-1)^{|I|}(k(X),A/\langle a_i\rangle_{i\in I},(a_1,\dots,a_m)), \\
C&=-\sum_{I\subseteq \{1,\dots,m\}}(\text{analogous, with characters $(a_1,a_2-a_1,\dots,a_m)$}), \\
D&=-\sum_{I\subseteq \{1,\dots,m\}}(\text{analogous, with characters $(a_1-a_2,a_2,\dots,a_m)$}), \\
E&=-\sum_{I'\subseteq \{2,\dots,m\}}(-1)^{|I'|}(k(X),A/\langle a_1-a_2,a_i\rangle_{i\in I'},(0,a_2,\dots,a_m)).
\end{align*}
Here, the $a_i$ on the right-hand side are understood to denote classes in the respective quotients of $A$.
We have used the relation from \cite[Defn.\ 3.3]{Bbar} to express all of the symbols in terms of a common function field.

Let us decompose $B$ according to $I\cap \{1,2\}$ as $B_\emptyset+B_1+B_2+B_{12}$, and analogously decompose $C$, $D$, and $E=E_\emptyset+E_2$.
Then we have
\[ B_1+C_1=B_2+D_2=C_2+E_\emptyset=B_{12}+C_{12}=D_{12}+E_2=0.
\]
So $B+C+D+E$ is equal to $B_\emptyset+C_\emptyset+D_\emptyset+D_1$, which is
\[
\sum_{I''\subseteq\{3,\dots,m\}}(-1)^{|I''|}F_{I''},
\]
where
\begin{align*}
F_{I''}&=(k(X),A/\langle a_i\rangle_{i\in I''},(a_1,\dots,a_m)) \\
&\qquad-(k(X),A/\langle a_i\rangle_{i\in I''},(a_1,a_2-a_1\dots,a_m))\\
&\qquad-(k(X),A/\langle a_i\rangle_{i\in I''},(a_1-a_2,a_2,\dots,a_m))\\
&\qquad+(k(X),A/\langle a_1-a_2,a_i\rangle_{i\in I''},(0,a_2,\dots,a_m).
\end{align*}
We have $F_{I''}=0$ in $\oBurn_n$ by the final relation of \cite[Defn.\ 3.1]{Bbar} (applied with $j=2$).
\end{proof}

%\subsection*{Injectivity}
 
\begin{prop}
\label{prop.ZBirtoBurn}
The homomorphism
$$
\Z[\mathcal{B}\mathit{ir}_n]\to \cBurn_n,
$$
sending the birational equivalence class of $\cX$ to $[\cX]\in \cBurn_n$, is injective. 
\end{prop}

\begin{proof}
The classes $(\cK,\alpha)$ with $\dim(\cK)=n$, modulo relation \textbf{(I)}, form a direct summand of $\cBurn_n$, isomorphic to $\Z[\mathcal{B}\mathit{ir}_n]$.
\end{proof}

We remark that the conclusion of Proposition \ref{prop.ZBirtoBurn} is no longer valid if we further compose with the canonical homomorphism to $\ocBurn_n$.
For instance, orbifold curves, which have $\PP^1$ as coarse moduli space and four points with stabilizer $\Z/2\Z$,
have varying birational type, since the birational type in this case is the same as the isomorphism type (Example \ref{exa.orbifoldcurves}).
But their classes all become equal modulo relation \textbf{(}$\overline{\mathbf{I}}$\textbf{)}.

\begin{lemm}
\label{lem.constantstabilizer}
Let $\cU$ be a quasiprojective smooth DM stack, with line bundles $L_1$, $\dots$, $L_r$,
such that $\dim(\cU)+r=n$.
Assume that $\cU$ has constant stabilizer group $H$, and the line bundles induce a faithful representation of $H$ with associated characters $\alpha=(a_1,\dots,a_r)$.
Then in $\oBurn_n$ we have
\[ \bar\kappa([L_1\oplus\dots\oplus L_r]^{\mathrm{naive}})=(k(\cU),\alpha). \]
\end{lemm}

\begin{proof}
Let $U$ denote the coarse moduli space of $\cU$.
By Definition \ref{def.naiveclassU} and relation \textbf{(}$\overline{\mathbf{I}}$\textbf{)} we have
\[
[L_1\oplus\dots\oplus L_r]^{\mathrm{naive}}=
\sum_{\substack{I\subseteq\{1,\dots,r\}\\ \forall\,j\notin I:\,a_j\notin \langle a_i\rangle_{i\in I}}}(U\times \PP^{|I|}\times BH_I,(\bar a_j)_{j\notin I}),
\]
where $H_I$ denotes the intersection of $\ker(a_i)$ over $i\in I$ and $\bar a_j$, the restriction of $a_j$ to $H_I$.
Using Remark \ref{rem.explicitHquotient} and the relations in $\oBurn_n$, mentioned at the beginning of this section, we find
\begin{align*}
\bar\kappa&([L_1\oplus\dots\oplus L_r]^{\mathrm{naive}}) \\
&\qquad=\sum_{\substack{I\subseteq\{1,\dots,r\}\\ \forall\,j\notin I:\,a_j\notin \langle a_i\rangle_{i\in I}}}\sum_{\substack{J\subseteq \{1,\dots,r\}\\ I\cap J=\emptyset}}(-1)^{|J|}(k(U),(\bar{\bar a}_1,\dots,\bar{\bar a}_r)),
\end{align*}
where $\bar{\bar a}_j$ denotes the restriction of $a_j$ to $H_{I\cup J}$.
The right-hand side may be rewritten as
\[
\sum_{I\subseteq\{1,\dots,r\}}\sum_{\substack{J\subseteq \{1,\dots,r\}\\ I\cap J=\emptyset}}(-1)^{|J|}(k(U),(\bar{\bar a}_1,\dots,\bar{\bar a}_r)),
\]
since the additional summands, with $I$ such that there exists $j_0\notin I$ with $a_{j_0}\in \langle a_i\rangle_{i\in I}$,
have cancelling pairs of terms in the inner sum, specifically those with $j_0\notin J$ and those with $j_0\in J$.
Now the summand has the sign $(-1)^{|J|}$ times a class, that depends only on $I\cup J$.
So we get cancellation, leaving only the contribution from $I=J=\emptyset$, which is $(k(U),\alpha)$.
\end{proof}

\begin{prop}
\label{prop.kappabar}
We have
\[ \bar\kappa([\cU])=\mathrm{cl}(\cU) \]
for every quasiprojective orbifold of dimension $n$.
In particular,
\[ \kappa([\cX])=\mathrm{cl}(\cX) \]
if $\cX$ is a projective orbifold of dimension $n$.
\end{prop}

The proof uses some additional notation.
Let $\cU$ be a quasiprojective smooth DM stack with line bundles $L_1$, $\dots$, $L_r$, such that the total space of $L_1\oplus \dots\oplus L_r$ is a divisorial orbifold of dimension $n$.
Then (cf.\ \cite[(7.4)]{BnG}) we define
\[ \mathrm{cl}(\cU,(L_1,\dots,L_r)) \]
by the same recipe, as in the definition of the class of a divisorial orbifold, except that to each pair $(K,\alpha)$ we adjoin to $\alpha$ the characters, determined by $L_1$, $\dots$, $L_r$.
We also recall the formula from \cite[Lemma 7.4]{BnG}, when $\cU$ compactifies and the line bundles are extended to a projective smooth DM stack $\cX$, with $\cX\setminus \cU$ an snc divisor:
\begin{align}
\begin{split}
\label{eqn.classUL}
\mathrm{cl}&(\cU,(L_1,\dots,L_r))=\mathrm{cl}(\cX,(L_1,\dots,L_r)) \\
&+\sum_{\emptyset\ne I\subseteq \{1,\dots,r\}}(-1)^{|I|}\mathrm{cl}(\cD_I,(L_1|_{\cD_I},\dots,L_r|_{\cD_I},\overbrace{\dots,\mathcal{N}_{\cD_i/\cX}|_{\cD_I},\dots}^{i\in I})).
\end{split}
\end{align}

\begin{proof}
We follow closely the proof of \cite[Prop.\ 7.9]{BnG}.
By applying \eqref{eqn.classUL} with $r=0$ and $\cX$ and $\cD$ as in Definition \ref{def.naiveclassU}, we are reduced to showing
\begin{equation}
\label{eqn.kappabar}
\bar\kappa([L_1\oplus\dots\oplus L_r]^{\mathrm{naive}})=\mathrm{cl}(\cX,(L_1,\dots,L_r)),
\end{equation}
for a projective smooth DM stack $\cX$ with line bundles $L_1$, $\dots$, $L_r$, such that the total space of $L_1\oplus\dots\oplus L_r$ is a divisorial orbifold of dimension $n$.
We prove of \eqref{eqn.kappabar} by induction on $\dim(\cX)$.
The case $\dim(\cX)=0$ follows from Lemma \ref{lem.constantstabilizer}.

Suppose $\dim(\cX)>0$.
Since the naive class is a birational invariant,
the left-hand side of \eqref{eqn.kappabar} remains unchanged if we blow up $\cX$ along a nonsingular substack.
Arguing as in the proof of \cite[Thm.\ 4.1]{Bbar}, we see that the right-hand side also remains unchanged.
Thus by embedded resolution of singularities it suffices to prove \eqref{eqn.kappabar}, under the assumption that there is dense open $\cU\subset \cX$ with $\cX\setminus \cU$ an snc divisor and $\cU \cong U\times BH$ for some scheme $U$ and finite abelian group $H$ (cf.\ Proposition \ref{prop.divisorial}).

By combining the expressions for $[\cU]$ from Definition \ref{def.naiveclassU} and Lemma \ref{lem.punctured} we get an equality, with $[\cX]$ and naive classes of normal bundles on the left-hand side and $[\cU]^{\mathrm{naive}}$ and naive classes of punctured normal bundles on the right-hand side.
An analogous result is valid, with $[\cX]$ replaced by the naive class of $L_1\oplus\dots\oplus L_r$ and the pullbacks of the line bundles inserted into the remaining terms, which have been rearranged so they all appear on the right-hand side:
\begin{align*}
[&L_1\oplus\dots\oplus L_r]^{\mathrm{naive}} \\
&=[L_1|_\cU\oplus\dots\oplus L_r|_\cU]^{\mathrm{naive}}+\sum_{\emptyset\ne I\subseteq \cI} (-1)^{|I|}[L_1|_{\cN^\circ_{\cD_I/\cX}}\oplus\dots\oplus L_r|_{\cN^\circ_{\cD_I/\cX}}]^{\mathrm{naive}}\\
&\qquad\qquad\qquad\qquad\qquad-\sum_{\emptyset\ne I\subseteq \cI}(-1)^{|I|}[L_1|_{\cN_{\cD_I/\cX}}\oplus\dots\oplus L_r|_{\cN_{\cD_I/\cX}}]^{\mathrm{naive}}
\end{align*}
Applying $\bar\kappa$, we get an equality in $\oBurn_n$.
Lemma
\ref{lem.constantstabilizer} is applicable to the first two terms on the right-hand side, and these combine to yield $\mathrm{cl}(\cU,(L_1,\dots,L_r))$ by \cite[Lemma 7.5]{BnG}.
The induction hypothesis is applicable to the third term, which then by \eqref{eqn.classUL} yields $$\mathrm{cl}(\cX,(L_1,\dots,L_r))-\mathrm{cl}(\cU,(L_1,\dots,L_r)).
$$
This completes the proof of \eqref{eqn.kappabar} and therefore also the proposition.
\end{proof}

\begin{prop}
\label{prop.toBurnq}
The composite homomorphism
\[ \ocBurn_n\to \Burn_n \]
in Definition \ref{defn.Burnmaps}
sends $[\cU]^{\mathrm{naive}}$ to $[k(U)]$ and $[\cU]$ to $[U]$, for an $n$-dimensional quasiprojective orbifold $\cU$ with coarse moduli space $U$.
\end{prop}

\begin{proof}
We may suppose $\cU$ divisorial.
By Remark \ref{rem.explicitHquotient}, any class $(\cK,\alpha)$ with $\dim(\cK)<n$ maps to $0$ in $\Burn_n$.
With this observation we get the assertion about $[\cU]^{\mathrm{naive}}$ from Definition \ref{def.naiveclassU}.
For the assertion about $[\cU]$, we suppose further that $\cU=\cX\setminus \cD$ for an snc divisor $\cD$ (Remark \ref{rem.compactify}).
Now the assertion about $[\cU]$ follows, by Definition \ref{def.naiveclassU} and \cite[Lemma 7.5]{BnG}, from the assertion about $[\cU]^{\mathrm{naive}}$.
\end{proof}

\subsection*{Burnside groups of $G$-varieties}

In \cite{BnG}, we defined,  via generators and relations, the abelian group
\[
\Burn_n(G),
\]
receiving birational invariants of $G$-actions on $G$-varieties. We established 
a well-defined homomorphism of graded abelian groups \cite[Prop.\ 7.7]{BnG}:
\[
\kappa^G\colon \Burn(G)=\bigoplus_{n=0}^\infty \Burn_n(G)\to \oBurn. 
\]
We now establish a refinement of $\kappa^G$ to a homomorphism
$$
\Burn(G)\to \cBurn.
$$
For this we use the style of presentation of $\Burn_n(G)$, given in
\cite{KT-vector}, involving centralizers of stabilizers.

\begin{prop}
\label{prop:compareburn}
There is a homomorphism
$$
\Burn_n(G) \to \cBurn_n,
$$
given by
$$
(H,Z\actsfromleft K,\beta) \mapsto (\cK,\beta),
$$
where $\cK=[W/Z_G(H)]$, for some smooth projective $W$ with regular action of $Z=Z_G(H)/H$ and therefore also action of $Z_G(H)$,
and $Z$-equivariant isomorphism $k(W)\cong K$.
The composite
\[ 
\Burn_n(G) \to \cBurn_n
\stackrel{\kappa}{\lra} \oBurn_n 
\]
is equal to $\kappa^G$.
\end{prop}

\begin{proof}
By Example \ref{exa.reproverZ} and relation \textbf{(I)}, the element $(\cK,\beta)\in \cBurn_n$ is independent of the choice of $W$.
The defining relations of $\Burn_n(G)$ (see \cite[\S 2]{KT-vector}) match the defining relations of $\cBurn_n$.
Therefore we have a well-defined homomorphism.
Comparison of \cite[Defn.\ 7.6]{BnG} and Definition \ref{defn.Burnmaps} yields the compatibility.
\end{proof}

We denote by $\bar\kappa^G$ the composite homomorphism
\[ 
\Burn_n(G) \to \cBurn_n\to \ocBurn_n. 
\]

\subsection*{Identification with $\oBurn_n$}

We now turn to the homomorphism 
$$
\bar\kappa\colon \ocBurn_n \to \oBurn_n,
$$ 
from Definition \ref{defn.Burnmaps}. 

We claim, first, that the definition of $\oBurn_n$ remains unchanged if we omit all the $j\ge 3$ defining relations in
\cite[Defn.\ 3.1]{Bbar}.
We impose just the relations of reordering of characters, isomorphisms of groups of characters, and the $j=2$ relation
\begin{align}
\begin{split}
\label{eqn.jtworelation}
(A,(a_1,a_2,\dots))
=(A,(a_1&,a_2-a_1,\dots))+(A,(a_1-a_2,a_2,\dots))\\ &\qquad\qquad-(A/\langle a_1-a_2\rangle,(0,\bar a_2,\dots)).
\end{split}
\end{align}
Then we deduce the $j\ge 3$ relations by an
inductive argument as in the proof of \cite[Prop.\ 4.7(ii)]{BnG}: if we apply the induction hypothesis and, to every term, the $j=2$ relation, we get exactly the relation to be proved modulo one further application of the induction hypothesis.

Consequently, $\oBurn_n$ admits a presentation by symbols $(K,A,S)$
with finitely generated field $K$ over $k$, finite abelian group $A$, and generating system $S=(a_1,\dots,a_m)$ of $A$ with $m+\mathrm{trdeg}_{K/k}=n$, modulo reordering of characters, isomorphism of finite abelian groups, isomorphism of fields, relation
\begin{equation}
\label{eqn.modulet}
(K,A,(0,a_2,\dots,a_m))=(K(t),A,(a_2,\dots,a_m)),
\end{equation}
and relation \eqref{eqn.jtworelation} with any given $K$.

Next we recall the presentation of $\ocBurn_n$ from the proof of Proposition \ref{prop.kappa}.
Generators are symbols $(K,A,S)$ as above.
Relations are reordering of characters, isomorphism of finite abelian groups, isomorphism of fields, vanishing of symbols containing the trivial character, and \eqref{eqn.relationB} with $K=k(X)$, where in place of $k(X\times \bP^1)$ we put $K(t)$.

\begin{prop}
\label{prop.kappabarisomorphism}
The homomorphism
\[ \bar\kappa\colon \ocBurn_n\to \oBurn_n \]
is an isomorphism.
\end{prop}

\begin{proof}
A map in the other direction sends $(K,A,S)$ to
\[ \sum_{I\subseteq\{1,\dots,m\}} (K(t_1,\dots,t_{|I|}),A/\langle a_i\rangle_{i\in I},(\bar a_j)_{j\in \{1,\dots,m\}\setminus I}). \]
It is clear that this map respects reordering of characters, isomorphism of finite abelian groups, and isomorphism of fields.
By the vanishing in $\ocBurn_n$ of symbols containing the trivial character, the map respects relation \eqref{eqn.modulet}.
The verification that it respects relation \eqref{eqn.jtworelation} is achieved just as in the proof of Proposition \ref{prop.kappa}.
So we have a homomorphism $\oBurn_n\to \ocBurn_n$.
We see that this is inverse to $\bar\kappa$ by direct computation, using relation \eqref{eqn.modulet} respectively the vanishing in $\ocBurn_n$ of symbols containing the trivial character.
\end{proof}

\subsection*{Compatibility with specialization}
Let $K$ be the fraction field of a complete DVR with residue field $k$.
In \cite{BnG}, we have defined a specialization homomorphism of equivariant Burnside groups
$$
\rho_\pi^G\colon
\Burn_{n,K} (G)\to \Burn_{n,k}(G). 
$$
Here we show the compatibility with $\bar{\rho}_\pi$.

\begin{prop}
\label{prop.kapparho}
Let $G$ be a finite group, $K$ the fraction field of  a complete DVR with residue field $k$ of characteristic zero containing all roots of unity, and $\pi$ a uniformizer.
Then we have a commutative diagram
\[
\xymatrix{
\Burn_{n,K}(G) \ar[r]^{\rho^G_\pi}\ar[d]_{\bar\kappa^G} &
\Burn_{n,k}(G)\ar[d]^{\bar\kappa^G} \\
\ocBurn_{n,K}\ar[r]^{\bar\rho_\pi}&
\ocBurn_{n,k} 
}
\]
\end{prop}

\begin{proof}
It suffices to consider a symbol in $\Burn_{n,K}(G)$ of the form $(H,Z\actsfromleft K(W),\beta)$, where $W$ is smooth projective over $K$ and $[W/Z]$ is divisorial.
Then, by \cite[Prop.\ 2.2]{BnG} (really its proof) and Proposition \ref{prop.divisorial}, also
$[W/Z_G(H)]$ is divisorial.
We take $W_{\mathfrak{o}}$ to be a regular model over $\mathfrak{o}$ with compatible $Z$-action and special fiber an snc divisor $D_1\cup\dots\cup D_r$, where each $D_i$ is $Z$-invariant.
With this model, it is straightforward to verify the claimed compatibility.
\end{proof}

\subsection*{Failure of refined specialization}
In Section \ref{sect:spec} we indicated that the specialization 
map $\bar\rho_\pi$ cannot be refined to the level of $\cBurn$.

\begin{theo}
\label{thm.norefine}
Suppose $n\ge 2$.
There does not exist a homomorphism
\[ \cBurn_{n,K}\to \cBurn_{n,k}, \]
such that the diagrams
\[
\xymatrix@C=16pt{
\Burn_{n,K}(\Z/2\Z) \ar[r]^{\rho^{\Z/2\Z}_\pi}\ar[d] &
\Burn_{n,k}(\Z/2\Z)\ar[d]\ar@{}[dr]|{\textstyle \text{and}}&\Burn_{n,K}(\mathfrak{Q}_8) \ar[r]^{\rho^{\mathfrak{Q}_8}_\pi}\ar[d]&\Burn_{n,k}(\mathfrak{Q}_8)\ar[d] \\
\cBurn_{n,K}\ar[r]&
\cBurn_{n,k} &\cBurn_{n,K}\ar[r]&\cBurn_{n,k}
}
\]
commute, where $\mathfrak{Q}_8$ denotes the quaternion group and the vertical maps are as defined in Proposition \ref{prop:compareburn}.
\end{theo}

The proof is by an explicit construction, essentially a version of Example \ref{exa.ellipticcurve} with the nontrivial gerbes realized as $\mathfrak{Q}_8$-quotient stacks.

\begin{proof}
It suffices to treat the case $n=2$.
Fix any $\lambda\in \Q\setminus \{0,1\}$.
Let $W$ be the smooth projective curve of genus $2$ over $K\cong k((t))$,
whose function field is the extension of $K(u)$, obtained by adjoining a square root of $u(u^2-u+t^2)(u^2-\lambda u+t^2)$. We consider the $\Z/2\Z$-action on $W\times \PP^1$, where the generator acts trivially on $W$ and by $-1$ on $\PP^1$.
To compute $\rho_\pi^{\Z/2\Z}([W\times \PP^1\actsfromright \Z/2\Z])$, we complete the affine chart
\[ \Spec \big(k((t))[u,v,w]/(uv-t^2,w^2-u(u+v-1)(u+v-\lambda))\big) \]
of $W$
to the regular projective model
\begin{align*}
\cW&=\Proj \big(k[[t]][u,v,w,x,y]/(uv-t^2y^2,\\
&\,\,\,\,w^2y-u(u+v-y)(u+v-\lambda y),x^2y-v(u+v-y)(u+v-\lambda y),\\
&\,\,\,\,wx-t(u+v-y)(u+v-\lambda y),ux-twy,vw-txy)\big).
\end{align*}
The model $\cW$ is semistable, with special fiber isomorphic to the union of two curves of genus $1$, isomorphic to each other, with function field
\begin{equation}
\label{eqn.E}
k(u)\big(\sqrt{u(u-1)(u-\lambda)}\big).
\end{equation}
So $\rho_\pi^{\Z/2\Z}([W\times \PP^1\actsfromright \Z/2\Z])$ is independent of the choice of $\pi$ and has the form
\begin{align}
\begin{split}
\label{eqn.C2class}
2&\big(\mathrm{triv},\Z/2\Z\actsfromleft k(u,v)\big(\sqrt{u(u-1)(u-\lambda)}\big),()\big) \\
&\qquad-(\mathrm{triv},\Z/2\Z\actsfromleft k(u,v),())+\dots,
\end{split}
\end{align}
where the $\Z/2\Z$-actions are on the variable $v$ only, and we display only the terms with trivial group as first factor.
Indeed, it will be enough to consider only such terms to see the impossibility of a homomorphism as in the statement of the theorem.

The remainder of the proof consists of the following steps.
\begin{itemize}
\item Exhibiting an isomorphism
\[ W\times [\PP^1/(\Z/2\Z)]\cong [X/\mathfrak{Q}_8] \]
for a smooth projective variety $X$ over $K$ with action of $\mathfrak{Q}_8$.
\item Determination of $\rho_\pi^{\mathfrak{Q}_8}([X\actsfromright \mathfrak{Q}_8])$.
\item Verification that the image of $\rho_\pi^{\mathfrak{Q}_8}([X\actsfromright \mathfrak{Q}_8])$ in $\cBurn_{2,k}$ differs from that of the class \eqref{eqn.C2class}.
\end{itemize}
For this, we introduce some divisors on $\cW$.
The locus where $u+v-1$, respectively $u+v-\lambda$ vanishes in the displayed affine chart determines
$\cD_1$, respectively $\cD_\lambda$.
The complement in $W$ of the displayed affine chart determines $\cD_\infty$.
We ignore multiplicity and take these divisors to have components of multiplicity one.
Each of these divisors is the closure of a pair of $K$-rational Weierstrass points of $W$.

We use the points of $\cD_\infty$ in the special fiber to endow each component of the special fiber with the structure of an elliptic curve over $k$.
Then, the points of order $2$ for the group law are the points of $\cD_1$ and $\cD_\lambda$ in the special fiber, as well as the point of intersection of the two components.

We let the pair of $2$-torsion line bundles
\[ \cO_{\cW}(\cD_1-\cD_\infty)\qquad\text{and}\qquad \cO_{\cW}(\cD_\lambda-\cD_\infty) \]
determine a Galois cover
\[ \cV\to \cW \]
for the Klein $4$-group $\mathfrak{K}_4$, with $K$-rational points in the generic fiber over the Weierstrass points of $W$.
(By the Kronecker-Weber theorem, $K$-rational points over one Weierstrass point lead to $K$-rational points over all of them.)
On $\cV$ we let $\mathfrak{Q}_8$ act via the quotient by its center, isomorphic to $\mathfrak{K}_4$.
We consider the generic fiber $V$ and claim that the resulting quotient stack is isomorphic to a trivial $\Z/2\Z$-gerbe over $W$:
\[ [V/\mathfrak{Q}_8]\cong W\times B\Z/2\Z. \]
Given such an isomorphism, the nontrivial one-dimensional representation of $\Z/2\Z$ determines a line bundle on $[V/\mathfrak{Q}_8]$, hence a $\mathfrak{Q}_8$-linearized line bundle $L$ on $V$.
Then the compactification
\[ X=\PP(\cO_V\oplus L) \]
satisfies $[X/\mathfrak{Q}_8]\cong W\times [\PP^1/(\Z/2\Z)]$.

The claim amounts to the ability to lift the structure of $V$ as $\mathfrak{K}_4$-torsor over $W$ to a structure of $\widetilde{V}$ as $\mathfrak{Q}_8$-torsor over $W$, for some
unramified degree $2$ cover $\widetilde{V}\to V$.
We write $V$ as a canonical curve in $\PP^4$:
\begin{equation}
\label{eqn.Vcanonical}
tu_0^2-tu_1^2+u_2^2=4t^2u_1^2-(2t+1)u_2^2-u_3^2=(\lambda-1)u_2^2-u_3^2+u_4^2=0,
\end{equation}
where $u=t(u_1-u_0)^{-1}(u_1+u_0)$.
Since the equations have a diagonal form, $V$ is a Humbert curve with $10$ vanishing theta nulls (effective even theta characteristics), one for each triple of coordinates; cf.\ \cite{varley}.
We consider the ones attached to the first and third equation in \eqref{eqn.Vcanonical}; their difference, $2$-torsion in the Jacobian, determines $\widetilde{V}$, with function field
\[ K(V)\bigg(\sqrt{\frac{u_1-u_0}{u_4-u_3}}\bigg). \]
The $\mathfrak{K}_4$-torsor structure, given by
\[
(u_0:u_1:u_2:u_3:u_4)\mapsto (u_0:u_1:(-1)^cu_2:(-1)^du_3:(-1)^{c+d}u_4),
\]
with $c$, $d\in \{0,1\}$,
lifts via ($i=\sqrt{-1}$)
\[
\sqrt{\frac{u_1-u_0}{u_4-u_3}}\mapsto
\pm \Big(\frac{u_4-u_3}{\sqrt{\lambda-1}\,u_2}\Big)^ci^d\sqrt{\frac{u_1-u_0}{u_4-u_3}}
\]
to a structure of $\widetilde{V}$ as $\mathfrak{Q}_8$-torsor, and the claim is established.

We determine $\rho_\pi^{\mathfrak{Q}_8}([X\actsfromright \mathfrak{Q}_8])$.
Let us consider the elliptic curve $(E,\infty)$ with function field \eqref{eqn.E} and $E[2]=\{\infty,0,1,\lambda\}$.
We exhibit an action of $\mathfrak{Q}_8$ on $E$, so that
\[ \sqrt{\cO_E(\infty)/E}\cong [E/\mathfrak{Q}_8]. \]
Pullback by the multiplication-by-$2$ map $[2]\colon E\to E$
sends $\cO_E(\infty)$ to
$\cO_E(\infty+0+1+\lambda)\cong \cO_E(2\infty)^{\otimes 2}$.
Consequently, the gerbe $\sqrt{\cO_E(\infty)/E}$ admits a section after base change by $[2]\colon E\to E$.
The composite
\[ E\to E\times_{[2],E}\sqrt{\cO_E(\infty)/E}\to \sqrt{\cO_E(\infty)/E} \]
sends an $E$-valued point $f\colon T\to E$, for any $k$-scheme $T$, to the $E$-valued point $[2]\circ f$ with line bundle $f^*\cO_E(2\infty)$ and isomorphism
\[ f^*\cO_E(2\infty)^{\otimes 2}\to f^*\cO_E(\infty+0+1+\lambda)\cong ([2]\circ f)^*\cO_E(\infty). \]
Denoting a square root of $u(u-1)(u-\lambda)$ in $k(E)$ by $z$, the isomorphism of line bundles is the pullback by $f$ of
\[ (\cdot z^{-1})\colon \cO_E(4\infty)\to \cO_E(\infty+0+1+\lambda). \]

Next, we determine the fiber product
\[ E\times_{\sqrt{\cO_E(\infty)/E}}E. \]
An object over $T$ consists of a pair of $E$-valued points $f$, $g\colon T\to E$, satisfying $[2]\circ f=[2]\circ g$, together with an isomorphism
\[ (\cdot s)\colon f^*\cO_E(2\infty)\to g^*\cO_E(2\infty), \]
such that the diagram
\[
\xymatrix{
f^*\cO_E(2\infty)^{\otimes 2}\ar[r]\ar[d]_{\cdot s^2} & ([2]\circ f)^*\cO_E(\infty) \ar@{=}[d] \\
g^*\cO_E(2\infty)^{\otimes 2}\ar[r] & ([2]\circ g)^*\cO_E(\infty)
}
\]
commutes.
The top map, respectively bottom map, is the pullback of $(\cdot z^{-1})$ by $f$, respectively $g$.
So $g$ is given by $f$ followed by addition of a point of $E[2]$.
The commutative diagram translates into the following possibilities for $s=f^*\sigma$, $\sigma\in \rH^0(E,\cO_E(2A-2\infty))$, $A\in E[2]$:
\[
\begin{array}{cc|c|c}
& \sigma & z^{-1}(A+(\,))^*z & A \\
\cline{2-4}
(\mu,\nu\in k, & \pm 1 & 1 & \infty \\
\mu^2=-\lambda, & \pm \mu u^{-1} & -\lambda u^{-2} & 0 \\
\nu^2=\lambda-1) & \pm \nu (u-1)^{-1} & -(1-\lambda) (u-1)^{-2} & 1 \\
& \pm \mu\nu (u-\lambda)^{-1} & \lambda(1-\lambda) (u-\lambda)^{-2} & \lambda
\end{array}
\]

Consequently,
\[ E\times_{\sqrt{\cO_E(\infty)/E}}E\cong E\times \{(\sigma,A)\} \]
with $(\sigma,A)$ as in the table, forming a group of order $8$ with composition law
\[ (\sigma,A)(\sigma',A')=(\sigma\cdot \tilde\sigma',A+A'), \]
where $\tilde\sigma'$ denotes the pullback of $\sigma'$ by the automorphism $A+(\,)$ of $E$.
This is the group $\mathfrak{Q}_8$, where the group element $(\sigma,A)$ acts on $E$ by translation by $A$.

Any line bundle on $[V/\mathfrak{Q}_8]$ may be extended to $[\cV/\mathfrak{Q}_8]$ (cf.\ the proof of Proposition \ref{prop.divisorial}).
So there is a
$\mathfrak{Q}_8$-linearized line bundle $\cL$ on $\cV$, restricting to $L$ on $V$.
Now $\cV$ is a semistable model of $V$, so $\PP(\cO_{\cV}\oplus \cL)$ is a semistable model of $X$.
Using this, we compute $\rho_\pi^{\mathfrak{Q}_8}([X\actsfromright \mathfrak{Q}_8])$ as
\begin{align}
\begin{split}
\label{eqn.Q8class}
2&\big(\mathrm{triv},\mathfrak{Q}_8\actsfromleft k(u,v)
\big(\sqrt{u(u-1)(u-\lambda)}\big),()\big) \\
&\qquad-(\mathrm{triv},\Z/2\Z\actsfromleft k(u,v),())+\dots.
\end{split}
\end{align}

Finally, we see that the images of \eqref{eqn.C2class} and \eqref{eqn.Q8class} in $\cBurn_2$ are different by an argument, as in Remark \ref{rem.ellipticcurve}, based on orbifold surfaces with non-isomorphic divisors $E\times B\Z/2\Z$ and $\sqrt{\cO_E(\infty)/E}$.
\end{proof}

\subsection*{Connection with Grothendieck groups}
The Grothendieck ring 
$$
\rK_0(\Var)
$$
of algebraic varieties over $k$ plays a special role in the theory of motivic integration: its quotient by the ideal generated the class of the affine line serves is the universal object in the category of motivic measures. Incidentally, this quotient is also isomorphic to the free abelian group generated by stable birational equivalence classes of algebraic varieties over $k$ \cite{LL}. 

Several authors considered versions of these constructions in the setting of equivariant geometry and in the framework of stacks, see \cite{Bergh-cat} and the references therein. There is a Grothendieck 
group of $G$-varieties
$$
\rK_0(\Var^G), 
$$
defined via {\em $G$-scissors relations} \cite[Defn.\ 3.1]{Bergh-cat}.
There is also a Grothendieck ring of equivariant varieties \cite[Defn.\ 3.7]{Bergh-cat}
$$
\rK_0(\Var^{\rm eq}):=\bigoplus_G \, \rK_0(\Var^G),
$$
with an additional {\em induction relation}.
The sum is over all isomorphism classes of finite groups.

Analogous definitions can be made for DM stacks.
The abelian group
\[
\rK_0(\mathrm{DM})
\]
is defined via scissors relations. This is (with the natural ring structure) the Grothendieck ring of DM stacks \cite[\S 1.1]{berghweak}.
There is a well-defined ring homomorphism
$$
\rK_0(\Var^{\rm eq}) \to  \rK_0(\mathrm{DM}), 
$$
cf.\ \cite[(3.15)]{Bergh-cat}, combined with \cite[Thm.\ 1.1]{berghweak}.
It takes the class of a smooth quasiprojective $G$-variety $X$ to the class of the quotient stack $[X/G]$.

\

From the initial definition of the Burnside group in \cite{KT}, it was understood that it was a refinement of $\rK_0(\Var)$, in the following sense: there is a surjective homomorphism 
$$
\Burn\to \mathrm{gr}(\rK_0(\Var)).
$$
Here we explain the connection of equivariant and stacky versions of Grothendieck groups with our formalism of Burnside groups. The key difference is that, in our approach, we take into account the actions in normal bundles of strata, in the stabilizer stratification.

There is a naive map
$$
\Burn(G) \to  \mathrm{gr}(\rK_0(\Var^G)),
$$
that annihilates all triples with nonempty sequence of characters.
The term $(1,G\actsfromleft k(X),())$ is sent to the class of the $G$-variety $X$.
There is an analogous map
$$
\cBurn \to \mathrm{gr}(\rK_0(\mathrm{DM})),
$$
annihilating pairs with nonempty sequence of characters and sending $(k(\cX),())$ to the class of the orbifold $\cX$.

\bibliographystyle{plain}
\bibliography{birstacks}

\begin{thebibliography}{10}

\bibitem{AF}
D.~Abramovich and B.~Fantechi.
\newblock Orbifold techniques in degeneration formulas.
\newblock {\em Ann. Sc. Norm. Super. Pisa Cl. Sci. (5)}, 16(2):519--579, 2016.

\bibitem{AGV}
D.~Abramovich, T.~Graber, and A.~Vistoli.
\newblock Gromov-{W}itten theory of {D}eligne-{M}umford stacks.
\newblock {\em Amer. J. Math.}, 130(5):1337--1398, 2008.

\bibitem{AOV}
D.~Abramovich, M.~Olsson, and A.~Vistoli.
\newblock Twisted stable maps to tame {A}rtin stacks.
\newblock {\em J. Algebraic Geom.}, 20(3):399--477, 2011.

\bibitem{AT}
D.~Abramovich and M.~Temkin.
\newblock Functorial factorization of birational maps for qe schemes in
  characteristic 0.
\newblock {\em Algebra Number Theory}, 13(2):379--424, 2019.

\bibitem{ATW}
D.~Abramovich, M.~Temkin, and J.~W{\l}odarczyk.
\newblock Functorial embedded resolution via weighted blowings up, 2020.
\newblock {\tt arXiv:1906.07106}.

\bibitem{toroidalorbifolds}
D.~Abramovich, M.~Temkin, and J.~W{\l}odarczyk.
\newblock Toroidal orbifolds, destackification, and {K}ummer blowings up.
\newblock {\em Algebra Number Theory}, 14(8):2001--2035, 2020.
\newblock With an appendix by David Rydh.

\bibitem{alperkresch}
J.~Alper and A.~Kresch.
\newblock Equivariant versal deformations of semistable curves.
\newblock {\em Michigan Math. J.}, 65(2):227--250, 2016.

\bibitem{behrendnoohi}
K.~Behrend and B.~Noohi.
\newblock Uniformization of {D}eligne-{M}umford curves.
\newblock {\em J. Reine Angew. Math.}, 599:111--153, 2006.

\bibitem{bergh}
D.~Bergh.
\newblock Functorial destackification of tame stacks with abelian stabilisers.
\newblock {\em Compos. Math.}, 153(6):1257--1315, 2017.

\bibitem{berghweak}
D.~Bergh.
\newblock Weak factorization and the {G}rothendieck group of
  {D}eligne-{M}umford stacks.
\newblock {\em Adv. Math.}, 340:193--210, 2018.

\bibitem{Bergh-cat}
D.~Bergh, S.~Gorchinskiy, M.~Larsen, and V.~Lunts.
\newblock Categorical measures for finite group actions.
\newblock {\em J. Algebraic Geom.}, 30(4):685--757, 2021.

\bibitem{berghrydh}
D.~Bergh and D.~Rydh.
\newblock Functorial destackification and weak factorization of orbifolds,
  2019.
\newblock {\tt arXiv:1905.00872}.

\bibitem{B-Brauer}
F.~A. Bogomolov.
\newblock The {B}rauer group of quotient spaces of linear representations.
\newblock {\em Izv. Akad. Nauk SSSR Ser. Mat.}, 51(3):485--516, 688, 1987.

\bibitem{BT-noether}
F.~A. Bogomolov and Yu. Tschinkel.
\newblock Noether's problem and descent.
\newblock In {\em Proceedings of the {S}eventh {I}nternational {C}ongress of
  {C}hinese {M}athematicians, {V}ol. {I}}, volume~43 of {\em Adv. Lect. Math.
  (ALM)}, pages 3--15. Int. Press, Somerville, MA, 2019.

\bibitem{BCS}
L.~A. Borisov, L.~Chen, and G.~G. Smith.
\newblock The orbifold {C}how ring of toric {D}eligne-{M}umford stacks.
\newblock {\em J. Amer. Math. Soc.}, 18(1):193--215, 2005.

\bibitem{cadman}
C.~Cadman.
\newblock Using stacks to impose tangency conditions on curves.
\newblock {\em Amer. J. Math.}, 129(2):405--427, 2007.

\bibitem{CT-S}
J.-L. Colliot-Th\'{e}l\`ene and J.-J. Sansuc.
\newblock The rationality problem for fields of invariants under linear
  algebraic groups (with special regards to the {B}rauer group).
\newblock In {\em Algebraic groups and homogeneous spaces}, volume~19 of {\em
  Tata Inst. Fund. Res. Stud. Math.}, pages 113--186. Tata Inst. Fund. Res.,
  Mumbai, 2007.

\bibitem{edidinmore}
D.~Edidin and Y.~More.
\newblock Partial desingularizations of good moduli spaces of {A}rtin toric
  stacks.
\newblock {\em Michigan Math. J.}, 61(3):451--474, 2012.

\bibitem{ER}
D.~Edidin and D.~Rydh.
\newblock Canonical reduction of stabilizers for {Artin} stacks with good
  moduli spaces.
\newblock {\em Duke Math. J.}, 170(5):827--880, 2021.

\bibitem{EGAIII}
A.~Grothendieck.
\newblock \'{E}l\'ements de g\'eom\'etrie alg\'ebrique. {III}. \'{E}tude
  cohomologique des faisceaux coh\'erents.
\newblock {\em Inst. Hautes \'Etudes Sci. Publ. Math.}, 11, 17, 1961--63.

\bibitem{EGAIV}
A.~Grothendieck.
\newblock \'{E}l\'ements de g\'eom\'etrie alg\'ebrique. {IV}. \'{E}tude locale
  des sch\'emas et des morphismes de sch\'emas.
\newblock {\em Inst. Hautes \'Etudes Sci. Publ. Math.}, 20, 24, 28, 32,
  1964--67.

\bibitem{harper}
A.~Harper.
\newblock Factorization for stacks and boundary complexes, 2017.
\newblock {\tt arXiv:1706.07999}.

\bibitem{HH}
B.~Hassett and D.~Hyeon.
\newblock Log canonical models for the moduli space of curves: the first
  divisorial contraction.
\newblock {\em Trans. Amer. Math. Soc.}, 361(8):4471--4489, 2009.

\bibitem{HKTconic}
B.~Hassett, A.~Kresch, and Yu. Tschinkel.
\newblock Stable rationality and conic bundles.
\newblock {\em Math. Ann.}, 365(3-4):1201--1217, 2016.

\bibitem{HKTsmall}
B.~Hassett, A.~Kresch, and Yu. Tschinkel.
\newblock Symbols and equivariant birational geometry in small dimensions.
\newblock In {\em Rationality of varieties}, volume 342 of {\em Progr. Math.},
  pages 201--236. Birkh\"{a}user/Springer, Cham, 2021.

\bibitem{keelmori}
S.~Keel and S.~Mori.
\newblock Quotients by groupoids.
\newblock {\em Ann. of Math. (2)}, 145(1):193--213, 1997.

\bibitem{KT}
M.~Kontsevich and Yu. Tschinkel.
\newblock Specialization of birational types.
\newblock {\em Invent. Math.}, 217(2):415--432, 2019.

\bibitem{K-Seattle}
A.~Kresch.
\newblock On the geometry of {D}eligne-{M}umford stacks.
\newblock In {\em Algebraic geometry---{S}eattle 2005. {P}art 1}, volume~80 of
  {\em Proc. Sympos. Pure Math.}, pages 259--271. Amer. Math. Soc., Providence,
  RI, 2009.

\bibitem{Bbar}
A.~Kresch and Yu. Tschinkel.
\newblock Birational types of algebraic orbifolds.
\newblock {\em Mat. Sb.}, 212(3):54--67, 2021.

\bibitem{KT-map}
A.~Kresch and Yu. Tschinkel.
\newblock Burnside groups and orbifold invariants of birational maps, 2022.
\newblock {\tt arXiv:2208.05835}.

\bibitem{BnG}
A.~Kresch and Yu. Tschinkel.
\newblock Equivariant birational types and {B}urnside volume.
\newblock {\em Ann. Sc. Norm. Super. Pisa Cl. Sci. (5)}, 23(2):1013--1052,
  2022.

\bibitem{KT-vector}
A.~Kresch and Yu. Tschinkel.
\newblock Equivariant {B}urnside groups and representation theory.
\newblock {\em Selecta Math. (N.S.)}, 28(4):Paper No. 81, 39, 2022.

\bibitem{KV}
A.~Kresch and A.~Vistoli.
\newblock On coverings of {D}eligne-{M}umford stacks and surjectivity of the
  {B}rauer map.
\newblock {\em Bull. London Math. Soc.}, 36(2):188--192, 2004.

\bibitem{LL}
M.~Larsen and V.~A. Lunts.
\newblock Motivic measures and stable birational geometry.
\newblock {\em Mosc. Math. J.}, 3(1):85--95, 2003.

\bibitem{LMB}
G.~Laumon and L.~Moret-Bailly.
\newblock {\em Champs alg\'{e}briques}, volume~39 of {\em Ergebnisse der
  Mathematik und ihrer Grenzgebiete. 3. Folge}.
\newblock Springer-Verlag, Berlin, 2000.

\bibitem{lev}
D.~Levchenko.
\newblock Models of quadric surface bundles.
\newblock {\em Eur. J. Math.}, 8:S518--S532, 2022.

\bibitem{manin2}
Yu.~I. Manin.
\newblock Rational surfaces over perfect fields. {II}.
\newblock {\em Mat. Sb. (N.S.)}, 72 (114):161--192, 1967.

\bibitem{milne}
J.~S. Milne.
\newblock {\em \'{E}tale cohomology}.
\newblock Princeton Mathematical Series, No. 33. Princeton University Press,
  Princeton, N.J., 1980.

\bibitem{mumford-git}
D.~Mumford, J.~Fogarty, and F.~Kirwan.
\newblock {\em Geometric invariant theory.}, volume~34 of {\em Ergeb. Math.
  Grenzgeb.}
\newblock Berlin: Springer-Verlag, 3rd enl. ed. edition, 1994.

\bibitem{olssonboundedness}
M.~Olsson.
\newblock A boundedness theorem for {H}om-stacks.
\newblock {\em Math. Res. Lett.}, 14(6):1009--1021, 2007.

\bibitem{Pro-ICM}
Yu. Prokhorov.
\newblock Finite groups of birational transformations.
\newblock In {\em European congress of mathematics. Proceedings of the 8th
  congress, 8ECM, Portoro\v{z}, Slovenia, June 20--26, 2021}, pages 413--437.
  Berlin: European Mathematical Society (EMS), 2023.

\bibitem{RG}
M.~Raynaud and L.~Gruson.
\newblock Crit\`eres de platitude et de projectivit\'{e}. {T}echniques de
  ``platification'' d'un module.
\newblock {\em Invent. Math.}, 13:1--89, 1971.

\bibitem{RYinvariant}
Z.~Reichstein and B.~Youssin.
\newblock A birational invariant for algebraic group actions.
\newblock {\em Pacific J. Math.}, 204(1):223--246, 2002.

\bibitem{rydhdevissage}
D.~Rydh.
\newblock \'{E}tale d\'{e}vissage, descent and pushouts of stacks.
\newblock {\em J. Algebra}, 331:194--223, 2011.

\bibitem{Salt}
D.~J. Saltman.
\newblock Noether's problem over an algebraically closed field.
\newblock {\em Invent. Math.}, 77(1):71--84, 1984.

\bibitem{schmitt}
J.~Schmitt.
\newblock Birational classification of toric orbifolds, 2023.
\newblock {\tt arXiv:2308.10528}.

\bibitem{temkin-nonembedded}
M.~Temkin.
\newblock Functorial desingularization of quasi-excellent schemes in
  characteristic zero: the nonembedded case.
\newblock {\em Duke Math. J.}, 161(11):2207--2254, 2012.

\bibitem{temkin-embedded}
M.~Temkin.
\newblock Functorial desingularization over {{\(\mathbb Q\)}}: boundaries and
  the embedded case.
\newblock {\em Isr. J. Math.}, 224:455--504, 2018.

\bibitem{thomason}
R.~W. Thomason.
\newblock Equivariant resolution, linearization, and {H}ilbert's fourteenth
  problem over arbitrary base schemes.
\newblock {\em Adv. in Math.}, 65(1):16--34, 1987.

\bibitem{varley}
R.~Varley.
\newblock Weddle's surfaces, {H}umbert's curves, and a certain
  {$4$}-dimensional abelian variety.
\newblock {\em Amer. J. Math.}, 108(4):931--951, 1986.

\bibitem{vistoli}
A.~Vistoli.
\newblock Intersection theory on algebraic stacks and on their moduli spaces.
\newblock {\em Invent. Math.}, 97(3):613--670, 1989.

\bibitem{yasuda}
T.~Yasuda.
\newblock Motivic integration over {D}eligne-{M}umford stacks.
\newblock {\em Adv. Math.}, 207(2):707--761, 2006.

\end{thebibliography}

\end{document}